\documentclass[12pt,a4paper,reqno]{amsart}
\usepackage[english]{babel}
\usepackage{array}
\usepackage{blkarray}
\usepackage{makecell}
\usepackage{extarrows}
\usepackage{amssymb,xspace} 
\usepackage{amstext}
\usepackage[hidelinks]{hyperref} 
\usepackage[multiple]{footmisc}
\usepackage{epigraph}
\usepackage[mathscr]{eucal} 
\theoremstyle{plain}
\usepackage{amsbsy,amssymb,amsfonts,latexsym,eucal,amscd}
\usepackage{mathtools}
\usepackage{tikz}
\usepackage{stmaryrd}
\usepackage[all]{xy}
\UseComputerModernTips
\usepackage{fourier-orns}
\usepackage{geometry}
\newtheorem*{theoremenonum}{Theorem}

\newcommand{\tens}[1][]{\mathbin{\otimes_{\raise1.5ex\hbox to-.1em{}{#1}}}}
\newcommand{\lltens}[1][]{{\mathop{\tens}\limits^{\rm \mathbb{L}}}_{#1}}
\newcommand{\lltensl}[1][]{{\mathop{\tens}\limits^{\rm \mathbb{L}, \, \ell}}_{#1}}
\newcommand{\lltensr}[1][]{{\mathop{\tens}\limits^{\rm \mathbb{L},\, r}}_{#1}}

\newtheorem{theorem}{Theorem}[section]

\newtheorem{theoremeA}{Theorem}
\renewcommand{\thetheoremeA}{\Alph{theoremeA}}

\newtheorem{lemma}[theorem]{Lemma}

\newtheorem{proposition}[theorem]{Proposition}

\newtheorem{corollary}[theorem]{Corollary}

{\theoremstyle{definition}}

{\theoremstyle{definition}}

{\theoremstyle{definition}\newtheorem{notations}[theorem]{Notations}}

{\theoremstyle{definition}}

{\theoremstyle{definition}\newtheorem{definition}[theorem]{Definition}}

{\theoremstyle{definition}}

{\theoremstyle{definition}\newtheorem{remark}[theorem]{Remark}}

\renewcommand{\proofname}{Proof}

\newcommand{\UN}[4][r]{%
    \ar@/^1pc/[#1]^{#2}_*=<0.3pt>{}="HAUT"
    \ar@/_1pc/[#1]_{#3}^*=<0.3pt>{}="BAS"
    \ar @{=>} "HAUT";"BAS" ^{#4}
  }

\author{Julien Grivaux}

\address{Sorbonne Université, Université Paris Diderot, CNRS, Institut de Mathématiques de Jussieu-Paris Rive Gauche, IMJ-PRG, F-75005, Paris, France 
}

\email{jgrivaux@math.cnrs.fr}

\title[]{Derived geometry of the first formal neighbourhood of a smooth analytic cycle}

\renewcommand{\epigraphsize}{\tiny}
\begin{document}

\maketitle

\epigraph{La Nature est un temple o\`{u} de vivants piliers \\
Laissent parfois sortir de confuses paroles; \\
L'homme y passe \`{a} travers des for\^{e}ts de symboles \\
Qui l'observent avec des regards familiers.}{Charles Baudelaire -- \textit{Les fleurs du mal}}

\begin{abstract}
If $X$ is a smooth scheme of characteristic zero or a complex analytic manifold, and $S$ is a locally split infinitesimal thickening of $X$, we compute explicitly the derived self-intersection of $X$ in $S$.
\end{abstract}
\setcounter{tocdepth}{6}

\tableofcontents
\section{Introduction}
The Hochschild-Kostant-Rosenberg isomorphism, introduced in \cite{HKRoriginal} for regular algebras and extended later on in a series of papers (e.g. \cite{Swan}, \cite{BuchF}, \cite{Yekutieli}, \cite{lettre}\footnote{For a reproduction of this letter addressed to P. Schapira, \textit{see} the book \cite[Chap. 5]{KS1}.}) to different geometric settings, can be stated as follows:
\begin{theoremenonum} 
If $X$ is either a complex manifold or a smooth scheme over a field of characteristic zero, and if $\delta$ is the diagonal injection, then there is a canonical formality isomorphism
\[
\mathbb{L} \delta^* (\delta_* \mathcal{O}_X) \simeq \bigoplus_{p=0}^{\mathrm{dim}\,X} \Omega^p_X[p]
\]
in the bounded derived category of coherent sheaves on $X$.
\end{theoremenonum}
This result turns out to be extremely useful in algebraic and complex geometry as well as in deformation quantization, we refer the interested reader to the non exhaustive list of papers \cite{CalMukai}, \cite{Markarian}, \cite{Ramadoss}, \cite{conjecture}, \cite{CV}, \cite{CVDBR}, \cite{DTT}, \cite{Kontsevich}, \cite{KS1}, \cite{BNT} as well as references therein.
\par \medskip
After the pioneering unpublished contribution of Kashiwara \cite{lettre}, there has been a lot of efforts in recent years to understand more general forms of this isomorphism, corresponding to arbitrary closed immersions instead of the diagonal embedding. It started with the work of Arinkin and C\u{a}ld\u{a}raru \cite{Arinkin-Caldararu}, and was carried on by lot of others including Calaque, Tu, Habliczek, Yu and the author (\textit{see} \cite{CalaqueCT}, \cite{ACH}, \cite{Shilin}, \cite{Grivaux-HKR}, \cite{Grivaux-formality}). 
\par \medskip
In the present paper, we won't deal with arbitrary closed immersions into an ambient smooth scheme, but rather in the corresponding first order thickening. Let $\mathbf{k}$  be a fixed base field of characteristic zero. We state the results in the algebraic setting, but all of them remain true in the analytic setting as well. One of the principal existing result in this theory is due to Arinkin and C\u{a}ld\u{a}raru:
\begin{theoremeA} \cite{Arinkin-Caldararu}
If $X$ is a smooth $\mathbf{k}$-scheme and $j \colon X \hookrightarrow S$ is a first-order thickening of $X$ by a locally free sheaf $\mathcal{I}$, then for any locally free sheaf $\mathcal{V}$ on $X$, the derived pullback $\mathbb{L}j^* (j_* \mathcal{V})$ is formal if and only if $\mathcal{I}$ and $\mathcal{V}$ extend to locally free sheaves on $S$.
\end{theoremeA}
The main ingredient in the proof is the identification of three cohomology classes attached to a locally free sheaf $\mathcal{V}$ that live in the cohomology group $\mathrm{H}^2(X, \mathcal{H}om(\mathcal{V}, \mathcal{I} \otimes \mathcal{V}))$, whose construction we recall now:
\begin{enumerate}
\vspace{0.2cm}
\item[(a)] The sheaf of sets on $X$ associating to any open subscheme $U$ of $X$ the set of locally free $\mathcal{O}_S$-extensions of $\mathcal{V}$ on $U$ is an abelian gerbe whose automorphism sheaf is $\mathcal{H}om(\mathcal{V}, \mathcal{I} \otimes \mathcal{V})$, so it defines a class in $\mathrm{H}^2(X, \mathcal{H}om(\mathcal{V}, \mathcal{I} \otimes \mathcal{V}))$.
\vspace{0.2cm}
\item[(b)] There is a distinguished truncation triangle
\[
\mathcal{I} \otimes \mathcal{V}\, [1] \longrightarrow \tau^{\geq -1}\, \mathbb{L}j^* (j_* \mathcal{V}) \longrightarrow \mathcal{V} \xlongrightarrow{+1}
\]
yielding a morphism from $\mathcal{V}$ to $\mathcal{I} \otimes \mathcal{V} \, [2]$ in $\mathrm{D}^{\mathrm{b}}(X)$, which is the same as a class in $\mathrm{H}^2(X, \mathcal{H}om(\mathcal{V}, \mathcal{I} \otimes \mathcal{V}))$.
\vspace{0.2cm}
\item[(c)] If $\eta \in \mathrm{Ext}^1_{\mathcal{O}_X}(\Omega^1_X, \mathcal{I})$ is the extension class of the conormal exact sequence of the embedding $j$ \footnote{The class $\eta$ is called the Kodaira-Spencer class in \cite{Huybrechts-Thomas}, it is zero exactly if and only if the thickening $S$ is trivial.}, then the Yoneda product of the Atiyah class of $\mathcal{V}$ in $\mathrm{Ext}^1_{\mathcal{O}_X}(\mathcal{V}, \Omega^1_X \otimes \mathcal{V})$ with $\eta \otimes \mathrm{id}_{\mathcal{V}}$ yields a class in $\mathrm{Ext}^2_{\mathcal{O}_X}(\mathcal{V}, \mathcal{I} \otimes \mathcal{V})$ which is again $\mathrm{H}^2(X, \mathcal{H}om(\mathcal{V}, \mathcal{I} \otimes \mathcal{V}))$.
\vspace{0.2cm}
\end{enumerate}
The main breakthrough in Arinkin-C\u{a}ld\u{a}raru's approach is the identification between the classes defined in (a) and (b). The relation between (a) and (c) had already been settled earlier on for arbitrary complexes of sheaves by Huybrechts and Thomas \cite{Huybrechts-Thomas}, refining previous works of Lieblich \cite{Lieblich} and Lowen \cite{Lowen}. In the absolute smooth case, their result can be stated as follows:
\begin{theoremeA} \cite{Huybrechts-Thomas}
Let $X$ be a smooth $\mathbf{k}$-scheme and let $j \colon X \hookrightarrow S$ be a first-order thickening of $X$ by a locally free sheaf. Then the essential image of
\[
\mathbb{L}j^* \colon \mathrm{D}^{\mathrm{perf}}(S) \longrightarrow \mathrm{D}^{\mathrm{perf}}(X)
\]
consists of elements $\mathcal{V}_{\bullet}$ in $\mathrm{D}^{\mathrm{perf}}(X)$ such that the composition
\[
\mathcal{V}_{\bullet} \xlongrightarrow{\mathrm{at}_X (\mathcal{V}_{\bullet})}  \Omega^1_X \otimes \mathcal{V}_{\bullet}[1] \xlongrightarrow{\eta \,\otimes\, \mathrm{id}} \mathcal{I} \otimes \mathcal{V}_{\bullet} [2]
\]
vanishes.
\end{theoremeA} 
The present object of this paper is twofold: first we present generalizations of Theorems A and B
for arbitrary sheaves on $S$, which are neither locally free nor push-forwards of sheaves on $X$. However, we want to emphasize that we don't generalize Theorem B in full generality, because we are only dealing with the case of a smooth ambient scheme in order to avoid considerations about the full cotangent complex. 
\par \medskip
A crucial tool introduced in the paper is a generalization to complexes of sheaves on $S$ of the Yoneda product of the Atiyah and Kodaira-Spencer classes: for any complex of sheaves $\mathcal{K}_{\bullet}$ in $\mathrm{C}^{-}(S)$ we define a morphism\footnote{The functors $\mathrm{Tor}^i_{\mathcal{O}_S}(*, \mathcal{O}_X)$ are not the usual hypertor functors, but simply the canonical extension to complexes of the functors $\mathrm{Tor}^i_{\mathcal{O}_S}(\,*\,, \mathcal{O}_X) \colon \mathrm{Sh}(S) \longrightarrow \mathrm{Sh}(X)$.}
\[
\Theta_{\mathcal{K}_{\bullet}} \colon j^* \mathcal{K}_{\bullet}=\mathrm{Tor}^0_{\mathcal{O}_S}(\mathcal{K}_{\bullet}, \mathcal{O}_X) \longrightarrow \mathrm{Tor}^1_{\mathcal{O}_S}(\mathcal{K}_{\bullet}, \mathcal{O}_X)[2]
\]
in $\mathrm{D}^{-}(X)$, which is the connection morphism attached to a canonical distinguished triangle
\[
\mathrm{Tor}^1_{\mathcal{O}_X}(\mathcal{K}_{\bullet}, \mathcal{O}_X) [1] \longrightarrow j^* \,\mathrm{cone} \{ \Omega^1_S \otimes \mathcal{K}_{\bullet} \longrightarrow \mathrm{P}^1_S(\mathcal{K}_{\bullet})\} \longrightarrow j^* \mathcal{K}_{\bullet} \xlongrightarrow{+1}
\] 
where $\mathrm{P}^1_S$ is the principal parts functor.
\par \medskip 
In our setting, we replace strict perfect complexes on $S$ by a larger class of complexes, called bounded admissible complexes: these are the bounded complexes $\mathcal{K}_{\bullet}$ such that the complex $\mathrm{Tor}^1_{\mathcal{O}_S}(\mathcal{K}_{\bullet}, \mathcal{O}_X)$ is quasi-isomorphic to zero. Up to quasi-isomorphism, bounded admissible complexes and perfect complexes have a very simple common description: a complex $\mathcal{K}_{\bullet}$ in $\mathrm{D}^-(X)$ is quasi-isomorphic to a bounded admissible complex (resp. is a perfect complex) if and only if $\mathbb{L}j^* \mathcal{K}_{\bullet}$ is cohomologically bounded (resp. is perfect). However admissible sheaves, even coherent ones, form a much larger class than locally free ones.
\begin{theorem} \label{1}
Let $X$ be a smooth $\mathbf{k}$-scheme and $j \colon X \hookrightarrow S$ be a first-order thickening of $X$ by a locally free sheaf. For any bounded complex $\mathcal{K}_{\bullet}$ of $\mathcal{O}_S$-modules, the following properties are equivalent:
\begin{enumerate}
\item[--] The morphism ${\Theta}_{\mathcal{K}_{\bullet}}$ vanishes.
\vspace{0.2cm}
\item[--] The morphism $\mathbb{L}j^* \mathcal{K}_{\bullet} \longrightarrow j^* \mathcal{K}_{\bullet}$ admits a right inverse in $\mathrm{D}^{\mathrm{-}}(X)$.
\vspace{0.2cm}
\item[--] There exists a bounded admissible complex $\mathcal{L}_{\bullet}$ and a morphism in $\mathrm{D}^{\mathrm{b}}(S)$  from  $\mathcal{L}_{\bullet}$ to $\mathcal{K}_{\bullet}$ such that the composition
\[
\mathbb{L}j^* \mathcal{L}_{\bullet} \longrightarrow \mathbb{L} j^* \mathcal{K}_{\bullet} \longrightarrow j^* \mathcal{K}_{\bullet}
\]
is an isomorphism in $\mathrm{D}^{-}(X)$.
\vspace{0.2cm}
\end{enumerate}
\end{theorem}
Even in the case where $\mathcal{K}_{\bullet}$ is the push-forward of a perfect complex on $X$, this gives a new and lighter proof of Theorem B. We also want to emphasize that the equivalent conditions in Theorem \ref{1} do not depend only on the isomorphism class of $\mathcal{K}_{\bullet}$ in $\mathrm{D}^{-}(S)$, unlike the situation described in Theorem B. However, we can make the link with the two settings as follows: we construct a suitable localization $\mathrm{D}^{\mathrm{adm}}(S)$ of $\mathrm{C}^{-}(S)$, which is finer than the usual localization that gives rise to the derived category $\mathrm{D}^{-}(X)$, such that:
\par \smallskip
-- The Tor functors $\mathrm{Tor}^i_{\mathcal{O}_S}(\,* \,, \mathcal{O}_X) \colon \mathrm{C}^-(S) \longrightarrow \mathrm{C}^-(X)$ factor through triangulated functors from $\mathrm{D}^{\mathrm{adm}}(S)$ to $\mathrm{D}^-(X)$.
\par \smallskip
-- The standard push forward functor $j_* \colon \mathrm{D}^-(X) \longrightarrow \mathrm{D}^-(S)$ lifts to the admissible derived category $\mathrm{D}^{\mathrm{adm}} (S)$.
\par \smallskip
-- The morphism $\Theta$ can be interpreted as a natural transformation in the following diagram
\[
\xymatrix@R=30pt@C=50pt{\relax
  \mathrm{D}^-(X)  \ar@/^5pc/[rr]^-{\mathrm{id}_{\mathrm{D}^-(X)}}  \ar@/_5pc/[rr]_-{\mathcal{I} [2] \, \lltens \, *}  \ar[r]^-{j_*} & \mathrm{D}^{\mathrm{adm}}(S) \UN[r]{\mathrm{Tor}^0_{\mathcal{O}_S}(*, \,\mathcal{O}_X)\,\,\,}{\mathrm{Tor}^1_{\mathcal{O}_S}(*, \,\mathcal{O}_X)[2]\, \, \,}{\Theta} & \mathrm{D}^-(X)}
\]
Then the equivalent conditions in Theorem \ref{1} depend only on the isomorphism class of $\mathcal{K}_{\bullet}$ in $\mathrm{D}^{\mathrm{adm}}(S)$.
\par \medskip
A geometric example\footnote{By ``geometric'' we mean that $S$ is the first formal neighbourhood of $X$ in some ambient smooth scheme.} for which the morphism  $\Theta_{\mathcal{V}}$ is nonzero for some line bundle $\mathcal{V}$ on $X$ has been constructed by Arinkin and C\u{a}ld\u{a}raru in \cite[\S 4]{Arinkin-Caldararu}. It is possible to produce examples that are in some sense much worse, since the morphism $\Theta_{\mathcal{V}}$ doesn't vanish even locally.
\par \medskip
We now get back to Theorem A. We give a necessary and sufficient condition for the formality of a derived pullback, as well as an intrinsic interpretation of $\Theta_{\mathcal{V}}$: 
\begin{theorem}  \label{2}
Let $X$ be a smooth $\mathbf{k}$-scheme and let $j \colon X \hookrightarrow S$ be a first-order thickening of $X$ by a locally free sheaf. If $\mathcal{K}$ is a sheaf of $\mathcal{O}_S$-modules, then $\Theta_{\mathcal{K}}$ is the connection morphism attached to the distinguished truncation triangle
\[
\mathrm{Tor}^1_{\mathcal{O}_S}(\mathcal{K}, \mathcal{O}_X)[1] \longrightarrow \tau^{\geq -1} \mathbb{L}j^* \mathcal{K} \longrightarrow j^* \mathcal{K} \xlongrightarrow{+1}
\]
The object $\mathbb{L}j^* \mathcal{K}$ is formal in $\mathrm{D}^-(X)$ if and only $\Theta_{\mathcal{K}}$ and $\{ \Theta_{\mathrm{Tor}^p_{\mathcal{O}_S}(\mathcal{K}, \mathcal{O}_X)} \}_{p \geq 0}$ vanish.
If $\mathcal{K}$ is the push-forward of a coherent sheaf on $X$ which is not a torsion sheaf, these conditions are equivalent to the vanishing of $\Theta_{\mathcal{K}}$ and $\Theta_{\mathcal{I}}$.
\end{theorem}
The morphism $\Theta$ is the key to understand more completely the endofunctor $\mathbb{L}j^* j_*$ of $\mathrm{D}^{-}(X)$, which is the second and principal purpose of the paper. This functor is a locally (but in general not globally) trivial twist of the formal functor $\mathcal{V} \longrightarrow \bigoplus_{p \geq 0} \mathcal{I}^{\otimes p} \otimes \mathcal{V} [p]$. We construct bounded approximations of $\mathbb{L}j^* j_*$ as follows: let $H$ be the exact endofunctor of $\mathrm{C}^{-}(X)$ defined by
\[
H(\mathcal{V}_{\bullet})=\mathrm{cone}\, \{\Omega^1_X \otimes \mathcal{V}_{\bullet} \longrightarrow \mathrm{P}^1_X(\mathcal{V}_{\bullet}) \}.
\] In other words, the functor $H$ is the Fourier-Mukai transformation associated with the kernel $\mathcal{E}_{\Delta} \longrightarrow \mathcal{O}_{\overline{\Delta}}$, where $\mathcal{E}_{\Delta}$ is the pushforward of $\mathcal{E}$ by the diagonal embedding and $\mathcal{O}_{\overline{\Delta}}$ is the structural sheaf of the first formal neighborhood of the diagonal in $X \times X$. Then $H$ is naturally endowed with a morphism to the identity functor. For any positive integer $n$, we denote by $H^{[n]}$ the equalizer of the $n$ natural maps from $H^n$ to $H^{n-1}$ induced by this morphism. Then we prove the following structure theorem:
\begin{theorem} \label{3}
Let $X$ be a smooth $\mathbf{k}$-scheme and let $j \colon X \hookrightarrow S$ be a first-order thickening of $X$ by a locally free sheaf. Then the sequence $(H^{[n]})_{n \geq 0}$ induces a projective system of lax multiplicative endofunctors of $\mathrm{D}^{-}(X)$, and there is a canonical multiplicative isomorphism
\[
\mathbb{L}j^*j_* \simeq \underset{n}{\varprojlim} \, H^{[n]}
\]
Besides, if $S$ is globally trivial, then there is a natural isomorphism of functors
\[
\underset{n}{\varprojlim} \, H^{[n]} (\star) \simeq \left( \bigoplus_{n \geq 0} \mathcal{E}^{\otimes n} [n] \right) \otimes (\star)
\]
and for any sheaf $\mathcal{V}$ of $\mathcal{O}_X$-modules, the composite isomorphism
\[
\mathbb{L}j^* (j_* \mathcal{V}) \simeq \underset{n}{\varprojlim} \, H^{[n]}(\mathcal{V}) \simeq \bigoplus_{n \geq 0} \mathcal{E}^{\otimes n} [n] \otimes \mathcal{V}
\] 
is the generalized HKR isomorphism constructed by Arinkin and C\u{a}ld\u{a}raru in \cite{Arinkin-Caldararu} .
\end{theorem}
Let us give some motivation to compute the functor $\mathbb{L}j^*j_*$. The first motivation comes from the work of Kapranov \cite{Kapranov} and Markarian \cite{Markarian}: they construct a structure of derived Lie algebra on the shifted tangent bundle $TX[-1]$ of any complex manifold $X$, the derived Lie bracket being given by the Atiyah class\footnote{This is the geometric counterpart of Quillen's theorem \cite{Quillen}.}. In the case of the diagonal embedding, this derived Lie structure has been studied in the framework of Lie groupoids to prove the geometric Duflo isomorphism conjectured by Kontsevich (\textit{see} \cite{CVDBR}, \cite{CalaqueEMS}), and its extension to arbitrary closed embeddings is widely open and of high interest. We believe that the explicit description of $\mathbb{L}j^*j_*$ can lead to substantial progress on this question. 
\par \medskip
The second motivation originates from Kontsevich's homological mirror symmetry conjecture \cite{miroir}: if $X$ is a closed submanifold of a complex manifold $Y$, then the global Ext groups $\mathrm{Ext}^i_{\mathcal{O}_Y}(\mathcal{O}_X, \mathcal{O}_X)$ are the counterpart in the B-model of the Fl\oe{}r homology groups, and are strongly related to the generalized HKR isomorphism for this closed immersion. 
\par \medskip
The last and perhaps more important motivation, that overlaps with the two previous ones, is that the object $\mathbb{L}j^*(j_*\mathcal{O}_X)$ is the structural sheaf of the derived fiber product $X \times^{\mathrm{h}}_S X$, this operation being performed in the category of derived algebraic schemes\footnote{For an overview of derived algebraic geometry, \textit{see} \cite{Toen}}. It is of real interest to understand what geometric information can be extracted from this derived scheme.
\par \medskip
Let us now present the organization of the paper. 
\par \medskip
-- \S \ref{debut} recalls well-known constructions on the category of complexes of an additive category, and its use is mainly to fix the notation and conventions. 
\par \medskip
-- The entire \S \ref{pita} sets the categorical framework in order to find a reasonable candidate for the functor $\mathbb{L}j^*j_*$. In \S \ref{natanz}, we explain how the formal objects $\bigoplus_{p=0}^{n} G^{p}[-p]$ attached to a dg-endofunctor $G$ of the category $\mathrm{C}^{{\mathrm{b}}}(\mathcal{C})$ of bounded complexes of an additive category $\mathcal{C}$ can be twisted by a closed dg morphism $\Theta \colon \mathrm{id}_{\mathrm{C}^{{\mathrm{b}}}(\mathcal{C})} \longrightarrow G$, thus defining dg-endofunctors $(F_n)_{n \geq 0}$ of $\mathrm{C}^{\mathrm{b}}(\mathcal{C})$. This is the content of Theorem \ref{hard}.
In \S \ref{twist}, we prove that the functors $F_n$ constructed in the previous part are naturally isomorphic to the equalizers of the $n$ natural maps from $\Delta^n_{\Theta}$ to $\Delta^{n-1}_{\Theta}$ induced by the morphism $\Delta_{\Theta} \longrightarrow \mathrm{id}$, where $\Delta_{\Theta}$ is the cone of $\Theta$ shifted by minus one (Theorem \ref{orange}).
\par \medskip
-- \S \ref{carrenul} deals with algebraic properties of modules over trivial square zero extensions of commutative $\mathbf{k}$-algebras. In \S \ref{nakon}, we prove a few crucial properties for such modules: if $B$ is a trivial square zero extension of a commutative $\mathbf{k}$-algebra $A$ and $V$ is a $B$-module, then the $A$-module $\mathrm{Tor}^1_B(V, A)$ admits a very simple description (Corollary \ref{dur}), and the higher Tor modules $\mathrm{Tor}^p_B(V, A), p \geq 2$ can also be explicitly computed (Proposition \ref{Tor2}). The most important result we prove is the vanishing theorem for principal parts (Theorem \ref{wazomba}). In \S \ref{lauvitel}, we introduce special classes of complexes of $B$-modules: admissible and $n$-admissible complexes. These complexes are a substitute for bounded flat resolutions or strict complexes (Proposition \ref{cns}) and for perfect complexes if $n=+ \infty$ (Corollary \ref{bute} and Proposition \ref{parfait}). However, $n$-admissible resolutions are much more easy to construct canonically than flat resolutions (Corollary \ref{stair} and Theorem \ref{loire}). In \S \ref{ex}, we define the admissible triangulated category $\mathrm{D}^{\mathrm{adm}}(S)$, which is a substitute for the derived category of perfect $S$-modules. Then we prove a structure theorem (Proposition \ref{marre}) allowing to reconstruct any complex of $B$-modules up to an isomorphism in $\mathrm{D}^{\mathrm{adm}}(S)$ from elementary bricks that are objects and morphisms in $\mathrm{D}^-(A)$. In \S \ref{tnt}, we define the HKR morphism attached to a complexes of $B$-modules, and give equivalent algebraic conditions equivalent to its vanishing (Theorem \ref{localobs}), which is the local version of Theorem \ref{1}. \S \ref{esthete} is devoted to prove crucial splitting-free results: Proposition \ref{campingaz} is the key tool to define the HKR in a geometric non-split setting, and Theorem \ref{shihiro} is an essential ingredient for proving Theorem \ref{tabriz}.
\par \medskip
-- \S \ref{ispahan} generalizes the construction of \S \ref{carrenul} to the geometric setting. The main result is Theorem \ref{tabriz}, which is a refined version of Theorem \ref{1}. Then we deduce Theorem \ref{2}, which is obtained by combining Theorem \ref{belote} and Corollary \ref{hmpf}.
\par \medskip
-- The last section \S \ref{colline} is entirely devoted to the proof of Theorem \ref{winner}, which is a refined version of Theorem \ref{3}.
\par \medskip
\textbf{Acknowledgments} I would like to thank Richard Thomas for many useful comments, and Bertrand Toën for is invaluable help. My warmest thanks go to the referee, whose work and dedication led to a considerable improvement of the paper.
\section{The dg-category of complexes} \label{debut}

\subsection{Generalities on mapping cones}
Let $\mathcal{C}$ be an additive category. We introduce the following standard notation:
\begin{enumerate}
\vspace{0.2cm}
\item[--] The categories of complexes of elements of $\mathcal{C}$ which are arbitrary, bounded, bounded from above and bounded from below are denoted by $\mathrm{C}(\mathcal{C})$, $\mathrm{C}^{\mathrm{b}}(\mathcal{C})$, $\mathrm{C}^{-}(\mathcal{C})$, and $\mathrm{C}^{+}(\mathcal{C})$ respectively. If we want to specify complexes concentrated in degrees that are between two integers $a$ and $b$, we write $\mathrm{C}^{[a, b]}(\mathcal{C})$.
\vspace{0.2cm}
\item[--] The corresponding homotopy categories are denoted by $\mathrm{K}(\mathcal{C})$, $\mathrm{K}^{\mathrm{b}}(\mathcal{C})$, $\mathrm{K}^{-}(\mathcal{C})$, $\mathrm{K}^{+}(\mathcal{C})$ and $\mathrm{C}^{[a, b]}(\mathcal{C})$.
\vspace{0.2cm}
\item[--] If $\mathcal{C}$ is abelian, the corresponding derived categories are denoted by $\mathrm{D}(\mathcal{C})$, $\mathrm{D}^{\mathrm{b}}(\mathcal{C})$, $\mathrm{D}^{-}(\mathcal{C})$, $\mathrm{D}^{+}(\mathcal{C})$ and $\mathrm{D}^{[a, b]}(\mathcal{C})$.
\item[--] The category $\mathrm{C}(\mathcal{C})$ is a $k$-linear dg-category: for any complexes $K$ and $L$ and for any integer $n$ we have 
\[
\underline{\mathrm{Hom}}^n({{K}}, L)=\bigoplus_{p \in \mathbb{Z}} \mathrm{Hom}_{\mathcal{C}} ({{K}}^p, L^{p+n}),
\]
the differential
\[
\delta_n \colon \underline{\mathrm{Hom}}^n(K, L) \longrightarrow \underline{\mathrm{Hom}}^{n+1}(K, L)
\]
being given by the formula
\[
\delta_n(f)=d_L \circ f + (-1)^{n+1} f \circ d_{{K}}.
\]
\item[--] All three categories $\mathrm{C}^{\mathrm{b}}(\mathcal{C})$, $\mathrm{C}^{-}(\mathcal{C})$, $\mathrm{C}^{+}(\mathcal{C})$ are dg subcategories of $\mathrm{C}(\mathcal{C})$.
\vspace{0.2cm}
\item[--] For any complexes $K$ and $L$, we use dashed arrows for morphisms in $\underline{\mathrm{Hom}}^0(K, L)$, and plain arrows for morphisms in $Z^0(\underline{\mathrm{Hom}}(K, L))$, that is for closed morphisms of degree zero. 
\vspace{0.2cm}
\item[--]
For any arbitrary dg morphism  $\varphi \colon K \dashrightarrow L$, we denote by $\partial \varphi $ its differential considered as an element in $Z^0(\underline{\mathrm{Hom}}(K, L[1]))$. Hence, $\partial  \varphi  \colon K \longrightarrow L[1]$.
\end{enumerate}
\par \medskip
Let $f \colon {{K}} \longrightarrow L$ be a morphism of complexes of $\mathcal{C}$. Recall that the cone of $f$ is the complex $K[1] \oplus L$ endowed with the differential $\begin{pmatrix}
d_{K[1]} &  0 \\ 
f[1] & d_L
\end{pmatrix}$.
We denote by 
\[
\begin{cases}
\kappa \colon \mathrm{cone}\, (f) \dashrightarrow L\\
\sigma \colon {{K}} \dashrightarrow \mathrm{cone}\,(f)[-1]
\end{cases}
\] 
the natural projection and injection respectively. The following lemma is straightforward:
\begin{lemma} \label{MC} $ $ \par
\begin{enumerate} 
\item[(i)] The composition 
$\mathrm{cone}\, (f) \longrightarrow {{K}}[1] \longrightarrow L[1]$ is $ \partial \kappa $. \vspace{2pt}
\item[(ii)] The composition ${{K}} \longrightarrow L \longrightarrow  \mathrm{cone}\, (f) $
is $\partial \sigma$.
\end{enumerate}
\end{lemma}
Let us now consider another morphism $T \xlongrightarrow{u} \mathrm{cone}\,(f)$ and assume that the composite map
\[
T \xlongrightarrow{u} \mathrm{cone}\,(f) \xlongrightarrow{\pi} {{K}}[1]
\]
is homotopic to zero.
\begin{lemma}\label{lift}
If $\rho \colon T\dashrightarrow {{K}}$ satisfies $\partial  \rho =\pi \circ u$, then the map 
$\hat{u}$ defined by $\hat{u}=\kappa \circ u - f \circ \rho$ is a morphism of complexes, and the diagram
\[
\xymatrix{&L \ar[d]^-{j} \\
T \ar[rd]_-{ \partial \rho}\ar[ru]^-{\widehat{u}} \ar[r]^-{u} &\mathrm{cone}\,(f) \ar[d]^-{\pi}\\
& {{K}}[1] \ar[d]^-{+1}\\
&
}
\]
commutes in the homotopy category $\mathrm{K}(\mathcal{C})$. More precisely, 
$u-j \circ \hat{u}=\partial  (\sigma \circ \rho) $.
\end{lemma}

\begin{proof}
Let us write $u=(\alpha, \beta)$ where $\alpha =\pi \circ u$ and $\beta=\kappa \circ u$. Then $\alpha=\partial \rho=-d_{{{K}}} \circ \rho + \rho \circ d_T$ and $\beta \circ d_T=f \circ \alpha+d_L \circ \beta$. Hence
\begin{align*}
\hat{u} \circ d_T&=\beta \circ d_T-f \circ \rho \circ d_T \\
&=(f \circ \alpha+d_L \circ \beta) - (f \circ d_{{K}} \circ \rho+f \circ \alpha)\\
&=d_L \circ (\beta-f \circ \rho)\\
&=d_L \circ \hat{u}.
\end{align*}
Now $u-j \circ \hat{u} = (\alpha, f \circ \rho)=\partial (\sigma \circ \rho)$ since
\[
\partial (\sigma \circ \rho)=\partial (\rho, 0)=(-d_{{K}} \circ \rho, f \circ \rho)+(\rho \circ d_T, 0).
\]
\end{proof}
\begin{lemma} \label{doublecone}
For any morphism $f \colon K \longrightarrow L$, we have a canonical isomorphism of complexes
\[
\mathrm{cone}\,\{L \longrightarrow \mathrm{cone}\,(f)\}[-1] \simeq K \oplus \mathrm{cone}\,\, \mathrm{id}_{L} [-1].
\]
\end{lemma}
\begin{proof}
The complex  $Z=\mathrm{cone}\,\{L \longrightarrow \mathrm{cone}\,(f)\}[-1]$ is $L \oplus K \oplus L[-1]$ with differential given by the matrix
\[
d_Z=\begin{pmatrix}
d_L & 0 & 0 \\ 
0 & d_K & 0 \\ 
-\mathrm{id} & -f & -d_L
\end{pmatrix} .
\] 
The second projection defines an epimorphism from $Z$ to $K$, which admits a retraction given by $(f, -\mathrm{id}, 0)$. Hence $Z \simeq K \oplus T$ where
$T=L \oplus L[-1]$ endowed with the differential
\[
d_T=\begin{pmatrix}
d_L & 0 \\ 
-\mathrm{id} & -d_L
\end{pmatrix}.
\]
\end{proof}

Let us recall the link between cones and total complexes. Let $(K_{i,j}, d_{i,j}, \delta_{i,j})$ be a bounded double complex of objects in $\mathcal{C}$, as shown in the picture below: 
\[
\xymatrix{
K_{p, j} \ar[r]^-{d_{p, j}}  \ar[d]^-{\delta_{p, j}} & K_{p, j+1} \\
K_{p+1, j} &
}
\]
The corresponding total complex is
\[
\mathrm{Tot}(K)=\bigoplus_{j \in \mathbb{Z}} K_{\bullet, j} [-j]
\]
the differential being given on each factor $K_{\bullet, j}$ par $d_{\bullet, j} +(-1)^j \delta_{\bullet, j}$.
In the sequel, we will consider a bounded double complex as the bounded complex of its columns, that is as an element in $\mathrm{C}^{\mathrm{b}}(\mathrm{C}^{\mathrm{b}}(\mathcal{C}))$.
Given such double complex $K_{ij}$, let $p+1$ be the largest nonzero column index of $K$, and denote by $K'$ the double complex obtained by removing the last column. We define a morphism
\[
\Lambda \colon \mathrm{Tot} (K') \longrightarrow K_{\bullet, \,p+1} [-p]
\]
as follows: $\Lambda$ is zero on all columns of $K'$ except the last one, and $\Lambda =-d_{\bullet, \,p}$
on the last column $K_{\bullet, \,p}[-p]$.

\begin{lemma} \label{aleph}
The morphism $\Lambda$ is closed, and $\mathrm{Tot} (K)=\mathrm{cone}\, \Lambda[-1]$.
\end{lemma}
  
\begin{proof}
Let us first check that $\Lambda$ is closed.
The quantity $\Lambda \circ d_{\mathrm{Tot}(K')}-(-1)^p \delta_{\bullet, \,p+1} \circ \Lambda$ obviously vanishes on all the $p-2$ first columns.  On $K_{ \bullet,\, p-1,}[-(p-1)]$, it is $-d_{\bullet, \,p} \circ d_{\bullet, \,p-1, }$, which is also zero. Lastly, on the component $K_{\bullet, \,p}[-p]$, it is 
\[
(-d_{\bullet+1, \,p}) \circ (-1)^p\delta_{\bullet, \, p}- (-1)^{p} \delta_{\bullet, \,p+1} \circ (-d_{\bullet, \, p}),
\] 
which is zero. For the last point, the underlying object of $\mathrm{cone}\, \Lambda[-1]$ is
$K' \oplus K_{\bullet, \,p+1} [-p-1]$, which is $K$. The differential on $K_{\bullet, \,p}$ is
\[
- (-(-1)^p \delta_{\bullet, \,p} - d_{\bullet, \, p})= (-1)^p \delta_{\bullet, \,p} + d_{\bullet, \, p} = d_{\mathrm{Tot}(K)}.
\]
The other verifications are left to the reader.
\end{proof}

\subsection{Iterated cones for dg-functors} \label{puf}
If $\mathcal{C}_1$ and $\mathcal{C}_2$ are two additive categories, let us recall some elementary facts:
\begin{enumerate}
\vspace{0.2cm}
\item[--] The objects of the category $\mathrm{Fct}_{\mathrm{dg}}(\mathrm{C}^{\mathrm{b}}(\mathcal{C}_1), \mathrm{C}^{\mathrm{b}}(\mathcal{C}_2))$ are the additive functors from $\mathrm{C}^{\mathrm{b}}(\mathcal{C}_1)$ to $\mathrm{C}^{\mathrm{b}}(\mathcal{C}_2)$ that commute with shift and cones. 
\vspace{0.2cm}
\item[--]The natural morphism $\mathcal{C}_1 \longrightarrow \mathrm{C}^{\mathrm{b}}(\mathcal{C}_1)$ that maps any object in $\mathcal{C}_1$ to the complex having this single object in degree zero yields a restriction functor
\[
\mathrm{Fct}_{\mathrm{dg}}(\mathrm{C}^{\mathrm{b}}(\mathcal{C}_1), \mathrm{C}^{\mathrm{b}}(\mathcal{C}_2)) \longrightarrow \mathrm{Fct}\,(\mathcal{C}_1, \mathrm{C}^{\mathrm{b}}(\mathcal{C}_2))
\]
where on the right hand side we consider all additive functors. This functor is an equivalence of categories.
\vspace{0.2cm} 
\item[--] Let us denote by $\mathrm{Fct}_{\mathrm{dg}}^*(\mathrm{C}^{\mathrm{b}}(\mathcal{C}_1), \mathrm{C}^{\mathrm{b}}(\mathcal{C}_2))$ the subcategory of \textit{bounded} dg-functors, that is elements of $\mathrm{Fct}_{\mathrm{dg}}(\mathrm{C}^{\mathrm{b}}(\mathcal{C}_1), \mathrm{C}^{\mathrm{b}}(\mathcal{C}_2))$ corresponding to $
\underset{n}{\varinjlim} \,\mathrm{Fct}(\mathcal{C}_1, \mathrm{C}^{[-n, n]}(\mathcal{C}_2))
$
via the above equivalence. Then there are extension functors from  $\mathrm{Fct}_{\mathrm{dg}}^*(\mathrm{C}^{\mathrm{b}}(\mathcal{C}_1), \mathrm{C}^{\mathrm{b}}(\mathcal{C}_2))$ to  $\mathrm{Fct}_{\mathrm{dg}}(\mathrm{C}^{*}(\mathcal{C}_1), \mathrm{C}^{*}(\mathcal{C}_2))$ for $* \in \{ \varnothing, +, - \}$.
\vspace{0.2cm}
\end{enumerate}
For any bounded complex
\[
\xymatrix{
\cdots \ar@{->}[r]^-{t^{n-2}}&  T^{n-1} \ar@{->}[r]^-{t^{n-1}}& T^{n} \ar@{->}[r]^-{t^n} & T^{n+1} \ar@{->}[r]^-{t^{n+1}} &\cdots}
\]
of objects of $\mathrm{Fct}_{\mathrm{dg}}(\mathrm{C}^{\mathrm{b}}(\mathcal{C}_1), \mathrm{C}^{\mathrm{b}}(\mathcal{C}_2))$, we define a dg-functor $\Psi_{T}$ as follows: for any bounded complex $K$, there is a complex of bounded complexes
\[
\ldots \xrightarrow{t^{j-2}_K} T^{j-1}(K)  \xrightarrow{t^{j-1}_K} T^{j}(K)  \xrightarrow{t^{j}_K} T^{j+1}(K) \ldots
\]
which defines a double complex $\{ T^j (K_{i-j}), t^j_K, T^j(d_K) \}$. Then we define
\[
\Psi_{T}(K)= \mathrm{Tot} \{ T^j (K_{i-j}), t^j_K, T^j(d_K) \}.
\]
It is an exercise to check that the functor
\[
\Psi \colon \mathrm{C}^{\mathrm{b}}(\mathrm{Fct}_{\mathrm{dg}}(\mathrm{C}^{\mathrm{b}}(\mathcal{C}_1), \mathrm{C}^{\mathrm{b}}(\mathcal{C}_2))) \longrightarrow \mathrm{Fct}_{\mathrm{dg}}(\mathrm{C}^{\mathrm{b}}(\mathcal{C}_1), \mathrm{C}^{\mathrm{b}}(\mathcal{C}_2))
\]
is a dg-functor, where on the left hand side $\mathrm{Fct}_{\mathrm{dg}}(\mathrm{C}^{\mathrm{b}}(\mathcal{C}_1), \mathrm{C}^{\mathrm{b}}(\mathcal{C}_2))$ is considered as an additive category (and not its dg enhancement); we call $\Psi_T$ the iterated cone of $T$. 
\par \medskip
Let us explain how iterated mapping cones can be constructed by taking successive ordinary cones. For a bounded complex $T$ of dg-functors, let $p=\mathrm{max	}\, \{i \in \mathbb{Z}\, \, \textrm{such that}\, \, T^{i+1} \neq 0\}$, and let $T'$ be the complex of functors obtained by removing the last functor  $T^{p+1}$. We define a morphism
\[
\Lambda \colon \Psi_{T'} \dashrightarrow T^{p+1}[-p]
\]
as follows: $\Lambda_K$  is zero on all factors $T^i(K)[-i]$ for $0 \leq i \leq p-1$ and $-t^p_K$ on $T^{p}(K)[-p]$. Then it follows directly from Lemma \ref{aleph} that $\Lambda$ is a closed morphism and that 
\[
\Psi_{T} = \mathrm{cone}\,\Lambda [-1],
\] 
whence the terminology ``iterated cones".

\subsection{Lax monoidal functors}
In this section, we recall the notion of lax monoidal functors between tensor categories. These functors form a weaker class than the usual monoidal functors (also called tensor functors), as introduced for instance in \cite[\S 4.2]{CatSheaves}.
\par \medskip
Let $(\mathcal{S}_1, \otimes)$ and $(\mathcal{S}_2, \otimes)$ be unital tensor categories, with unit elements $\textbf{1}_{\mathcal{S}_1}$ and $\textbf{1}_{\mathcal{S}_2}$, and let $H$ be an additive functor from $\mathcal{S}_1$ to $\mathcal{S}_2$. Besides, assume to be given two morphisms
\[
\begin{alignedat}{1}
\mathfrak{m} &\colon  H(\star) \otimes H(\star\star) \longrightarrow H(\star \otimes \star \star)\\
\mu &\colon  {\mathbf{1}}_{\mathcal{S}_2} \longrightarrow H({\mathbf{1}}_{\mathcal{S}_1})
\end{alignedat}
\]
where in the first line $\mathfrak{m}$ is a natural transformation between functors from $\mathcal{S}_1 \times \mathcal{S}_1$ to $\mathcal{S}_2$.
\begin{definition}  \label{mul}
The triple $(H, \mathfrak{m}, \mu)$ defines a lax monoidal functor if:
\begin{enumerate}
\item[--] For any $K_1$, $K_2$, $K_3$ in $\mathcal{S}_1$, the diagram
\[
\xymatrix{
(H(K_1) \otimes H(K_2)) \otimes H(K_3) \ar[d]^-{\mathfrak{m}_{K_1, K_2} \otimes \,\mathrm{id}_{H(K_3)}} \ar[r]^-{\sim} & H(K_1) \otimes (H(K_2) \otimes H(K_3)) \ar[d]^-{\mathrm{id}_{H(K_1)} \otimes \,\mathfrak{m}_{K_2, K_3}} \\
H(K_1 \otimes K_2) \otimes H(K_3) \ar[d]^-{\mathfrak{m}_{K_1 \otimes K_2, K_3}} & H(K_1) \otimes H(K_2 \otimes K_3) \ar[d]^-{\mathfrak{m}_{K_1, K_2 \otimes K3}} \\
H((K_1 \otimes K_2) \otimes K_3) \ar[r]^-{\sim} & H(K_1 \otimes (K_2 \otimes K_3))
}
\]
commutes. \vspace{0.2cm}
\item[--] For any $K$ in $\mathcal{S}_1$, the compositions
\[
\begin{alignedat}{4}
H(K)& \!\simeq H(K) \! \otimes {\mathbf{1}}_{\mathcal{S}_2} & \xrightarrow{\mathrm{id} \otimes  \mu} H(K) \otimes H(\mathbf{1}_{\mathcal{S}_1}) &\xrightarrow{\mathfrak{m}_{K, {\mathbf{1}}_{\mathcal{S}_1}}} H(K \otimes {\mathbf{1}}_{\mathcal{S}_1}) & \simeq H(K)\\
H(K) &\! \simeq \!{\mathbf{1}}_{\mathcal{S}_2} \otimes H(K) & \xrightarrow{\mu \otimes \mathrm{id}} H(\mathbf{1}_{\mathcal{S}_1}) \otimes H(K) &\xrightarrow{\mathfrak{m}_{{\mathbf{1}}_{\mathcal{S}_1}, K}} H( {\mathbf{1}}_{\mathcal{S}_1} \otimes K) & \,\simeq H(K)
\end{alignedat}
\]
are the identity morphisms.
\end{enumerate}
\end{definition}
\begin{remark} $ $
If $H$ is a lax monoidal endofunctor of $\mathcal{S}$, then $H(\mathbf{1}_{\mathcal{S}})$ is a ring object in $\mathcal{S}$ and for any element $K$ in $\mathcal{S}$, $H(K)$ is a left and right module over this ring object. 
%If $\mathcal{C}_1$ and $\mathcal{C}_2$ are additive categories, then any bounded dg-functor from $\mathrm{C}^{\mathrm{b}}(\mathcal{C}_1)$ to $\mathrm{C}^{\mathrm{b}}(\mathcal{C}_2)$ having a lax monoidal structure extends naturally to a lax monoidal dg-functor from $\mathrm{C}^*(\mathcal{C}_1)$ to $\mathrm{C}^*(\mathcal{C}_2)$ where $* \in \{ \varnothing, +, -\}$.
\end{remark}
\begin{definition}
Let $(H_1, \mathfrak{m}_1, \mu_1)$ and $(H_2, \mathfrak{m}_2, \mu_2)$ be two lax monoidal  functors. A morphism $\varphi \colon H_1 \longrightarrow H_2$ is multiplicative if the two following diagrams commute 
\vspace{0.2cm}
\[
\xymatrix{
H_1(K) \otimes H_1(L) \ar[r]^-{\mathfrak{m}_1} \ar[d]^-{\varphi_K \otimes \,	 \varphi_L} & H_1(K \otimes L) \ar[d]^-{\varphi_{K \otimes L}} \\
H_2(K) \otimes H_2(L) \ar[r]^-{\mathfrak{m}_2} &H_2(K \otimes L)
}
\qquad 
\xymatrix{
&H_1(\mathbf{1}_{   \mathcal{S}_1   }) \ar[dd]^-{   \varphi_{     \mathbf{1}_{\mathcal{S}_1}      }     } \\
\mathbf{1}_{\mathcal{S}_2} \ar[ru]^-{\mu_1} \ar[rd]_-{\mu_2}& \\
&H_2(\mathbf{1}_{\mathcal{S}_1})
}
\]
\end{definition}

If $H_1 \colon \mathcal{S}_1 \longrightarrow  \mathcal{S}_2$ and $H_2 \colon  \mathcal{S}_2 \longrightarrow  \mathcal{S}_3$ are two lax monoidal functors, then so is $H_1 \circ H_2$, the multiplication being given by the composition
\[
H_2(H_1 (K)) \otimes H_2 (H_1(L)) \xrightarrow{\mathfrak{m}_2} H_2 (H_1(K) \otimes H_1(L)) \xrightarrow{H_2(\mathfrak{m}_1)} H_2(H_1(K \otimes L)) 
\]
and the unit is 
\[
\mathbf{1}_{\mathcal{S}_3} \xrightarrow{\mu_2} H_2(\mathbf{1}_{\mathcal{S}_2}) \xrightarrow{H_2(\mu_1)} H_2(H_1(\mathbf{1}_{\mathcal{S}_1})).
\]
The category of lax monoidal endofunctors of a tensor category $\mathcal{S}$ is itself a tensor category, the tensor structure being the composition, and the unit object being the identity endofunctor.
\par \medskip
Let $\mathcal{C}$ be an abelian tensor category, let $H$ be a lax monoidal endofunctor of $\mathrm{C}^{\mathrm{b}}(\mathcal{C})$, and assume that $H$ fits into an exact sequence
\[
0 \longrightarrow N \xlongrightarrow{\iota} H \xlongrightarrow{p} \mathrm{id}_{\mathrm{C}^{\mathrm{b}}(\mathcal{C})} \longrightarrow 0
\]
where $p$ is multiplicative.
\begin{proposition} \label{brique}
The functor $
\mathrm{cone}\,\, (N \longrightarrow H)$
is naturally a lax monoidal functor, and the natural morphism from $\mathrm{cone}\,\, (N \longrightarrow H)$ to $\mathrm{id}_{\mathrm{C}^{\mathrm{b}}(\mathcal{C})}$ is multiplicative.
\end{proposition}
\begin{proof}
For any objects $K$ and $L$ of $\mathrm{C}^{\mathrm{b}}(\mathcal{C})$, we have a commutative diagram of elements in $\mathrm{C}^{\mathrm{b}}(\mathcal{C})$:
\begin{equation} \label{huez}
\xymatrix@C=10pt{
N(K) \otimes N(L) \ar[r] \ar[d] & N(K) \otimes H(L) \oplus H(K) \otimes N(L) \ar[r]&H(K) \otimes H(L) \ar[r] \ar[d]^-{\mathfrak{m}_{K, L}} &K \otimes L \ar@{=}[d] \ar[r] & 0 \ar@{=}[d]\\
0 \ar[r] & N(K \otimes L) \ar[r] & H(K \otimes L) \ar[r] & K \otimes L \ar[r] & 0
}
\end{equation}
where the first horizontal arrow is ${(\iota_K \otimes \mathrm{id}_{H(L)}, -\mathrm{id}_{H(K)} \otimes \iota_L)}$. Since the second line is exact,  there is a unique morphism
\[
\phi \colon N(K) \otimes H(L) \oplus H(K) \otimes N(L) \longrightarrow N(K \otimes L)
\]
making \eqref{huez} commutative. We define the multiplicative structure on  
$Y=\mathrm{cone}\, \, (N \longrightarrow H)$ by extracting the leftmost part of \eqref{huez}, namely
\[
\xymatrix@C=17pt{
N(K) \otimes N(L) \ar[r] \ar[d] & N(K) \otimes H(L) \oplus H(K) \otimes N(L) \ar[r] \ar[d]^-{\phi} &H(K) \otimes H(L) \ar[d]^-{\mathfrak{m}_{K, L}} \\
0 \ar[r] &N(K \otimes L) \ar[r] &H(K \otimes L)
}
\]
Taking the total complex of each line yields a morphism from $Y(K) \otimes Y(L)$ to $Y(K \otimes L)$.
We leave to the reader the tedious verification that this morphism is associative.
The unit of $Y$ is analogously obtained using the diagram
\[
\xymatrix{
0 \ar[r] \ar[d]& \mathbf{1}_{\mathrm{C}^{\mathrm{b}}(\mathcal{C})} \ar[d]^-{\mu} \\
N(\mathbf{1}_{\mathrm{C}^{\mathrm{b}}(\mathcal{C})}) \ar[r] &  H(\mathbf{1}_{\mathrm{C}^{\mathrm{b}}(\mathcal{C})})
}
\]
where $\mu$ is the unit of $H$. Lastly, \eqref{huez} yields a commutative diagram
\[
\xymatrix{Y(K) \otimes Y(L) \ar[d] \ar[r] &K \otimes L \ar@{=}[d] \\
Y(K\otimes L) \ar[r]& K \otimes L
}
\]
which gives the multiplicativity of the morphism from $Y$ to $\mathrm{id}_{\mathrm{C}^{\mathrm{b}}(\mathcal{C})}$.
\end{proof}

\section{Construction of dg-endofunctors of $\mathrm{C}^{\mathrm{b}}(\mathcal{C})$}
\label{pita}

\subsection{Canonical functors} \label{natanz}
\subsubsection{Main construction} 
Let $\mathcal{C}$ be an additive category, and assume to be given a pair $(G, \Theta)$ where $G$ is in $\mathrm{EndFct}_{\mathrm{dg}}(\mathrm{C}^{\mathrm{b}}(\mathcal{C}))$ and
\[
\Theta \colon \mathrm{id}_{\mathrm{C}^{\mathrm{b}}(\mathcal{C})} \longrightarrow G
\]
is a dg morphism. For any nonnegative integer $n$ we define a dg morphism 
\[
S_n\colon G^n\longrightarrow G^{n+1}
\] 
by the formula
\begin{equation} \label{sum}
S_n=\displaystyle\sum_{i=0}^{n} (-1)^{n+i+1} G^{n-i} \left(\Theta_{G^i}  \right).
\end{equation}
\begin{lemma}
For any nonnegative integer $n$, $S_{n+1} \circ S_n=0$.
\end{lemma}
\begin{proof}
We have 
\begin{equation} \label{doublesomme}
S_{n+1} \circ S_n =\displaystyle\sum_{i=0}^{n+1} \displaystyle\sum_{j=0}^n
(-1)^{i+j+1} G^{n+1-i}(\Theta_{G^i}) \circ G^{n-j}(\Theta_{G^j}).
\end{equation}
If $i \leq j$, we can write
\begin{align*}
G^{n+1-i}(\Theta_{G^i}) \circ G^{n-j}(\Theta_{G^j}) &=G^{n-j} (G^{j-i+1}(\Theta_{G^i})) \circ G^{n-j}(\Theta_{G^i})\\
&=G^{n-j} (G^{j-i+1}(\Theta_{G^i} ) \circ \Theta_{G^i}).
\end{align*}
For any morphism $f \colon U \longrightarrow V$ of bounded complexes of $\mathcal{C}$, we have 
\[
G(f) \circ \Theta_U=\Theta_V \circ f. 
\]
We put $f=G^{j-i}(\Theta_{G^i(K)})$, where $K$ is any bounded complex of objects in ${\mathcal{C}}$. This gives 
\[
G^{j-i+1}(\Theta_{G^i}) \circ \Theta_{G^i}=\Theta_{G^{j+1}} \circ G^{j-i}(\Theta_{G^i})
\]
so that we get
\begin{align*}
G^{n+1-i}(\Theta_{G^i}) \circ G^{n-j}(\Theta_{G^j}) &=G^{n-j}(\Theta_{G^{j+1}}) \circ G^{n-i}(\Theta_{G^i}) \\
&= G^{n+1-(j+1)}(\Theta_{G^{j+1}}) \circ G^{n-i}(\Theta_{G^i}).
\end{align*}
Hence in the double sum \eqref{doublesomme}, every component indexed by a couple $(i, j)$ with $i \leq j$ cancels with the component indexed by $(j+1, i)$.  

\end{proof}

\begin{definition} $ $
\par
\begin{enumerate}
\item[(i)] The functor $F_n$ is the element in $\mathrm{EndFct}_{\mathrm{dg}}(\mathrm{C}^{\mathrm{b}}(\mathcal{C} ))$ obtained as the iterated cone (for the definition, see \S \ref{puf})
\[
\xymatrix@C=30pt{
\mathrm{id}_{\mathrm{C}^{\mathrm{b}}(\mathcal{C})} \ar[r]^-{S_0}& G  \ar[r]^-{S_1}& G^2  \ar[r]^-{S_2}& \cdots  \ar[r]^-{S_{n-2}}& G^{n-1} \ar[r]^-{S_{n-1}}& G^n,
}
\]
where the functor $\mathrm{id}_{\mathrm{C}^{\mathrm{b}}(\mathcal{C})}$ sits in degree zero. 
\vspace{0.2cm}
\item[(ii)] The transformation $\Theta_n \colon F_n \dashrightarrow G^{n+1}[-n]$ is defined by the composition
\[
(\Theta_n)_{K}  \colon \bigoplus_{i=0}^n G^i(K)[-i] \longrightarrow  G^n(K)[-n] \xrightarrow {-(S_n)_K} G^{n+1}(K)[-n].
\]
\vspace{0.2cm}
\item[(iii)] The transformation $\rho_n \colon F_n \dashrightarrow G(F_n)[-1]$ is the map
\[
(\rho_n)_K \colon \bigoplus_{i=0}^n G^i(K) [-i] \dashrightarrow  \bigoplus_{i=1}^{n+1} G^i(K) [-i]
\]
that is zero on the first factor $K$, and the identity morphism on the components $G^i(K)[-i]$ for $1 \leq i \leq n$.
\end{enumerate}
\end{definition}

\begin{theorem} \label{hard}
The elements satisfy the following properties:
\begin{enumerate}
\item[(i)] $F_0=\mathrm{id}$ and $\Theta_0=\Theta$.
\vspace{0.2cm}
\item[(ii)] $\Theta_n$ is closed and $F_{n}=\mathrm{cone}\, (\Theta_{n-1})[-1]$. In particular there is an associated natural transformation $\Pi_n$ from $F_{n}$ to $F_{n-1}$ and we have a cone  exact sequence 
\[
\xymatrix{
0 \ar@{->}[r]& G^{n}[-n]  \ar@{->}[r]^-{\tau_n} & F_{n}  \ar@{->}[r]^-{\Pi_n}& F_{n-1} \ar@{->}[r]& 0.
}
\]

\vspace{0.2cm}
\item[(iii)] $\Theta_{F_n}-G(\tau_n) \circ \Theta_n=\partial \rho_n$.
\[
\xymatrix{ & G^{n+1}[-n] \ar@{->}[d]^-{G(\tau_n)} \\
F_{n} \ar@{->}[r]_-{\Theta_{F_n}} \ar@{->}[ru]^-{\Theta_{n}}  &G(F_{n}) 
}
\]
\item[(iv)] If $G$ is bounded, then all $F_{n}$ are also bounded.
\end{enumerate}
\end{theorem}

\begin{proof}
The first point is obvious, and the second point is a consequence of Lemma \ref{aleph}. Let us prove the third point. For this, we write down explicitly everything in terms of bicomplexes of complexes .
\par \medskip
The morphisms $\Theta_n$ and $\tau_n$ are obtained by applying the iterated cone functor $\Psi$ defined in \S \ref{puf} to each line of the two commutative diagrams of dg-functors
\[
\xymatrix@C=30pt{
\mathrm{id}_{\mathrm{C}^{\mathrm{b}}(\mathcal{C})} \ar[r]^-{S_0} \ar[d]& G  \ar[r]^-{S_1} \ar[d] & G^2  \ar[r]^-{S_2} \ar[d] & \cdots  \ar[r]^-{S_{n-2}}  & G^{n-1} \ar[r]^-{S_{n-1}} \ar[d] & G^n \ar[d]^-{-S_n} \\
0 \ar[r] & 0 \ar[r] & 0 \ar[r] & \cdots \ar[r] &0 \ar[r] & G^{n+1}
}
\]
and
\[
\xymatrix{
0 \ar[r] \ar[d]& 0 \ar[r] \ar[d] & 0 \ar[r] \ar[d] & \cdots  \ar[r]  &0 \ar[r] \ar[d] & G^{n} \ar[d] \\
\mathrm{id}_{\mathrm{C}^{\mathrm{b}}(\mathcal{C})} \ar[r]^-{S_0} & G  \ar[r]^-{S_1}  & G^2  \ar[r]^-{S_2}  & \cdots  \ar[r]^-{S_{n-2}}  & G^{n-1} \ar[r]^-{S_{n-1}}  & G^n 
} \vspace{0.2cm}
\]
respectively. Hence $G(\tau_n) \circ \Theta_n \colon F_n \longrightarrow G(F_n)$ is induced (via the same procedure) by the diagram
\[
\xymatrix@C=40pt{
\mathrm{id}_{\mathrm{C}^{\mathrm{b}}(\mathcal{C})} \ar[r]^-{S_0} \ar[d]^-{0}& G  \ar[r]^-{S_1} \ar[d]^-{0} & G^2  \ar[r]^-{S_2} \ar[d]^-{0} & \cdots  \ar[r]^-{S_{n-2}} & G^{n-1} \ar[r]^-{S_{n-1}} \ar[d]^-{0} & G^n \ar[d]^-{-S_n} \\
G \ar[r]^-{G(S_0)} & G^2  \ar[r]^-{G(S_1)} & G^3  \ar[r]^-{G(S_2)}  & \cdots  \ar[r]^-{G(S_{n-2})}  & G^{n} \ar[r]^-{G(S_{n-1})}  & G^{n+1}  
}\vspace{0.2cm}
\]
From this, we deduce that $\Theta_{F_n}-G(\tau_n) \circ \Theta_n$ is induced by the diagram
\[
\xymatrix@C=40pt{
\mathrm{id}_{\mathrm{C}^{\mathrm{b}}(\mathcal{C})} \ar[r]^-{S_0} \ar[d]^-{\Theta}& G  \ar[r]^-{S_1} \ar[d]^-{\Theta_G} & G^2  \ar[r]^-{S_2} \ar[d]^-{\Theta_{G^2}} & \cdots  \ar[r]^-{S_{n-2}} & G^{n-1} \ar[r]^-{S_{n-1}} \ar[d]^-{\Theta_{G^{n-1}}} & G^n \ar[d]^-{\Theta_{G^n}+S_n}  \\
G \ar[r]^-{G(S_0)} & G^2  \ar[r]^-{G(S_1)} & G^3  \ar[r]^-{G(S_2)}  & \cdots  \ar[r]^-{G(S_{n-2})}  & G^{n} \ar[r]^-{G(S_{n-1})}  & G^{n+1}
} \vspace{0.2cm}
\]
Now we express $\rho_n$ using the diagram of dg-functors\footnote{The minus signs on the bottom line come from the shift by $1$.}:
\[
\xymatrix@C=40pt{
\mathrm{id}_{\mathrm{C}^{\mathrm{b}}(\mathcal{C})} \ar[r]^-{S_0} \ar[d]^-{0}& G  \ar[r]^-{S_1} \ar@{=}[d]& G^2  \ar[r]^-{S_2} \ar@{=}[d] & \cdots  \ar[r]^-{S_{n-1}} & G^{n} \ar[r] \ar@{=}[d] & 0 \ar@{=}[d] \\
0 \ar[r]& G  \ar[r]^-{-G(S_0)} & G^2  \ar[r]^-{-G(S_1)}  & \cdots  \ar[r]^-{-G(S_{n-2})}  & G^{n} \ar[r]^-{-G(S_{n-1})}  & G^{n+1}  
}\vspace{0.2cm}
\]
This diagram is noncommutative, and induces the morphism $\rho_n$  (which is not closed) by taking the iterated cones of each line. Now $\partial  \rho_n $ is the closed morphism given by the diagram
\[
\xymatrix@C=50pt@R=35pt{
\mathrm{id}_{\mathrm{C}^{\mathrm{b}}(\mathcal{C})} \ar[r]^-{S_0} \ar[d]^-{-S_0}& G  \ar[r]^-{S_1} \ar[d]^-{-G(S_0)-S_1}& \cdot  \ar[r]  &G^{n-1}   \ar[r]^-{S_{n-1}} \ar[d]^-{-G(S_{n-2})-S_{n-1}} & G^{n} \ar[d]^-{-G(S_{n-1})}   \\
G  \ar[r]^-{G(S_0)} & G^2  \ar[r]^-{G(S_1)}  & \cdots  \ar[r]^-{G(S_{n-2})}  & G^{n} \ar[r]^-{G(S_{n-1})}  & G^{n+1}  }
\vspace{0.2cm}
\]
By the very definition of the $S_i$'s, 
\begin{equation} \label{ok}
(S_{i-1})+S_i=-\Theta_{G^i}
\end{equation} 
which finishes the proof.
\end{proof}

\subsubsection{Derived invariance}
In this section, we study some specific properties of the functors $F_n$ when $G$ is exact.
\begin{lemma}
Assume that $G$ is an exact functor. Then all functors $F_n$ are also exact.
\end{lemma}

\begin{proof}
If $G$ is exact, then so are all the functors $G^n$. Let us now remark that the cone of a morphism of exact endofunctors of $\mathrm{C}^{\mathrm{b}}(\mathcal{C})$ is also exact. Hence the result follows by induction, since 
\[
F_{n+1} \, [1]\simeq \mathrm{cone}\, \, (\Theta_{n} \colon F_n \longrightarrow G^{n+1}[-n]).
\]
\end{proof}

\begin{definition} \label{groix}
Let $G_1$ and $G_2$ two endofunctors of $\mathrm{C}^{\mathrm{b}}(\mathcal{C})$, and let 
\[
\Gamma \colon G_1 \longrightarrow G_2
\] 
be a closed dg morphism. We say that $\Gamma$ is a quasi-isomorphism if for any bounded complex $K$ of elements of $\mathcal{C}$, the morphism 
\[
\Gamma(K) \colon G_1(K) \longrightarrow G_2(K)
\] 
is a quasi-isomorphism.
\end{definition}

\begin{remark}
If $\Gamma \colon G_1 \longrightarrow G_2$ is a quasi-isomorphism between exact functors, it induces a true isomorphism between the associated endofunctors of $\mathrm{D}^{\mathrm{b}}(\mathcal{C})$. 
\end{remark}

\begin{proposition} \label{morita}
Let $G_1$ and $G_2$ two exact endofunctors of the category $\mathcal{C}^{\mathrm{b}}(\mathcal{C})$
endowed with morphisms $\Theta_i \colon \mathrm{id}_{\mathrm{C}^{\mathrm{b}}(\mathcal{C})} \longrightarrow G_i$ for $i=1, 2$, and let $\Gamma \colon G_1 \longrightarrow G_2$ be a quasi-isomorphism such that $\Psi \circ \Theta_1=\Theta_2$. Then for any positive integer $n$, $\Gamma$ induces quasi-isomorphisms between $F_n^1$ and $F_n^2$.
\end{proposition}

\begin{proof}
Since $G_1$ and $G_2$ are exact, for any positive integer $n$, the morphism $\Gamma$ induces a quasi-isomorphism between $G_1^n$ and $G_2^n$. For any positive integer $n$, we have a morphism of exact sequences
\[
\xymatrix{
0 \ar[r]& G_1^n[-n] \ar[r] \ar[d] & F_n^1 \ar[r] \ar[d] & F_1^{n-1} \ar[r] \ar[d] & 0 \\
0 \ar[r] &G_2^n[-n] \ar[r] & F_n^2 \ar[r] & F_{n-1}^2 \ar[r] & 0 \\
}
\]
and the result follows by induction.
\end{proof}
\subsection{Comparison} \label{twist}
\subsubsection{The octahedron triangle}
For any object $K$ in $\mathrm{C}^{\mathrm{b}}(\mathcal{C})$, we have a diagram in $\mathrm{D}^{\mathrm{b}}(\mathcal{C})$:
\begin{center}
\begin{tikzpicture}
\node at (-1.2,0) {$F_{n+1}(K)[1] $};
\node at (-1.2,-2) {$F_n(K)$};
\node at (2,1.8) { $\mathrm{cone}\,\Theta_{F_n(K)}$};
\node at (5.2,0) {$G(F_{n-1}(K)) $};
\node at (5.2,-2) {$G(F_{n}(K)) $};
\node at (2,-4) {$G^{n+1}(K)[-n] $};
\node at (2,-1.8) {$\scriptstyle\Theta_{F_n(K)}$};
\draw[line width=2pt,->](-0.6,0.4)--(1.2,1.52);
\draw[line width=2pt,->](2.8,1.46)--(4.5,0.4);
\draw[line width=1pt,dashed,->](1.5,1.48)--(-0.8,-1.65);
\draw[line width=1pt,dashed,->](-1.2,-0.3)--(-1.2,-1.65);
\draw[line width=.5pt,->](5,-1.65)--(5,-0.3);
\draw[line width=0.5pt,->](-1,-2.3)--(01.1,-3.7);
\draw[line width=0.5pt,->](3.1,-3.65)--(5,-2.3);
\draw[line width=1pt,dashed,->](4.6,-.3)--(2.6,-3.65);
\draw[line width=3pt,color=white](1.5,-3.65)--(-.8,-.3);
\draw[line width=0.5pt,->](1.5,-3.65)--(-.8,-.3);
\draw[line width=2.5pt,color=white,->](4.5,-1.65)--(2.3,1.5);
\draw[line width=0.5pt,->](4.5,-1.65)--(2.3,1.5);
\draw[line width=5pt,color=white,->](3.9,0)--(-0.2,0);
\draw[line width=2pt,dashed,->](3.9,0)--(-0.2,0);
\draw[line width=3pt,color=white,->](-0.4,-2)--(4.1,-2);
\draw[line width=0.5pt,->](-0.4,-2)--(4.1,-2);
\end{tikzpicture}
\end{center}
where all dashed arrows are shifted by one. The octahedron axiom yields a triangle
\[
F_{n+1}(K) \longrightarrow \mathrm{cone}\, \Theta_{F_n(K)}[-1] \longrightarrow G(F_{n-1}(K))[-1] \xrightarrow{+1}
\]
This triangle can be lifted at the level of complexes, as we will show in the next result. Before stating it, we introduce some notation: let  $\Delta_{\Theta}$ be the functor $\mathrm{cone} \, \Theta[-1]$. It is equipped with a natural left inverse $\iota \colon \Delta_{\Theta} \longrightarrow \mathrm{id}_{\mathrm{C}^{\mathrm{b}}(\mathcal{C})}$ (so it is a faithful functor).
\begin{theorem} \label{ecolo}
There is a canonical exact sequence
\[
0 \longrightarrow F_{n+1} \xlongrightarrow{\mathfrak{p}_{n+1}} \Delta_{\Theta}(F_n)  \xlongrightarrow{\nu_n} G(F_{n-1})[-1] \longrightarrow 0
\]
in $\mathrm{C}^{\mathrm{b}}(\mathcal{C})$
such that:
\begin{enumerate} 
\item[(i)] The map $\mathfrak{p}_{n+1}$ lifts $\Pi_{n+1} \colon F_{n+1} \longrightarrow F_n$ with respect to $\iota_{F_n}$, that is $\iota_{F_n} \circ \mathfrak{p}_{n+1}=\Pi_{n+1}$.
\vspace{0.2cm}
\item[(ii)] The map $\Delta_{\Theta}(\Pi_{n})-\mathfrak{p}_{n} \circ \iota_{F_n} \colon \Delta_{\Theta}(F_n) \longrightarrow \Delta_{\Theta}(F_{n-1})$ factors through $\nu_n$.
\end{enumerate}
\end{theorem}
\begin{proof}
The morphism of functors $F_{n} \longrightarrow G(F_n)$ can be represented by the diagram 
\[
\xymatrix{
\mathrm{id} \ar[r]^-{S_0} \ar[d]^-{\Theta} & G  \ar[r]^-{S_1} \ar[d]^-{\Theta_G} &  \cdots \ar[r] & G^{n-1} \ar[r]^-{S_{n-1}} \ar[d]^-{\Theta_{G^{n-1}}}  &G^{n}  \ar[d]^-{\Theta_{G_n}}    \\
G                \ar[r]_-{G(S_0)}           & G^2  \ar[r]_-{G(S_1)}     &  \cdots  \ar[r]& G^n \ar[r]_-{G(S_{n-1})} & G^{n+1}
}
\]
Hence $\mathrm{cone} \, \Theta(F_n)[-1]$ is the iterated cone associated with the complex dg-functors
\[
\mathrm{id} \xlongrightarrow{L^0} G \oplus G  \xlongrightarrow{L^1} G^2 \oplus G^2 \xlongrightarrow {L^2} \cdots \xlongrightarrow {L^{n-1}} G^{n} \oplus G^{n} \xlongrightarrow {\mathrm{pr}_2 \circ L^{n}} G^{n+1}
\]
where 
\[
\begin{cases}
L^0=\begin{pmatrix}
\Theta \\ 
-\Theta
\end{pmatrix} \vspace{0.2cm}\\
\vspace{0.2cm}
L^i=\begin{pmatrix}
S_i & 0 \\ 
-\Theta_{G^i} & -G(S_{i-1})
\end{pmatrix} \, \, \, \textrm{for} \, \, 1 \leq i \leq {n-1}.
\end{cases}
\]
Now remark that by \eqref{ok}, we have an exact sequence
\[
\xymatrix{
0 \ar[r]  & G^{i} \ar[r] \ar[d]^{S_i}& G^i \oplus G^i \ar[r] \ar[d]^-{L^i}& G^i \ar[r] \ar[d]^-{G(S_{i-1})}& 0\\
0 \ar[r]& G^{i+1} \ar[r]& G^{i+1} \oplus G^{i+1} \ar[r]& G^{i+1} \ar[r]& 0
}
\]
where the map $G^i \longrightarrow G^i \oplus G^i$ is $\begin{pmatrix}
1 \\ 
1
\end{pmatrix} $ and the one from $G^i \oplus G^i$ to $G^i$ is $((-1)^i, (-1)^{i+1})$. This gives the required exact sequence.
Let us now prove the two remaining statements in the Theorem. Point $($i$)$ is obvious. For point $($ii$)$, the morphism 
\[
\Delta_{\Theta}(\Pi_{n})-\mathfrak{p}_{n-1} \circ \iota_{F_n} \colon \Delta_{\Theta}(F_n) \longrightarrow \Delta_{\Theta}(F_{n-1})
\]
can we written as 
\[
\xymatrix{
\mathrm{id}\ar[r]^-{L^0} \ar[d]^-{\mu_0}& G \oplus G  \ar[r]^-{L^1} \ar[d]^-{\mu_1}& G^2 \oplus G^2 \ar[r]^-{L^2} \ar[d]^-{\mu_2}&\cdots \ar[r]^-{L^{n-1}} &G^{n} \oplus G^{n} \ar[r]^-{L^{n}} \ar[d]^-{\mu_n} &G^{n+1}\\
\mathrm{id}\ar[r]^-{L^0} & G \oplus G  \ar[r]^-{L^1}& G^2 \oplus G^2 \ar[r]^-{L^2} &\cdots \ar[r]^-{\mathrm{pr}_2 \circ L^{n-1}} &G^{n}&
}
\]
where
\[
\mu_i=\begin{pmatrix}
1 & 0 \\ 
0 & 1
\end{pmatrix}-\begin{pmatrix}
0 & 1 \\ 
0 & 1
\end{pmatrix} =\begin{pmatrix}
1 & -1 \\ 
0 & 0
\end{pmatrix}
\]
if $0 \leq i \leq n-1$ and $\mu_n=\mathrm{pr}_2$.
Hence $\Delta_{\Theta}(\Pi_{n-1})-\mathfrak{p}_n \circ \iota_{F_n}$ is obtained as the composition 
\[
\Delta_{\Theta}(F_n) \xlongrightarrow{\nu_n} G(F_{n-1})[-1] \hookrightarrow \Delta_{\Theta}(F_{n-1}).
\]
This finishes the proof.
\end{proof}

\begin{corollary} \label{zwip}
The map $\mathfrak{p}_{n+1} \colon F_{n+1} \longrightarrow \Delta_{\Theta}(F_n)$ is the equalizer of the two morphisms
\[
\begin{cases}
\Delta_{\Theta}(F_n) \xrightarrow{\Delta_{\Theta}(\Pi_{n})} \Delta_{\Theta}(F_{n-1}) \vspace{2pt}\\
\Delta_{\Theta}(F_n) \xrightarrow{\iota_{F_n}} F_n \xrightarrow{\mathfrak{p}_{n-1}} \Delta_{\Theta}(F_{n-1}).
\end{cases}
\]
\end{corollary}

\subsubsection{Structure theorem}

\begin{notations} \label{tt} Let us fix a pair $(G, \Theta)$ as before.
\begin{enumerate}
\item[(i)] The monomorphism $\mathfrak{j}_n \colon F_{n} \hookrightarrow \Delta_{\Theta}^n$
is
\[
\mathfrak{j}_n \colon F_n \xhookrightarrow{\mathfrak{p}_n} \Delta_{\Theta} \circ F_{n-1} \xhookrightarrow{\Delta_{\Theta}(\mathfrak{p}_{n-1})} \cdots \xhookrightarrow{\Delta_{\Theta}^{n-2}(\mathfrak{p}_{2})} \Delta_{\Theta}^{n-1} \circ F_1 \xhookrightarrow{\Delta_{\Theta}^{n-1}(\mathfrak{p}_{1})} \Delta_{\Theta}^n 
\]
\item[(ii)] The maps $(\pi_{n, i})_{1 \leq i \leq n}$ are the $n$ natural projections from $\Delta_{\Theta}^n$ to $\Delta_{\Theta}^{n-1}$ induced by the map $\Delta_{\Theta} \longrightarrow \mathrm{id}_{\mathrm{C}^{\mathrm{b}}(\mathcal{C})}$. 
\vspace{0.2cm}
\item[(iii)] The functor $\Delta_{\Theta}^{[n]}$ is the equalizer of the $n$ maps $\pi_{n, i}$.
\vspace{0.2cm}
\item[(iv)] The maps $(\pi_{n})_{1 \leq i \leq n}$ are the $n$ natural projections from $\Delta_{\Theta}^{[n]}$ to $\Delta_{\Theta}^{[n-1]}$ given for any $
n$ by any $\pi_{n, i}$.
\end{enumerate}
\end{notations}
\begin{lemma} \label{bof}
For any integer $n$, the map $\mathfrak{j}_n$ factors through the functor $\Delta_{\Theta}^{[n]}$, and \[
(\mathfrak{j}_n)_{n \geq 0} \colon (F_n)_{n \geq 0} \longrightarrow (\Delta_{\Theta}^{[n]})_{n \geq 0}\] is a morphism of projective systems.
\end{lemma}

\begin{proof}
We proceed by induction. The morphism $\mathfrak{j}_{n+1}$ can be written as
\begin{equation}
F_{n+1} \xhookrightarrow{\mathfrak{p}_{n+1}} \Delta_{\Theta} (F_n) \xhookrightarrow{\Delta_{\Theta}(\mathfrak{j}_n)} \Delta_{\Theta}^{n+1}.
\end{equation}
Let us consider the following commutative diagram.
\begin{equation} \label{enfiiin}
\xymatrix@C=40pt{
F_{n+1} \ar@{^{(}->}[r]^-{\mathfrak{p}_{n+1}} & \Delta_{\Theta}(F_n) \ar@{^{(}->}[r]^-{\Delta_{\Theta}(\mathfrak{j}_{n})} \ar[d]^-{\iota_{F_n}}&  \ar[d]^-{\iota_{ \Delta_{\Theta}^n }} \Delta_{\Theta}^{n+1}\\
& F_n \ar@{^{(}->}[r]^-{\mathfrak{j}_{n}} & \Delta_{\Theta}^n}
\end{equation}
Since $\iota_{\Delta_{\Theta}^n}=\pi_{n+1, 1}$ and $\iota_{F_n} \circ \mathfrak{p}_{n+1}=\Pi_{n+1}$, the morphism
\[
F_{n+1} \xhookrightarrow{\mathfrak{j}_{n+1}} \Delta_{\Theta}^{n+1} \xrightarrow{\pi_{n+1, 1}} \Delta_{\Theta}^n
\]
is equal to
\[
F_{n+1} \xrightarrow{\Pi_{n+1}} F_n \xhookrightarrow{\mathfrak{j}_{n}} \Delta_{\Theta}^n.
\]
We have $\iota_{\Delta_{\Theta}^n}=\Delta_{\Theta} (\pi_{n, 1})$, so that the compositions of the up horizontal  and right down arrows of the square in diagram \eqref{enfiiin} is $\Delta_{\Theta}(\pi_{n, 1} \circ \mathfrak{j}_n)$. By induction, for any integer $i$ with $1 \leq i \leq n$, 
\[
\Delta_{\Theta}(\pi_{n, 1} \circ \mathfrak{j}_n)=\Delta_{\Theta}(\pi_{n, i} \circ \mathfrak{j}_n)=\pi_{n+1, i+1} \circ \Delta_{\Theta}(\mathfrak{j}_n), 
\]
so that the morphisms
\[
F_{n+1} \xhookrightarrow{\mathfrak{j}_{n+1}} \Delta_{\Theta}^{n+1} \xrightarrow{\pi_{n+1, i+1}} \Delta_{\Theta}^n
\]
are all equal to
\[
F_{n+1} \xrightarrow{\Pi_{n+1}} F_n \xhookrightarrow{\,\mathfrak{j}_n\,} \Delta_{\Theta}^n.
\]
This finishes the proof.
\end{proof}

\begin{theorem} \label{orange}
The sequence of morphisms
 \[
(\mathfrak{j}_n)_{n \geq 0} \colon (F_n)_{n \geq 0} \longrightarrow (\Delta_{\Theta}^{[n]})_{n \geq 0}
\]
defines an isomorphism of projective systems of dg-endofunctors of $\mathrm{C}^{\mathrm{b}}(\mathcal{C})$.
\end{theorem}

\begin{proof}
We argue by induction. As $\Delta_{\Theta}$ is faithful, the equalizer of the $n$ maps 
\[
\pi_{n+1, i} \colon \Delta_{\Theta}^{n+1} \longrightarrow \Delta_{\Theta}^n \qquad \qquad 2 \leq i \leq n+1
\]
is
\[
\Delta_{\Theta}(\mathfrak{j}_n) \colon \Delta_{\Theta}(F_n) \hookrightarrow \Delta_{\Theta}^{n+1}
\]
If $(\mathfrak{D}, \chi)$ is the equalizer of the $(n+1)$ maps $\pi_{n+1, i}$ for $1 \leq i \leq n$, then $\chi$ factors through $\Delta_{\Theta}(F_{n})$ as shown below
\[
\xymatrix@R30pt@C=30pt{ \mathfrak{D} \, \ar@{^{(}->}[r]^-{\chi} \ar@{^{(}->}[d]^-{\widetilde{\chi}}& \Delta_{\Theta}^{n+1} \\
\Delta_{\Theta}(F_n) \ar@{^{(}->}[ru] &
}
\]
and $(\mathfrak{D}, \widetilde{\chi})$ is the equalizer of the two maps
\[
\begin{cases}
\Delta_{\Theta}(F_n) \hookrightarrow \Delta_{\Theta}^{n+1} \xrightarrow{\pi_{n+1, i}} \Delta_{\Theta}^n \\
\Delta_{\Theta}(F_n)  \hookrightarrow \Delta_{\Theta}^{n+1} \xrightarrow{\pi_{n+1, 1}} \Delta_{\Theta}^n
\end{cases}
\]
where in the first morphism, $i$ is any integer such that $2 \leq i \leq n+1$ (the corresponding morphism doesn't depend on $i$).
These morphisms can be written as
\[
\begin{cases}
\Delta_{\Theta}(F_n) \xrightarrow{\Delta_{\Theta}(\Pi_{n})} \Delta_{\Theta}(F_{n-1}) \xhookrightarrow{\Delta_{\Theta}(\mathfrak{j}_{n-1})} \Delta_{\Theta}^n \\
\Delta_{\Theta}(F_n) \xrightarrow{\iota_{F_n}} F_n \xhookrightarrow{\,\,\mathfrak{p}_n\,\,} \Delta_{\Theta}(F_{n-1}) \xhookrightarrow{\, \, \Delta_{\Theta}(\mathfrak{j}_{n-1})} \Delta_{\Theta}^n
\end{cases}
\]
Since the equalizer remains unchanged after post-composition with a monomorphism, it is isomorphic to $(F_{n+1}, \mathfrak{p}_{n+1})$ thanks to Lemma \ref{zwip}. This completes the induction step.
\end{proof}

\subsubsection{Derived equalizers} \label{modelcat}

Let $H$ be an element of $\mathrm{EndFct}_{\mathrm{dg}}(\mathrm{C}^{\mathrm{b}}(\mathcal{C}))$ endowed with a closed dg morphism $\Psi \colon H \longrightarrow \mathrm{id}_{\mathrm{C}^{\mathrm{b}}(\mathcal{C})}$. For any positive integer $n$, we denote by $(\pi_{n, i})_{1 \leq i \leq n}$ the $n$ natural projections from $H^n$ to $H^{n-1}$ induced by $\Psi$.
\begin{definition} \label{definitif}
If $(H, \Psi)$ is given, let $n$ be a positive integer..
\begin{enumerate}
\item[(i)] The $n^{\mathrm{th}}$ (standard) equalizer of $(H, \Psi)$, denoted by $H^{[n]}$, is the equalizer of the $n$ maps $(\pi_{n, i})_{1 \leq i \leq n}$.
\vspace{0.2cm}
\item[(ii)] If $\widetilde{H}$ denotes the functor $\mathrm{cone}\, \Psi$ and $\Theta \colon \mathrm{id}_{\mathrm{C}^{\mathrm{b}}(\mathcal{C})} \longrightarrow \widetilde{H}$ is the associated morphism, the $n^{\mathrm{th}}$ derived equalizer of $(H, \Psi)$, denoted by $H^{[[n]]}$, is the functor $\Delta_{\Theta}^{[n]}$ (\textit{see} Definition \ref{tt}).
\end{enumerate}
\end{definition}
\begin{remark}
The word ``derived'' in the name ``derived equalizer'' refers to the formalism of Quillen derived functors between model categories. Although not strictly necessary, this approach is explained in Appendix \ref{quillen} and sheds light on many considerations about derived equalizers.
\end{remark}
Let $U$ denote the element of the center of $\mathrm{EndFct}_{\mathrm{dg}}(\mathcal{C})$ defined by 
\[
U=\mathrm{cone}\,\left(  \mathrm{id}_{\mathrm{C}^{\mathrm{b}}(\mathcal{C})} \longrightarrow  \mathrm{id}_{\mathrm{C}^{\mathrm{b}}(\mathcal{C})} \right)\!.
\]
and let $\iota \colon U[-1]  \longrightarrow \mathrm{id}_{\mathrm{C}^{\mathrm{b}}(\mathcal{C})}$ be the corresponding natural morphism. Thanks to Lemma \ref{doublecone}, there is an isomorphism
\begin{equation} \label{fib}
\Delta_{{\Theta}} \simeq H \oplus U [-1]
\end{equation}
such that the diagram
\[
\xymatrix@C=20pt@R=20pt{
\Delta_{\widetilde{H}} \ar[rr]^-{\sim} \ar[rd] && H \oplus U [-1] \ar[ld]^-{\Psi \oplus \,\iota} \\
& \mathrm{id}_{\mathrm{C}^{\mathrm{b}}(\mathcal{C})} &
}
\]
commutes. Hence there is a natural morphism of projective systems
\[
(H^{[n]})_{n \geq 0} \longrightarrow (H^{[[n]]})_{n \geq 0}.
\]
Let us give a few properties:
\begin{proposition} \label{morita2} $ $
\vspace{0.2cm}
\begin{enumerate}
\item[--] If $H$ is an exact dg-endofunctor of $\mathcal{C}^{\mathrm{b}}(\mathcal{C})$ endowed with a morphism from $H$ to $\mathrm{id}_{\mathrm{C}^{\mathrm{b}}(\mathcal{C})}$, then all functors $H^{[[n]]}$ are also exact.
\vspace{0.2cm}
\item[--] Assume that $\mathcal{S}$ is a tensor category, and let $H$ be a lax monoidal dg-endofunctor of $\mathcal{C}^{\mathrm{b}}(\mathcal{S})$ endowed with a multiplicative morphism from $H$ to $\mathrm{id}_{\mathrm{C}^{\mathrm{b}}(\mathcal{S})}$. Then $(H^{[n]})_{n \geq 0}$ and $(H^{[[n]]})_{n \geq 0}$ form a projective system of lax monoidal functors, and the natural  morphisms from $H^{[n]}$ to $H^{[[n]]}$ are multiplicative.
\vspace{0.2cm}
\item[--] Let $H_1$ and $H_2$ be two exact dg-endofunctors of the category $\mathcal{C}^{\mathrm{b}}(\mathcal{C})$
endowed with morphisms $\Psi_i \colon H_i \longrightarrow  \mathrm{id}_{\mathrm{C}^{\mathrm{b}}(\mathcal{C})}$ for $i=1, 2$, and let $\Gamma \colon H_1 \longrightarrow H_2$ be a quasi-isomorphism such that $\Psi_2 \circ \Gamma=\Psi_1$. Then for any positive integer $n$, $\Gamma$ induces quasi-isomorphisms between $H_1^{[[n]]}$ and $H_2^{[[n]]}$.
\end{enumerate}
\end{proposition}

\begin{proof}
The first and third point follow directly by induction using the exact sequence provided in Theorem \ref{hard} (ii). For the second point, the multiplicativity of $H^{[n]}$ is straightforward. It also implies the multiplicativity of $H^{[[n]]}$, as we shall see now. First we remark that for any bounded complexes $K$ and $L$ of elements of $\mathcal{S}$, there is a natural morphism
\[
U(K) \otimes U(L) \longrightarrow U(K \otimes L) [1]
\]
given by the morphism
\[
\xymatrix@C=40pt@R=30pt{
K \otimes L \ar@{=}[d] \ar[r]^-{ \begingroup \tiny {\begin{pmatrix}
1 \\ -1
\end{pmatrix}} \endgroup} & K \otimes L \oplus K \otimes L \ar[r]^-{\tiny{\begin{pmatrix}
1 & 1
\end{pmatrix}} } \ar[d]^-{\tiny{\begin{pmatrix}
1 & 1
\end{pmatrix}} }& K \otimes L \\
K \otimes L \ar[r] & K \otimes L &
}
\]
Now we can endow $H \oplus U [-1]$ with a lax monoidal structure as follows: we define the multiplicative morphism as a matrix of the type
{\renewcommand{\arraystretch}{1.5}
\let\oldbullet\bullet
\renewcommand{\bullet}{\scalebox{0.5}{$\oldbullet$}}% nécessite \usepackage{graphicx}
\[
\begin{blockarray}{ccccc}
&\scriptstyle{H(K) \,\otimes\, H(L)}&\scriptstyle{H(K) \,\otimes\, U(L)[-1]} &\scriptstyle{U(K) \,\otimes \,H(L) [-1]}& \scriptstyle{U(K) \,\otimes\, U(L) [-2]}\\
\begin{block}{r(cccc)}
\scriptstyle{H(K \, \otimes \, L)} &*  & 0  &  0    &  0  \\
\scriptstyle{U(K \,\otimes \,L)} [-1] &0   &*   &  * & *    \\
\end{block}
\end{blockarray}
\]}
whose components are:
\begin{enumerate}
\vspace{0.2cm}
\item[--] the morphism $H(K) \otimes H(L) \longrightarrow H(K \otimes L)$ provided by the lax monoidal structure of $H$,
\vspace{0.2cm}
\item[--] the morphism $H(K) \otimes U(L) \longrightarrow K \otimes U(L)  \xrightarrow{\sim} U(K \otimes L)$,
\vspace{0.2cm}
\item[--] the morphism $U(K) \otimes H(L) \longrightarrow U(K) \otimes L \xrightarrow{\sim} U(K \otimes L)$,
\vspace{0.2cm}
\item[--] the morphism $U(K) \otimes U(L) \longrightarrow  U(K \otimes L)[1]$ formerly introduced.
\vspace{0.2cm}
\end{enumerate}
The unit of $\Delta_{{\Theta}}$ is defined by the composition
\[
\mathbf{1}_{\mathcal{S}} \longrightarrow H(\mathbf{1}_{\mathcal{S}}) \longrightarrow \Delta_{\widetilde{H}}(\mathbf{1}_{\mathcal{S}}).
\]
Hence $H^{[[n]]}$ are also lax monoidal functors, and the compatibility of the multiplicative structures follows from the fact that the natural morphism from $H$ to $\Delta_{{\Theta}}$ is multiplicative.
\end{proof}

\begin{lemma}
Let $
0 \longrightarrow N \longrightarrow A \xlongrightarrow{j} B \xlongrightarrow{p} \mathrm{id}_{\mathrm{C}^{\mathrm{b}}(\mathcal{C})} \longrightarrow 0
$
be an exact sequence in $\mathrm{EndFct}_{\mathrm{dg}}(\mathrm{C}^{\mathrm{b}}(\mathcal{C}))$ and let $H=\mathrm{cone}\,\,(A \longrightarrow B)$ and $Z=\mathrm{cone}\,\,(A \longrightarrow A/N)$. Finally, let $\Psi \colon H \longrightarrow \mathrm{id}_{\mathrm{C}^{\mathrm{b}}(\mathcal{C})}$ be the natural morphism induced by the commutative diagram
\[
\xymatrix{
A \ar[r] \ar[d] & B \ar[d]^-{p} \\
0 \ar[r] &  \mathrm{id}_{\mathrm{C}^{\mathrm{b}}(\mathcal{C})}
}
\]
If $A$ is right exact, then the morphism $H(\Psi)-\Psi_H \colon H^2 \longrightarrow H$ factors through a surjective morphism $\nu \colon H^2 \longrightarrow Z$.
\end{lemma}

\begin{proof}
There is an exact sequence $0 \longrightarrow Z \longrightarrow H \xlongrightarrow{\Psi} \mathrm{id}_{\mathrm{C}^{\mathrm{b}}(\mathcal{C})} \longrightarrow 0$. Since $\Psi \circ (H(\Psi)-\Psi_H)$ vanishes, the image of $H(\Psi)-\Psi_H$ lies in $Z$. To prove that it is the whole of $Z$, we consider the composition
\begin{equation} \label{raki}
H \circ B \longrightarrow H^2 \xlongrightarrow {H(\Psi)-\Psi_H} H
\end{equation}
It can be realized by taking the total complexes of the lines of the commutative diagram
\[
\xymatrix{
A \circ B \ar[r] \ar[d]_-{A(p)} & B^2 \ar[d]^-{B(p)-p_B}\\
A \ar[r]^-{j} & B
}
\]
Since $A$ is right exact, $A(p)$ is surjective. This proves that the image of $B(p)-p_B$ contains $A/N$, so it is equal to $A/N$. As a consequence $Z$ is contained in the image of $\eqref{raki}$, hence in the image of $H(\Psi)-\Psi_H$.
\end{proof}
\begin{lemma} \label{mince}
For any positive integer $n$, the morphisms
\[
\begin{cases}
H^{[n]} \longrightarrow H^{[n-1]} \\
H \circ H^{[n]} \longrightarrow H^2\circ H^{[n-1]} \xrightarrow{\nu_{H^{[n-1]}}} Z  \circ H^{[n-1]}
\end{cases}
\]
are onto. In particular there is a natural exact sequence
\[
0 \longrightarrow H^{[n+1]} \longrightarrow H(H^{[n]}) \longrightarrow Z (H^{[n-1]}) \longrightarrow 0.
\]
\end{lemma}
\begin{proof}
We argue by induction. There is a natural morphism $Z \longrightarrow H$ given by the morphism
\[
\xymatrix@C=25pt@R=15pt{
A \ar[r] \ar[d] &A/N \ar[d] \\
A \ar[r] & B
}
\]
Besides, the composition 
\[
Z \circ H \longrightarrow H^2 \xlongrightarrow{\nu} Z
\]
is induced by the morphism $H \longrightarrow \mathrm{id}_{\mathrm{C}^{\mathrm{b}}(\mathcal{C})}$. Thus the composition
\[
Z \circ H^{[n-1]} \longrightarrow H  \circ H^{[n-1]} \longrightarrow H^2\circ H^{[n-2]} \xrightarrow{\nu_{H^{[n-2]}}} Z  \circ H^{[n-2]}
\]
is induced by the natural morphism $H^{[n-1]} \longrightarrow H^{[n-2]}$ which is onto by induction. Let us now consider the diagram
\[
\xymatrix{
&&Z (H^{[n-1]}) \ar[d] \ar[rd]^-{\alpha} & \\
0 \ar[r] & H^{[n]} \ar[r] \ar[rd]_-{\gamma} & H (H^{[n-1]}) \ar[r]_-{\beta} \ar[d] & Z (H^{[n-2]})  \\ 
&&H^{[n-1]} \ar[d]& \\
&&0&
}
\]
where both horizontal and vertical sequences are exact. As $\alpha$ is onto, so are $\beta$ and $\gamma$. This completes the induction step.
\end{proof}
\begin{theorem} \label{canal}
Assume to be given an exact sequence 
\[
A \longrightarrow B \xlongrightarrow{p} \mathrm{id}_{\mathrm{C}^{\mathrm{b}}(\mathcal{C})} \longrightarrow 0
\]
in $\mathrm{EndFct}_{\mathrm{dg}}(\mathrm{C}^{\mathrm{b}}(\mathcal{C}))$, and let $H=\mathrm{cone}\, \, (A \longrightarrow B)$ and $U=\mathrm{cone}\, \, (\mathrm{id}_{\mathrm{C}^{\mathrm{b}}(\mathcal{C})} \xrightarrow{\mathrm{id}} \mathrm{id}_{\mathrm{C}^{\mathrm{b}}(\mathcal{C})})$. 
\par \smallskip
\begin{enumerate}
\item[(i)] If $A$ is right exact and if $\mathrm{ker}\, \, (A \longrightarrow B)= \{0\}$, for any positive  integer $n$, there is an exact sequence
\[
0 \longrightarrow H^{[n+1]} \longrightarrow H(H^{[n]}) \longrightarrow U \circ A \,(H^{[n-1]}) \longrightarrow 0.
\]
\item[(ii)] 
If $A$ and $B$ are exact, for any nonnegative integer $n$, the functors $H^{[n]}$ and $H^{[[n]]}$ are exact, and the morphism
\[
H^{[n]} \longrightarrow H^{[[n]]}
\]
is an isomorphism of endofunctors of $\mathrm{D}^{\mathrm{b}}(\mathcal{C})$.
\end{enumerate}
\end{theorem}
\begin{proof}
The first point results directly from Lemma \ref{mince} and the fact that $Z=U \circ A$. For the second point, the functors $H^{[[n]]}$ are exact thanks to Proposition \ref{morita2}. Besides, since $Z$ is exact, the exact sequence
\[
0 \longrightarrow H^{[n+1]} \longrightarrow H(H^{[n]}) \longrightarrow Z (H^{[n-1]}) \longrightarrow 0
\]
provided by Lemma \ref{mince} and the nine lemma show by induction that the functors $H^{[n]}$ are also exact. The functor $\widetilde{H}$ is the cone of the complex of functors $A \longrightarrow B \longrightarrow \mathrm{id}_{\mathrm{C}^{\mathrm{b}}(\mathcal{C})}$. Hence there is a morphism $Z \longrightarrow \widetilde{H}[-1]$ given by
\[
\xymatrix@C=25pt@R=15pt{
A \ar[r] \ar[d] & A/N \ar[d]& \\
A \ar[r] & B \ar[r] & \mathrm{id}_{\mathrm{C}^{\mathrm{b}}(\mathcal{C})}
} 
\]
which is an isomorphisms of endofunctors of $\mathrm{D}^{\mathrm{b}}(\mathcal{C})$.
Let us consider the diagram
\[
\xymatrix{
0 \ar[r] & H^{[n+1]} \ar[r] \ar[d] & H (H^{[n]}) \ar[r] \ar[d] &Z (H^{[n-1]}) \ar[r]\ar[d] &0 \\
0 \ar[r] & H^{[[n+1]]} \ar[r] & \Delta_{{\Theta}} (H^{[[n]]}) \ar[r] &\widetilde{H} (H^{[[n-1]]}) [-1] \ar[r]  &0
}
\]
where the first line is exact thanks to Lemma \ref{mince}, and the second line is also exact thanks to Theorem \ref{ecolo}. It implies directly the required result by induction.
\end{proof}
Lastly, we give a more simple situation where standard and derived equalizers are quasi-isomorphic:

\begin{remark} \label{since}
Let $L$ be an element of $\mathrm{EndFct}_{\mathrm{dg}}(\mathrm{C}^{\mathrm{b}}(\mathcal{C}))$, and define a pair $(H, \Psi)$ as follows: $H=L \oplus \mathrm{id}_{\mathrm{C}^{\mathrm{b}}}(\mathcal{C})$ and $\Psi$ is the second projection. Then for any nonnegative integer $n$, the map from $H^{[n]}$ to $H^{[[n]]}$ is a quasi-isomorphism, and $H^{[n]}$ is isomorphic to the functor $\bigoplus_{p=0}^n L^p$.
\end{remark}

\section{Square zero extensions} \label{carrenul}

\subsection{General properties} \label{nakon}

\subsubsection{Setting}
Let $A$ be a commutative algebra over a field $\mathbf{k}$ of characteristic zero. If no ring is specified, tensor product will always be taken over $A$.
\par \medskip
Let $I$ be a free $A$-module of finite rank, and let $B$ be a $\mathbf{k}$-square-zero extension of $A$ by $I$, i.e. we have an exact sequence
\begin{equation} \label{ati}
0 \longrightarrow I \longrightarrow B \xlongrightarrow{\pi} A \longrightarrow 0
\end{equation}
of $k$-algebras, where $I^2=0$. We will always assume that $B$ is trivial (as a $\mathbf{k}$-extension), which means that \eqref{ati}
 splits. Hence $B$ is isomorphic to the trivial extension $I \oplus A$ as a $\mathbf{k}$-vector space, the ring structure being given by
\[
(i, a) . (i', a')=(ia'+ai', aa').
\]
Splittings of the sequence \eqref{ati} form an affine space over the module $\mathrm{Der}\,(A, I)$ of $I$-values derivations of $A$. In the majority of the results that will follow, we fix a splitting $\sigma$ of \eqref{ati}.
\par \medskip
Modules over $B$ admit a simple description, that we give now. Let $M$ and $N$ be two $A$-modules, and fix a splitting of \eqref{ati}.
\begin{enumerate}
\vspace{0.2cm}
\item[(i)] Any extension $V$ in $\mathrm{Ext}^1_B(M, N)$ yields an $A$-linear multiplication map
\[
\mu_V \colon I \otimes M \longrightarrow N
\]
given by $\mu_{{V}}(i\otimes m)=i{v}$ where ${v}$ is any lift of $m$ in ${V}$. This definition is meaningful since two different lifts lie in $N$, which is annihilated by $I$. Besides $V$ defines an extension class in $\mathrm{Ext}^1_A(M, N)$ because via the splitting we took, every $B$-module becomes an $A$-module.
\vspace{0.2cm}
\item[(ii)] If $Z$ and $\mu$ are in $\mathrm{Ext}^1_A(M, N)$ and $\mathrm{Hom}_A(I \otimes M, N)$ respectively, there is an associated extension $Z_{\mu}$ in $\mathrm{Ext}^1_B(M, N)$ defined as follows: as an $A$-module $Z_{\mu}=Z$, and the action of $I$ is given by the composition
\[
I \otimes Z \twoheadrightarrow I \otimes M \xlongrightarrow{\mu} N \hookrightarrow Z.
\]
\end{enumerate}
\begin{lemma} \label{easy}
The map
\[
\mathrm{Ext}^1_B(M, N) \xrightarrow{\sim} \mathrm{Ext}^1_A(M, N) \oplus \mathrm{Hom}_A(I \otimes M, N)
\]
where the second component is $V \longrightarrow \mu_V$, is a group isomorphism. Its inverse is given by $(Z, \mu) \longrightarrow Z_{\mu}$. Besides, it is independent of the chosen splitting  of \eqref{ati}.
\end{lemma}
The proof is straightforward, we leave it to the reader. 
\begin{corollary} A module over $B$ is given by two $A$-modules $M$ and $N$, an extension class in $\mathrm{Ext}^1_A(M, N)$, and a surjective $A$-linear morphism from $I \otimes M$ to $N$.
\end{corollary}

\begin{proof}
Let $V$ be a $B$-module and put $M=V \otimes_B A$ and $N=IV$. There is an exact sequence of $B$-modules
\[
0 \longrightarrow N \longrightarrow U \longrightarrow M \longrightarrow 0
\]
which allows to see $V$ as a class in $\mathrm{Ext}^1_B(M, N)$: hence it defines a class in $\mathrm{Ext}^1_A(M, N)$ as well as a morphism $\mu_V \colon I \otimes M \longrightarrow N$ which is surjective by definition of $M$ and $N$. 
\par \medskip
Conversely, given two $A$-modules $M$ and $N$, an extension class in $\mathrm{Ext}^1_A(M, N)$ and a surjective map $\mu \colon I \otimes M \longrightarrow N$, we can consider the associated element $Z_{\mu}$ in $\mathrm{Ext}^1_B(M, N)$. To prove that both constructions are mutually inverse, we must prove that there is an isomorphism of extensions
\[
\xymatrix{
0 \ar[r] & IZ_{\mu} \ar[r] \ar[d]^-{\sim} & Z_{\mu} \ar[r] \ar@{=}[d] & Z_{\mu} \otimes_B A \ar[d]^-{\sim} \ar[r] & 0 \\
0 \ar[r] & N \ar[r] & Z_{\mu} \ar[r] & M \ar[r] & 0 
}
\]
There is always a natural morphism between this extension, since $IZ_{\mu}$ is a $A$-submodule of $N$. The surjectivity of $\mu$ ensures that $IZ_{\mu}=N$, which yields the desired result.
\end{proof}

Lastly, let us present two useful base change operations for $B$-modules.
Let $V$ be a $B$-module, and put $M=V \otimes_B A$ and $N=IV$. Assume to be given an $A$-module $Q$ (resp. $R$) and a surjective $A$-linear morphism $u: Q \longrightarrow M$ (resp. $v \colon N \longrightarrow R$). Define the $B$-module $V^{'}$ (resp. $V^{''}$) as follows: $V^{'}$ is given by the cartesian (resp. cocartesian) diagram 
\[
\xymatrix{
V^{'} \ar[r] \ar[d] & Q \ar[d] \\
V\ar[r] & M
}
\quad \textrm{resp.} \quad
\xymatrix{
N \ar[r] \ar[d] & V \ar[d] \\
R \ar[r] &{V}^{''}
}
\]
\begin{lemma}  \label{reidemaster}
There are isomorphisms
\[
\begin{cases}
{V^{'}}\otimes_B A \simeq Q, \, IV^{'} \simeq N \\
 \textrm{resp.}\,\, {V^{''}} \otimes_B A \simeq M, \,IV^{''} \simeq R
\end{cases}
\] 
such that the composition
\[
\begin{cases}
Q \simeq {V^{'}} \otimes_B A\longrightarrow M \\
 \textrm{resp.}\,\, N \longrightarrow IV^{''}\simeq R
\end{cases}
\]
is equal to $u$ (resp. $v$). Besides, the multiplication morphisms of $V^{'}$ (resp. $V^{''}$) is given by the composition
\[
\begin{cases}
\mu_{V^{'}} \colon I \otimes (V^{'} \otimes_B A)\simeq I \otimes Q  \longrightarrow I \otimes M \xlongrightarrow{\mu_V} N \simeq IV^{'} \\
 \textrm{resp.} \,\, \mu_{V^{''}} \colon I \otimes (V^{'} \otimes_B A) \simeq I \otimes M \xlongrightarrow{\mu_V} N \longrightarrow R \simeq IV^{''}
\end{cases}
\]
\end{lemma}
\begin{proof}
Left to the reader.
\end{proof}

\subsubsection{The functors $\mathrm{Tor}^p_B(\,*\,, A)$}

\begin{lemma}  Let $M$ be an $A$-module. Then $\mathrm{Tor}^1_B(M, A)$ is canonically isomorphic to $I  \otimes M$ as a $B$-module.
\end{lemma}
\begin{proof}
The sequence \eqref{ati} gives the long exact sequence
\[
0 \longrightarrow  \mathrm{Tor}^1_B(M, A) \longrightarrow M \otimes I \longrightarrow M \xrightarrow{\,\,\mathrm{id}\,\,} M  \longrightarrow 0.
\]
\end{proof}

\begin{proposition} \label{compris2}
Let $V$ be in $\mathrm{Ext}^1_B(M, N)$. Then the connection morphism 
\[
I \otimes M \simeq \mathrm{Tor}^1_B(M, A) \longrightarrow \mathrm{Tor}^0_B(N, A) \simeq N
\]
is exactly $\mu_V$, and there is an exact sequence
\[
0 \longrightarrow \mathrm{Tor}^1_B(V, A) \longrightarrow I \otimes M \xlongrightarrow{\mu_V} N
\]
\end{proposition}
\begin{proof}
Let $S$ be the kernel of the natural morphism from $\sigma^*V$\footnote{Here we make a slight abuse of notation, because $\sigma^* V$ should be $\sigma^* \sigma_* V$: we consider the $B$-module $V$ as an $A$-module.} to $V$ induced by the identity of $V$ via the isomorphism 
\[
\mathrm{Hom}_A(V, V) \simeq \mathrm{Hom}_B(\sigma^*V, V).
\]
We consider the diagram.
\[
\xymatrix{
&0 \ar[d]&0 \ar[d]&0 \ar[d] &\\
0 \ar[r] &I \otimes N \ar[d] \ar[r]& S  \ar[d] \ar[r] &I \otimes M \ar[d] \ar[r] &0\\
0 \ar[r]& \sigma^* N \ar[r] \ar[d]& \sigma^*V \ar[d] \ar[r] & \sigma^* M \ar[d] \ar[r] &0\\
0 \ar[r]& N \ar[r] \ar[d] & V \ar[r] \ar[d]& M \ar[r] \ar[d]& 0 \\
      &   0   &   0    & 0   & 
}
\]
For any $A$-module $P$ and any positive integer $i$ the module $\mathrm{Tor}^i_B(\sigma^*P, A)$ vanishes. Hence the Tor exact sequence 
\[
\mathrm{Tor}^1_B(V, A) \longrightarrow \mathrm{Tor}^1_B(M, A) \longrightarrow \mathrm{Tor}^0_B(N, A)
\] 
can be identified with the exact sequence
\[
\mathrm{ker}\, (S \otimes_B A \longrightarrow V)\longrightarrow I \otimes M \longrightarrow N
\]
obtained by the snake lemma. By diagram chase, we get the first point of the proposition. Now
\[
S=\{(i \otimes v, v') \in \sigma^* V \, \, \textrm{such that}\, \, iv+v'=0\}
\]
so that 
\[
IS=\{(i \otimes v, 0) \in \sigma^* V \, \, \textrm{such that}\, \, v \in IV\}
\]
and we get
\[
S \otimes_B A=\{(i \otimes m, v') \in I \otimes M \oplus V \, \, \textrm{such that}\, \, \mu_V (i \otimes m)+v'=0\}.
\]
Hence $\mathrm{ker}\, (S \otimes_B A \longrightarrow V)$ is the kernel of $\mu_V$, and embeds in $I \otimes M$. 
\end{proof}

\begin{corollary} \label{dur}
Let $V$  be a $B$-module, and let $M=V \otimes_B A$. Then the map
\[
\mathrm{Tor}^1_B(V, A) \longrightarrow \mathrm{Tor}^1_B(M, A) \simeq I \otimes M
\]
is injective, and its image is $\mathrm{ker}\,\mu_V$.
\end{corollary}

\begin{proof}
The $B$-module $V$ defines a canonical class in $\mathrm{Ext}^1_B(M, IV)$. Hence the result follows from Proposition \ref{compris2}.
\end{proof}

\begin{proposition} \label{Tor2}
Let $V$  be a $B$-module, and let $M=V \otimes_B A$. Then for every positive integer $p$, there is a functorial isomorphism
\[
\mathrm{Tor}^p_B(V, A) \simeq I^{\otimes (p-1)} \otimes \mathrm{Tor}^1_B(V, A).
\]
\end{proposition}

\begin{proof}
We prove the result by induction on $p$. Let us consider the module $S$ introduced in the proof of Proposition \ref{compris2}. We have an exact sequence
\begin{equation} \label{s}
0 \longrightarrow S \longrightarrow \sigma^* V \longrightarrow V \longrightarrow 0
\end{equation}
so that $\mathrm{Tor}^{p+1}_B(V, A)$ is isomorphic to $\mathrm{Tor}^p_B(S, A)$. Recall from the proof of Proposition \ref{compris2} that there is an isomorphism between $I \otimes M$ and $S \otimes_B A$ given by 
\[
i \otimes m \mapsto (i \otimes m, -\mu_V(i \otimes m)). 
\]
The multiplication map
\[
\mu_S \colon I \otimes (S \otimes_B A) \longrightarrow IS 
\] 
is given via this isomorphism by
\[
i' \otimes (i \otimes m) \longrightarrow (-i' \otimes \mu_V(i \otimes m), 0).
\]
Hence 
\[
\mathrm{Tor}^1_B(S, A) \simeq \mathrm{ker}\, \mu_S \simeq I \otimes \mathrm{ker} \mu_V \simeq I \otimes \mathrm{Tor}^1_B(V, A)
\]
and
\begin{align*}
\mathrm{Tor}^{p+1}_B(V, A) \simeq \mathrm{Tor}^p_B(S, A)& \simeq I^{\otimes (p-1)} \otimes \mathrm{Tor}^1_B(S, A)\\
& \simeq I^{\otimes p} \otimes \mathrm{Tor}^1_B(V, A). 
\end{align*}
\end{proof}

\subsubsection{Principal parts}
As $I^2=0$, the $B$-module $\Omega^1_B$ fits into a natural (split) exact sequence
\[
0 \longrightarrow I \longrightarrow \Omega^1_B \otimes_B A \xlongrightarrow{p} \Omega^1_A \longrightarrow 0
\] 
which is the conormal sequence associated with the map $B \longrightarrow A$. We put $E=\Omega^1_B \otimes_B A$, $E$ is canonically isomorphic to $I \oplus \Omega^1_A$ after a choice of a splitting $\sigma$ of \eqref{ati}. 
\par \medskip
Recall that for any module $V$ over a commutative $\mathbf{k}$-algebra $R$, the module of principal parts $\mathrm{P}^1_R(V)$ is the $R$-module defined (as a $\mathbf{k}$-vector space) by
\[
\mathrm{P}^1_R(V)=(\Omega^1_R \otimes_R V) \oplus V
\]
where $R$ acts by 
\[
r (\omega \otimes v, v')=(r \omega \otimes v+dr \otimes v' , rv'). 
\]
Hence there is an exact sequence
\[
0 \longrightarrow \Omega^1_R \otimes_R V \longrightarrow \mathrm{P}^1_R(V) \longrightarrow  V \longrightarrow 0
\]
of $R$-modules that splits over $\mathbf{k}$ (but not always over $R$). The main result we prove is:
\begin{theorem} \label{wazomba}
Let $V$ be a $B$-module and let $M=V \otimes_B A$. Then the map
\[
\mathrm{Tor}^1_B(\mathrm{P}^1_B(V), A) \longrightarrow \mathrm{Tor}^1_B(V, A)
\]
vanishes. More precisely, the connection morphism
\[
\mathrm{ker} \, \mu_V \simeq\mathrm{Tor}^1_B (V, A) \longrightarrow \mathrm{Tor}^0_B(\Omega^1_B \otimes_B V, A) \simeq E \otimes M
\]
is obtained by the chain of inclusions
\[
\mathrm{ker} \, \mu_V \hookrightarrow I \otimes M \hookrightarrow E \otimes M.
\]
\end{theorem}

\begin{proof}
It enough to prove the second statement of the theorem. We have a commutative diagram
\[
\xymatrix{
\mathrm{Tor}^1_B(V, A) \ar[r] \ar[d]& \mathrm{Tor}^0_B(\Omega^1_B \otimes_B V , A) \ar[d]^-{\wr} \\
\mathrm{Tor}^1_B (M, A) \ar[r] &  \mathrm{Tor}^0_B(\Omega^1_B \otimes_B M , A)
}
\]
Hence, by Corollary \ref{dur}, it suffices to prove that the connecting homomorphism 
\[
\mathrm{Tor}^1_B (M, A) \longrightarrow \mathrm{Tor}^0_B(\Omega^1_B \otimes_B M , A)
\] 
associated with the short exact sequence 
\[
0 \longrightarrow \Omega^1_B \otimes_B M \longrightarrow \mathrm{P}^1_B(M) \longrightarrow M \longrightarrow 0
\]
is naturally identified with the inclusion $I \otimes M \hookrightarrow E \otimes M$.
Via the trivialisation given by $\sigma$, the $B$-module $\mathrm{P}^1_B(M)$ is isomorphic (as a $\mathbf{k}$-vector space) to
\[
I \otimes M \oplus \Omega^1_A\otimes M \oplus M,
\]
and $B$ acts by the formula
\[
(i, a) . (i' \otimes m, \omega \otimes m', m'')=(ai' \otimes m + i \otimes m'', a \omega \otimes m' + da \otimes m'', a m'').
\]
Hence there are two natural exact sequences
\[
\begin{cases}
0 \longrightarrow \Omega^1_A \otimes M \longrightarrow \mathrm{P}^1_B(M) \longrightarrow \sigma^* M \longrightarrow 0 \\
0 \longrightarrow I \otimes M \longrightarrow \mathrm{P}^1_B(M) \longrightarrow \mathrm{P}^1_A(M) \longrightarrow 0
\end{cases}
\]
This gives a commutative diagram
\[
\xymatrix{
0 \ar[r]& \Omega^1_A \otimes M  \ar[r]  & \mathrm{P}^1_A(M) \ar[r]  & M \ar[r] & 0 \\
0 \ar[r]& E \otimes M  \ar[r] \ar[d] \ar[u]& \mathrm{P}^1_B(M) \ar[r] \ar[d] \ar[u]& M \ar[r] \ar@{=}[d] \ar@{=}[u]& 0 \\
0 \ar[r]&  I \otimes M \ar[r] & \sigma^* M \ar[r] & M \ar[r] &0
}
\]
This yields a commutative diagram of connecting morphisms
\[
\xymatrix{
\mathrm{Tor}^1_B(M, A) \ar[r]& \mathrm{Tor}^0_B(\Omega^1_A \otimes M, A)\\
\mathrm{Tor}^1_B(M, A) \ar@{=}[u] \ar@{=}[d] \ar[r]& \mathrm{Tor}^0_B(E \otimes M, A) \ar[u] \ar[d]\\
\mathrm{Tor}^1_B(M, A) \ar[r]&  \mathrm{Tor}^0_B(I \otimes M, A)
}
\]
and since the top connecting morphism is zero we get
\[
\xymatrix{
& \mathrm{Tor}^0_B(\Omega^1_A \otimes M, A)\\
\mathrm{Tor}^1_B(M, A) \ar[ru]^-{0} \ar@{=}[d] \ar[r]& \mathrm{Tor}^0_B(E \otimes M, A) \ar[u] \ar[d]\\
I \otimes M \ar@{=}[r] & I \otimes M
}
\]
This finishes the proof.
\end{proof}
As a corollary of this result, there is a natural exact sequence
\begin{equation} \label{boulon}
0 \longrightarrow \mathrm{Tor}^1_B(V, A)  \longrightarrow E \otimes M \longrightarrow \mathrm{P}^1_B(V) \otimes_B A \longrightarrow M \longrightarrow 0.	
\end{equation}
\begin{definition} \label{residual}
The \textit{residual Atiyah morphism} of a $B$-module $V$ is the morphism 
\[
\chi_{V} \colon V \otimes_B A \longrightarrow IV[1]
\] 
in $\mathrm{D}^{\mathrm{b}}(B)$ attached to the exact sequence of $B$-modules
\[
0 \longrightarrow IV \longrightarrow V \longrightarrow V \otimes_B A \longrightarrow 0.
\]
\end{definition}
\par \bigskip
Let us now fix an $A$-module $M$. The principal parts exact sequence
\[
0 \longrightarrow E \otimes M \longrightarrow \mathrm{P}^1_B(M) \longrightarrow M \longrightarrow 0
\]
defines a morphism
\[
\mathrm{at}_B(M) \colon M \longrightarrow E \otimes M [1]
\]
in $\mathrm{D}^{\mathrm{b}}(B)$, which is the Atiyah class of $M$ over $B$. 
\begin{proposition} \label{snort}
For any $A$-module $M$ and any splitting $\sigma$ of \eqref{atiyah}, the morphism
\[
\mathrm{at}_B(M) \colon M \longrightarrow E \otimes M [1] \simeq I \otimes M[1] \oplus \Omega^1_A \otimes M[1]
\]
is the couple $\{ \chi_{\sigma^*M}, \mathrm{at}_A(M)\}$.
\end{proposition}

\begin{proof}
This follows directly from the diagram
\[
\xymatrix{ 0 \ar[r] & I \otimes M \ar[r] & \sigma^* M \ar[r] & M \ar[r] &0 \\
0 \ar[r] & E \otimes M \ar[r] \ar[u] \ar[d]& \mathrm{P}^1_B(M) \ar[r] \ar[u] \ar[d]& M \ar[r] \ar[u] \ar[d]&0 \\
0 \ar[r] &  \Omega^1_A \ar[r] & \mathrm{P}^1_A(M) \ar[r] & M \ar[r] &0
}
\]
appearing in the proof of Theorem \ref{wazomba}.
\end{proof}

\subsection{Local obstruction theory} \label{oups}

\subsubsection{Admissible complexes} \label{lauvitel}
In this section, we use the homological grading convention for complexes in order to avoid negative indices. All complexes will be concentrated in nonnegative homological degrees. For any complex of $B$-modules $K_{\bullet}$, we denote by $\overline{K}_{\bullet}$ the complex $K_{\bullet} \otimes_B A$.
\begin{definition} \label{meije}
Let $n$ be in $\mathbb{N} \cup \{ \infty\}$. 
\begin{enumerate}
\vspace{0.2cm}
\item[--] A complex $K_{\bullet}$ of $B$-modules is $n$-\textit{admissible} if for any $i$ such that $0 \leq i \leq n-1$, the $A$-module $\mathrm{H}_i(\mathrm{Tor}^1_B(K_{\bullet}, A))$ vanishes. 
\vspace{0.2cm}
\item[--] For $n=+ \infty$, we simply say that $K_{\bullet}$ is admissible (instead of ``$\infty$-admissible''). 
\vspace{0.2cm}
\item[--]  We say that that a $B$-module ${K}$ is admissible if it admissible as a complex concentrated in degree $0$, that is if $\mathrm{Tor}^1_B(K, A)$ vanishes. 
\end{enumerate}
\end{definition}
Let us denote by $\mathbf{Tor}^p_B(*, A)$ the hypertor functors defined by the usual formula 
\[
\mathbf{Tor}^p_B(K_{\bullet}, A)=\mathrm{H}_p( K_{\bullet} \,\lltens{}_B \, A).
\]
\begin{proposition} \label{cns}
Let $K_{\bullet}$ be a complex of $B$-modules and $n$ be in $\mathbb{N} \cup \{+ \infty \}$. Then the complex $K_{\bullet}$ is $n$-admissible if and only if the natural map 
\[
\mathbf{Tor}^i_B(K_{\bullet}, A) \longrightarrow \mathrm{H}_{i}(K_{\bullet} \otimes_B A)
\] 
is an isomorphism for $0 \leq i \leq n$ and surjective for $i=n+1$. 
\end{proposition}

\begin{proof}
By \cite[Application 5.7.8]{Weibel}, there is a filtration 
\[
\mathrm{F}_0 \mathbf{Tor}^{i}_{B}(K_{\bullet}, A) \subseteq \mathrm{F}_1 \mathbf{Tor}^{i}_{B}(K_{\bullet}, A) \subseteq 
\cdots \subseteq 
\mathrm{F}_i \mathbf{Tor}^{i}_{B}(K_{\bullet}, A)
\]
and a spectral sequence of homological type such that 
\[
\begin{cases}
\mathrm{E}^2_{p, q}=\mathrm{H}_p(\mathrm{Tor}^q_B(K_{\bullet}, A)) \\
\mathrm{E}^{\infty}_{p, q}=\mathrm{Gr}_p \mathbf{Tor}^{p+q}_B(K_{\bullet}, A)
\end{cases}
\]
The map from $\mathbf{Tor}^i_B(K_{\bullet}, A)$ to $\mathrm{H}_{i}(K_{\bullet} \otimes_B A)$ is the composition
\[
\mathbf{Tor}^i_{B}(K_{\bullet}, A) \twoheadrightarrow \mathrm{Gr}_i \mathbf{Tor}^i_{B}(K_{\bullet}, A) \simeq \mathrm{E}^{\infty}_{i, 0} \hookrightarrow \mathrm{E}^2_{i, 0}.
\]
If the complex $K_{\bullet}$ is $n$-admissible, then $\mathrm{E}^2_{p, 1}$ vanish for $0 \leq p \leq n-1$. Thanks to Proposition \ref{Tor2}, $\mathrm{E}^2_{p, q}$ vanishes as well for $0 \leq p \leq n-1$ and $q \geq 1$. Hence $\mathrm{F}_j \mathbf{Tor}^i_B(K_{\bullet}, A)$ vanishes for $0 \leq j<i \leq n$. For any $r \geq 2$ and $0 \leq p \leq n+1$, the maps $\mathrm{d}^r_{p, 0} \colon \mathrm{E}^r_{p, 0} \longrightarrow \mathrm{E}^r_{p-r, r-1}$ vanishes, and no differential $\mathrm{d}^r$ abuts to $\mathrm{E}^r_{p, 0}$, so that $\mathrm{E}^2_{p, 0} \simeq \mathrm{E}^{\infty}_{p, 0}$. This proves that the map 
\[
\mathbf{Tor}^i_B(K_{\bullet}, A) \longrightarrow \mathrm{H}_{i}(K_{\bullet} \otimes_B A)
\]
is an isomorphism for $0 \leq i \leq n$, and surjective for $i=n+1$. Conversely, assume that all the maps 
\[
\mathbf{Tor}^i_B(K_{\bullet}, A) \longrightarrow \mathrm{H}_{i}(K_{\bullet} \otimes_B A)
\] 
are isomorphisms if $0 \leq i \leq n$, and are surjective for $i=n+1$. Then 
$\mathrm{F}_j \mathbf{Tor}^i_B(K_{\bullet}, A)$ vanishes for $0 \leq j <i \leq n$, and 
$\mathrm{d}^2_{i, 0}$ vanishes for $0 \leq i \leq n+1$. 
\par \medskip
The first point implies that $\mathrm{E}^{\infty}_{p, q}$ vanishes as soon as $q \geq 1$ and $p+q \leq n$. Let us now prove by induction on $k$ that $\mathrm{E}^r_{k, 1}$ vanishes if $r \geq 2$ and $k \leq n-1$. The module $\mathrm{E}^{r+1}_{k, 1}$ is the middle cohomology of the complex
\[
\mathrm{E}^r_{k+r, 2-r} \xlongrightarrow{\mathrm{d}^r_{k+r, 2-r}} \mathrm{E}^r_{k, 1} \xlongrightarrow{\mathrm{d}^r_{k,1}} \mathrm{E}^r_{k-r, r}
\]
Thanks to Proposition \eqref{Tor2}, $\mathrm{E}^2_{k-r, r}$ is isomorphic to $ I^{\otimes r} \otimes\mathrm{E}^2_{k-r, 1}$, so it vanishes by induction. Hence $\mathrm{E}^m_{k-r, r}$ vanishes for any $m \geq 2$ and so in particular $\mathrm{E}^r_{k-r, r}$ vanishes. Besides, $ \mathrm{E}^r_{k+r, 2-r}$ is always zero if $r \geq 3$, and if $r=2$ the differential $\mathrm{d}^2_{k+2, 0}$ vanishes. 
Hence $\mathrm{E}^{r+1}_{k, 1} \simeq \mathrm{E}^{r}_{k, 1}$, and since $\mathrm{E}^{\infty}_{k, 1}$ vanishes, all terms $\mathrm{E}^{r}_{k, 1}$ vanish as well. 
\end{proof}

Given any bounded complex of $B$-modules and any nonnegative integer $n$, there is a canonical procedure that allows to produce $n$-admissible complexes isomorphic to the initial one in $\mathrm{D}^{\mathrm{b}}(B)$. 
\begin{definition} \label{mu}
The functor $\mu$ is the element of $\mathrm{EndFct}_{\mathrm{dg}}(\mathrm{C}^{\mathrm{b}}(B))$ defined by the formula
\[
\mu(K_{\bullet})=\mathrm{cone}\, \, (\Omega_B^1 \otimes_B K_{\bullet} \longrightarrow \mathrm{P}^1_B(K_{\bullet})).
\] 
\end{definition}
The natural morphism from $\mu$ to $\mathrm{id}_{\mathrm{C}^{\mathrm{b}}(B)}$ is a quasi-isomorphism. We write $\widetilde{\mu}$ for the functor $\mathrm{cone} (\mu \longrightarrow \mathrm{id}_{\mathrm{C}^{\mathrm{b}}(B)})$.
\begin{proposition} \label{coince}
Let $K_{\bullet}$ be a bounded complex of $B$-modules. Then there are natural isomorphisms
\[
\begin{cases}
\mathrm{Tor}^1_B(\mu(K_{\bullet}),A) \simeq I \otimes \mathrm{Tor}^1_B(K_{\bullet}, A) [1] \\
\mathrm{Tor}^1_B(\widetilde{\mu}(K_{\bullet}),A) \simeq I \otimes \mathrm{Tor}^1_B(K_{\bullet}, A) [2] \oplus \mathrm{Tor}^1_B(K_{\bullet}, A)
\end{cases}
\]
in $\mathrm{D}^{\mathrm{b}}(B)$.
\end{proposition}

\begin{proof}
We start by noticing that
\[
\mathrm{Tor}^1_{B}({\mu ({}{K}}_{\bullet}), {A})\!=\!\mathrm{cone} \, \{\mathrm{Tor}^1_{B}(\Omega^1_B\otimes_{B} {{}{{{}{K}}}}_{\bullet}, {A}) \!\longrightarrow \!\mathrm{Tor}^1_{B}(\mathrm{P}^1_{B}({{}{{{}{K}}}}_{\bullet}), {A})\}.
\]
By Theorem \ref{wazomba}, there is an exact sequence
\[
0 \!\longrightarrow\!\mathrm{Tor}^2_{B}({{}{{{}{K}}}}_{\bullet}, {A}) \!\longrightarrow\!\mathrm{Tor}^1_{B} (\Omega^1_B\otimes_{{}{O}_B} {{}{{{}{K}}}}_{\bullet}, {A}) \!\longrightarrow\!\mathrm{Tor}^1_{B}(\mathrm{P}^1_{B}({{}{{{}{K}}}}_{\bullet}), {A}) \!\longrightarrow 0
\]
and thanks to Proposition \ref{Tor2}, 
\[
\mathrm{Tor}^2_B({{}{{{}{K}}}}_{\bullet}, {A}) \simeq I \otimes\mathrm{Tor}^1_{B}({{}{{{}{K}}}}_{\bullet}, {A}).
\]
This gives the first isomorphism. The second one is proven using the same method.
\end{proof}

\begin{corollary} \label{stair}
If ${}{K}_{\bullet}$ is $n$-admissible, then $\mu({}{K}_{\bullet})$ is $(n+1)$-admissible.
\end{corollary}
\begin{proof}
If $i \leq n+1$, we have 
\[
\mathrm{H}_i(\mathrm{Tor}^1_B(\mu({}{K}_{\bullet}), A)) \simeq \mathrm{H}_{i-1}({I} \otimes \mathrm{Tor}^1_B({}{K}_{\bullet}, A))=\{0\}.
\]
\end{proof}

\begin{theorem} \label{loire}
Let $n$ be a nonnegative integer and ${}{K}_{\bullet}$ be an $n$-admissible complex. Then for any positive integer $p$, the natural morphism from $\mu^{[p]}({}{K}_{\bullet})$ to $K_{\bullet}$ is a quasi-isomorphism, and $\mu^{[p]}({}{K}_{\bullet})$ is $n+p$-admissible.
\end{theorem}
\begin{proof}
Thanks to Theorem \ref{canal} (i), there is an exact sequence
\[
0 \longrightarrow \mu^{[p+1]}({}{K}_{\bullet}) \longrightarrow \mu(\mu^{[p]}({}{K}_{\bullet})) \longrightarrow U(\Omega^1_B \otimes_B \mu^{[p-1]}({}{K}_{\bullet})) \longrightarrow 0.
\]
where $U(M_{\bullet})=\mathrm{cone}(M_{\bullet} \longrightarrow M_{\bullet})$ for any $M_{\bullet}$ in  $\mathrm{C}^{\mathrm{b}}(B)$.
Hence, thanks to Proposition \ref{coince},  
\[
\mathrm{Tor}^1_B(\mu^{[p+1]}({}{K}_{\bullet}), A) \simeq {I} \otimes \mathrm{Tor}^1_B(\mu^{[p]}({}{K}_{\bullet}), A) [1].
\]
This gives the result.
\end{proof}

\subsubsection{The category $\mathrm{D}^{\mathrm{adm}}(B)$} \label{ex}
Given a bounded complex $K_{\bullet}$ of $B$-modules, it is interesting to know if $K_{\bullet}$ can be reconstructed from the two complexes $\overline{K}_{\bullet}=K_{\bullet} \otimes_B A$ and $\mathrm{Tor}^1_B(K_{\bullet}, A)$. At the level of complexes, the answer is given by Lemma \ref{easy}: $K_{\bullet}$ is entirely determined by $K_{\bullet}$, the submodule $\mathrm{Tor}^1_B(K_{\bullet}, A)$ of $I \otimes \overline{K}_{\bullet}$, and the extension class of the exact sequence
\begin{equation} \label{oubli}
0 \longrightarrow I \otimes \overline{K}_{\bullet}/\mathrm{Tor}^1_B(K_{\bullet}, A) \longrightarrow K_{\bullet} \longrightarrow \overline{K}_{\bullet} \longrightarrow 0
\end{equation}
in $\mathrm{D}^{\mathrm{b}}(A)$.
We can address the same problem in the derived setting: assume to be given a quadruplet  $(M_{\bullet}$, $N_{\bullet}, \mu, \delta)$ where:
\begin{enumerate}
\vspace{0.2cm}
\item[--] $M_{\bullet}, N_{\bullet}$ are in $\mathrm{D}^{\mathrm{b}}(A)$,
\vspace{0.2cm}
\item[--] $\mu$ is in $\mathrm{Hom}_{\mathrm{D}^{\mathrm{b}}(A)} (I \otimes M_{\bullet}, N_{\bullet})$,
\vspace{0.2cm}
\item[--] $\delta$ is in $\mathrm{Hom}_{\mathrm{D}^{\mathrm{b}}(A)} (M_{\bullet}, N_{\bullet}[1])$.
\vspace{0.2cm}
\end{enumerate}
We look for elements $K_{\bullet}$ in $\mathrm{C}^{-}(B)$ such that $\overline{K}_{\bullet}$ and $IK_{\bullet}$ are isomorphic in $\mathrm{D}^{\mathrm{b}}(A)$ to $M_{\bullet}$ and $N_{\bullet}$ respectively, and via this isomorphisms $\mu$ is the multiplication class $\mu_{K_{\bullet}}$, and $\delta$ is the extension class of \eqref{oubli}. Besides, we want to define a refined notion of weak equivalence in $\mathrm{C}^{-}(B)$ in order that such a complex $K_{\bullet}$ be unique up to weak equivalence. More precisely, a morphism $\varphi \colon K_{\bullet} \longrightarrow L_{\bullet}$ will be a weak equivalence if and only if both $\varphi$ and $\overline{\varphi}$ are quasi-isomorphisms in $\mathrm{C}^{-}(B)$ and $\mathrm{C}^{-}(A)$ respectively.
\par \medskip
\begin{definition}
Let $\mathfrak{N}$ be the null system in $\mathrm{K}^{-}(B)$ defined as follows:
\[
\mathfrak{N}=\{ K_{\bullet}\,\,\textrm{in}\,\,\mathrm{K}^{-}(B) \,\,\textrm{such that}\,\, \overline{K}_{\bullet} \,\,\textrm{and}\,\,\mathrm{Tor}^1_B(K_{\bullet}, A)\,\,\textrm{are exact}\}.
\]
The admissible derived category $\mathrm{D}^{\mathrm{adm}}(B)$ is the triangulated category defined as the localization of $\mathrm{K}^{-}(B)$ with respect to the null system $\mathfrak{N}$.
\end{definition}
Elements of $\mathfrak{N}$ are exact complexes, but the converse is not true. In fact, elements of $\mathfrak{N}$ are those for which the $\mathrm{E}^2$ page of the hypertor spectral sequence vanishes. 
Hence a morphism $\varphi \colon K_{\bullet} \longrightarrow L_{\bullet}$ is an isomorphism in $\mathrm{D}^{\mathrm{adm}}(B)$ if and only if $\overline{\varphi}$ and $ \mathrm{Tor}^1_B(\varphi, A)$ are quasi-isomorphisms.
\begin{remark} \label{gnagna}
Thanks to \eqref{oubli}, the null system can also be described as
\[
\mathfrak{N}=\{ K_{\bullet}\,\,\textrm{in}\,\,\mathrm{K}^{-}(B) \,\,\textrm{such that}\,\, {K}_{\bullet} \,\,\textrm{and}\,\,\overline{K}_{\bullet} \,\,\textrm{are exact}\}.
\]
Therefore a morphism of complexes $\varphi \colon K_{\bullet} \longrightarrow L_{\bullet}$ is an isomorphism in the category $\mathrm{D}^{\mathrm{adm}}(B)$ if and only if
$\varphi$ and $\overline{\varphi}$ are quasi-isomorphisms.
\end{remark}
Let us give a few properties related to the categories $\mathrm{D}^{\mathrm{adm}}(B)$. 
\begin{proposition} \label{menhir}
Let $\varphi \colon K_{\bullet} \longrightarrow L_{\bullet}$ be a quasi-isomorphism between two elements of $\mathrm{C}^{-}(B)$. 
\begin{enumerate}
\item[(i)] Assume that the complexes $K_{\bullet}$ and $L_{\bullet}$ are both $n$-admissible and of of length at most $n$ for some $n$ in $\mathbb{N} \cup \{ + \infty\}$.
Then $\varphi$ is an isomorphism in $\mathrm{D}^{\mathrm{adm}}(B)$.
\item[(ii)] Assume that $\varphi$ induces an isomorphism in $\mathrm{D}^{\mathrm{adm}}(B)$. Then $K_{\bullet}$ is admissible if and only if $L_{\bullet}$ is admissible.
\end{enumerate}
\end{proposition}

\begin{proof}
Let us consider the following diagram: 
\[
\xymatrix@C=50pt{
\mathbf{Tor}^i_B(K_{\bullet}, A) \ar[r]^-{\mathbf{Tor}^i_B(\varphi, A)} \ar[d] &  \mathbf{Tor}^i_B(L_{\bullet}, A) \ar[d] \\
\mathrm{H}_{i}(K_{\bullet} \otimes_B A) \ar[r]^-{\mathrm{H}_i(\overline{\varphi})} & \mathrm{H}_{i}(L_{\bullet} \otimes_B A)
}
\]
Since $\varphi$ is a quasi-isomorphism, the map $\mathbf{Tor}^i_B(\varphi, A)$ is an isomorphism. 
\par \medskip
For (i), thanks to Proposition \ref{cns}, both vertical arrows are isomorphisms for $0 \leq i \leq n$. It follows that $\mathrm{H}_i(\overline{\varphi})$ is an isomorphism, and since $K_{\bullet}$ and $L_{\bullet}$ are of length at most $n$, $\overline{\varphi}$ is an isomorphism. Hence $\varphi$ induces an isomorphism in $\mathrm{D}^{\mathrm{adm}}(B)$. 
\par \medskip
For (ii), if $\varphi$ induces an isomorphism in $\mathrm{D}^{\mathrm{adm}}(B)$, then, given any integer $i$, the map $\mathrm{H}_i(\overline{\varphi})$ is an isomorphism. It follows that the right vertical map is an isomorphism if and only if the left vertical map is an isomorphism. Using again Proposition \ref{cns}, we get the result.
\end{proof}

\begin{remark} \label{drone}
It follows from Proposition \ref{menhir} (ii) that if $K_{\bullet}$ and $L_{\bullet}$ are isomorphic in $\mathrm{D}^{\mathrm{adm}}(B)$, then $K_{\bullet}$ is admissible if and only if $L_{\bullet}$ is admissible.
\end{remark}

Let us now consider the upper truncation functors $\tau^{\geq -n} \colon \mathrm{C}^{\mathrm{b}}(B) \longrightarrow \mathrm{C}^{\mathrm{b}}(B)$, i.e., the functors defined by $\tau^{\geq -n} K_{\bullet}=(N \longrightarrow K_n \longrightarrow K_{n-1} \longrightarrow \ldots )$ where $N$ is the kernel of the map $K_n \longrightarrow K_{n-1}$ and is put in homological degree $n+1$. It is well known that $\tau^{\geq -n}$ induce an endofunctor of $\mathrm{D}^-(B)$. At the level of the category $\mathrm{D}^{\mathrm{adm}}(B)$, we have the following result: 
\begin{proposition} \label{tonk}
Let $K_{\bullet}$ be an admissible complex, and assume that there exists an integer $n$ such that $\mathrm{H}_p(\overline{K}_{\bullet})$ vanishes for $p>n$. Then $\tau^{\geq -n} K_{\bullet}$ is admissible and the morphism
$
K_{\bullet} \longrightarrow \tau^{\geq -n} K_{\bullet}
$
is an isomorphism in $\mathrm{D}^{\mathrm{adm}}(B)$.
\end{proposition}
\begin{proof}
Let $N=\mathrm{ker}\,\{K_{n} \longrightarrow K_{n-1}\}$. Then $\tau^{\geq -n} K_{\bullet}$ is the complex
\[
N \longrightarrow K_n \longrightarrow K_{n-1} \longrightarrow \cdots \longrightarrow K_0.
\]
We consider again the hypertor spectral sequence associated to this complex. Since $K_{\bullet}$ is admissible, $\mathrm{E}^2_{p, q}$ vanishes for all integers $p$ and $q$ such that $0 \leq p \leq n-1$ and $q \geq 1$. Hence 
\[
\mathrm{E}^2_{n, 1} \simeq \mathrm{E}^{\infty}_{n, 1} \simeq \mathrm{Gr}_n \mathbf{Tor}^{n+1}_B(\tau^{\geq -n} K_{\bullet}, A). 
\]
As $K_{\bullet}$ is admissible, we have isomorphisms
\[
\mathbf{Tor}^{n+1}_B(\tau^{\geq -n} K_{\bullet}, A) \simeq \mathbf{Tor}^{n+1}_B(K_{\bullet}, A) \simeq \mathrm{H}_{n+1}(\overline{K}_{\bullet})=\{0\}.
\]
Hence $\mathrm{E}^2_{n, 1}$ vanishes, so that $\tau^{\geq -n} K_{\bullet}$
is admissible. Then the result follows from Proposition \ref{menhir}.
\end{proof}
\begin{corollary} \label{bute}
A complex $K_{\bullet}$ is isomorphic in $\mathrm{D}^{-}(B)$ to a bounded admissible complex if and only if the derived pullback $K \, \lltensl{}_B \, A$ is cohomologically bounded.\footnote{The superscript ``$\ell$'' means that the tensor product is derived with respect to the left variable and not as a bifunctor, see \cite[\S 3]{Grivaux-HKR} for more details on this issue.}
\end{corollary}

\begin{proposition} \label{marre}
Given $(M_{\bullet}, N_{\bullet}, \mu, \delta)$, there exists $K_{\bullet}$ in $\mathrm{C}^{-}(B)$ corresponding to these data whose isomorphism class in $\mathrm{D}^{\mathrm{adm}}(B)$ is unique. Besides, the map
\[
(M_{\bullet}, N_{\bullet}, \mu, \delta) \longrightarrow K_{\bullet}
\]
is functorial, and $K_{\bullet}$ is admissible if and only if $\mu$ is an isomorphism in $\mathrm{D}^-(B)$.
\end{proposition}

\begin{proof}
Let $P_{\bullet}$ and  $Q_{\bullet}$ be projective resolution of $K_{\bullet}$ and $N_{\bullet}$ respectively. We can represent $\mu$ and $\delta$ by true morphisms $\mu \colon I \otimes P_{\bullet} \longrightarrow Q_{\bullet}$ and $\delta \colon P_{\bullet} \longrightarrow Q_{\bullet}[1]$. By adding if necessary a null homotopic complex to $P_{\bullet}$, we can even assume that $\mu$ becomes is surjective. Let $K_{\bullet}$ denote the cone of $\delta \colon P_{\bullet} \longrightarrow Q_{\bullet}[1]$ shifted by $-1$. Repeating the construction of Lemma \eqref{easy}, $K_{\bullet}$ is naturally a complex of $B$-modules satisfying all required properties.
\par \medskip
Let us now discuss uniqueness. Assume that a complex $L_{\bullet}$ corresponds to $(M_{\bullet}, N_{\bullet}, \mu, \delta)$. We have two morphisms $P_{\bullet} \longrightarrow \overline{L}_{\bullet}$ and $Q_{\bullet} \longrightarrow IL_{\bullet}$. By adding to $K_{\bullet}$ a null-homotopic complex if necessary, we can assume that these two morphisms are surjective. Then Lemma \ref{reidemaster} implies that $L_{\bullet}$ is isomorphic to $K_{\bullet}$ in $\mathrm{D}^{\mathrm{adm}}(B)$.
\end{proof}

To avoid dealing with unbounded projective complexes, it is possible to use perfect complexes instead of admissible ones. Perfect complexes of $B$-modules admit a very simple characterization, which we give now. We leave the adaptations from the admissible to the perfect setting to the reader.
\begin{proposition} \label{parfait}
Let $K_{\bullet}$ be a bounded complex of $B$-modules. Then $K_{\bullet}$ is perfect if and only if the derived pullback $K_{\bullet} \,\lltensl{}_B \,A$ is perfect.
\end{proposition}
\begin{proof}
If $K_{\bullet} \lltensl{}_B A$ is perfect, then Corollary \ref{bute} shows that $K_{\bullet}$ it quasi-isomorphic to a bounded admissible complex, so that we can assume that $K_{\bullet}$ is indeed bounded and admissible. Let us consider the exact sequence
\[
0 \longrightarrow I \otimes \overline{K}_{\bullet} \longrightarrow K_{\bullet} \longrightarrow \overline{K}_{\bullet} \longrightarrow 0
\]
as an exact sequence of $A$-modules, the $A$-module structure on $K_{\bullet}$ being given after the choice of a retraction $\sigma$ of the Atiyah sequence \eqref{ati}.
It gives a morphism 
\[
\alpha \colon \overline{K}_{\bullet} \longrightarrow I \otimes \overline{K}_{\bullet} [1]
\]
Thanks to Proposition \ref{marre}, the complex $K_{\bullet}$ can be reconstructed from the quadruplet $(\overline{K}_{\bullet}, I \otimes \overline{K}_{\bullet}, \mathrm{id}, \alpha)$. Since $\overline{K}_{\bullet}$ is a perfect complex, it admits a bounded projective resolution. Hence the proof of Proposition \ref{marre} shows that the complex $K_{\bullet}$ is isomorphic to a bounded complex of projective $B$-modules in $\mathrm{D}^{\mathrm{adm}}(B)$.
\end{proof}
%Hence we have a dictionary between perfect and admissible settings:
%\par \smallskip
%\begin{center}
%\begin{tabular}{>{\centering\arraybackslash}m{5cm}|>{\centering\arraybackslash}m{5cm}}
%\textbf{admissible}& \textbf{perfect}\\
%\hline
%bounded admissible complex & bounded complex of projectives \\
%\hline
%complex quasi-isomorphic to a bounded admissible complex &  perfect complex \\
%\hline
%$K_{\bullet} \, \lltensl{}_B A$ is cohomologically bounded & $K_{\bullet} \, \lltensl{}_B A$ is perfect \\
%\end{tabular}
%\end{center}
%\par \smallskip
\subsubsection{The local HKR class} \label{tnt}
Let $K_{\bullet}$ be a bounded complex of $B$-modules. We have an exact sequence of $A$-modules (the $A$-module-structure on $K_{\bullet}$ being given by $\sigma$):
\begin{equation} \label{ext2}
0 \longrightarrow \mathrm{Tor}^1_B(K_{\bullet}, A) \longrightarrow I \otimes \overline{K}_{\bullet} \xlongrightarrow{\mu_{K_{\bullet}}} K_{\bullet} \longrightarrow \overline{K}_{\bullet} \longrightarrow 0.
\end{equation}
\begin{definition} \label{extloc}
For any complex $K_{\bullet}$ of $B$-modules, the local HKR class of $K_{\bullet}$ is the morphism 
\[
{\theta}_{K_{\bullet}} \colon \overline{K}_{\bullet} \longrightarrow \mathrm{Tor}^1_B(K_{\bullet}, A) [2]
\]
in $\mathrm{D}^{\mathrm{b}}(A)$ associated with \eqref{ext2}.
\end{definition}

\begin{remark}
The morphism $\theta_{K_{\bullet}}$ is well defined on the admissible derived category $\mathrm{D}^{\mathrm{adm}}(B)$, but not on $\mathrm{D}^{-}(B)$. In fact we can see $\theta$ as a natural transformation 
\[
\xymatrix@R=30pt@C=50pt{\relax
   \mathrm{D}^{\mathrm{adm}}(B) \UN[r]{\mathrm{Tor}^0_B(\,*\,, \,A)}{\mathrm{Tor}^1_B(\,*\,, \, A)\,[2]}{\theta} & \mathrm{D}^-(A)}
\]
\end{remark}

\begin{theorem} \label{localobs}
Let $K_{\bullet}$ be a bounded complex of $B$-modules. Then the following properties are equivalent:
\begin{enumerate}
\vspace{0.2cm}
\item[(i)] The local HKR class ${\theta}_{K_{\bullet}}$ vanishes.
\vspace{0.2cm}
\item[(ii)] The map $K_{\bullet} \lltensl{}_B A \longrightarrow  K_{\bullet} \otimes_B A$ admits a right inverse in $\mathrm{D}^{-}(A)$.
\vspace{0.2cm}
\item[(iii)] There exists a bounded admissible complex $V_{\bullet}$ of $B$-modules and a morphism in $\mathrm{D}^{\mathrm{b}}(B)$ from $V_{\bullet}$ to $K_{\bullet}$ such that the induced map 
\[
V_{\bullet} \lltensl{}_B A \longrightarrow K_{\bullet} \lltensl{}_B A \longrightarrow K_{\bullet} \otimes_B A
\] 
in $\mathrm{D}^{-}(A)$ is an isomorphism.
\vspace{0.2cm}
\item[(iv)] There exists a bounded admissible complex $V_{\bullet}$ of $B$-modules and a morphism in $\mathrm{D}^{\mathrm{adm}}(B)$ from $V_{\bullet}$ to $K_{\bullet}$ such that the induced map $\overline{V}_{\bullet} \longrightarrow \overline{K}_{\bullet}$ 
in $\mathrm{D}^{\mathrm{b}}(A)$ is an isomorphism.
\vspace{0.2cm}
\item[(v)] There exists a bounded admissible complex $V_{\bullet}$ of $B$-modules and a sub-complex $T_{\bullet}$ of $I {V}_{\bullet}$ such that $K_{\bullet}$ is isomorphic to $V_{\bullet}/ T_{\bullet} $ in $\mathrm{D}^{\mathrm{adm}}(B)$.
\end{enumerate}
\par \medskip
Under any of these conditions,
\begin{align*}
K_{\bullet} \lltensl{}_B A & \simeq K_{\bullet} \otimes_B A \,\oplus\, \mathrm{Tor}^1_B(K_{\bullet}, A) \lltensl{}_B A \, [1] \\
& \simeq K_{\bullet} \otimes_B A \,\oplus\, \bigoplus_{p \geq 0} I^{\otimes p}  \otimes \mathrm{Tor}^1_B(K_{\bullet}, A) [p+1].
\end{align*}

\end{theorem}
\vspace{0.1cm}
\begin{proof}
$(\mathrm{i}) \longrightarrow (\mathrm{iv})$ Let us consider the two exact sequences
\[
\begin{cases}
0 \longrightarrow I K_{\bullet} \longrightarrow K_{\bullet} \longrightarrow \overline{K}_{\bullet} \longrightarrow 0\\
0 \longrightarrow \mathrm{Tor}^1_B(K_{\bullet}, A) \longrightarrow I \otimes \overline{K}_{\bullet} \xlongrightarrow{\mu_{K_{\bullet}}} I K_{\bullet} \longrightarrow 0
\end{cases}
\]
They yield two morphisms 
\[
\alpha \colon \overline{K}_{\bullet} \longrightarrow I K_{\bullet}[1] \quad \textrm{and} \quad \beta \colon  I K_{\bullet} \longrightarrow \mathrm{Tor}^1_B(K_{\bullet}, A)[1]
\] 
and $\beta \circ \alpha$ is exactly the local HKR class ${\theta}_{K_{\bullet}}$. By considering the exact sequence
\[
\xymatrix{
\mathrm{Hom}_{\mathrm{D}^{\mathrm{b}}(A)}(\overline{K}_{\bullet}, I \otimes \overline{K}_{\bullet} [1]) \ar[r] & \mathrm{Hom}_{\mathrm{D}^{\mathrm{b}}(A)}(\overline{K}_{\bullet}, I {K}_{\bullet} [1]) \ar[d]^-{\beta \,\circ\, (\star)} \\
& \mathrm{Hom}_{\mathrm{D}^{\mathrm{b}}(A)}(\overline{K}_{\bullet}, \mathrm{Tor}^1_B(K_{\bullet}, A) [2]) 
}
\]
we see that the map $\alpha$ can be lifted to a morphism 
\begin{equation} \label{sel}
\widetilde{\alpha} \colon \overline{K}_{\bullet} \longrightarrow I \otimes \overline{K}_{\bullet} [1]. 
\end{equation}
Using the notation of Proposition \ref{marre}, there is a morphism of quadruplets
\[
(\overline{K}_{\bullet}, I \otimes \overline{K}_{\bullet}, \mathrm{id}, \widetilde{\alpha}) \longrightarrow (\overline{K}_{\bullet}, I K_{\bullet}, \mu_{K_{\bullet}}, {\alpha}).
\]
Thanks to Proposition \ref{marre}, we get an admissible complex $V_{\bullet}$ and a morphism 
from $V_{\bullet}$ to $K_{\bullet}$
in $\mathrm{D}^{\mathrm{adm}}(B)$ such that the induced map from
$\overline{V}_{\bullet}$ to $\overline{K}_{\bullet}$
is an isomorphism in $\mathrm{D}^{\mathrm{b}}(A)$. Thanks to Proposition \ref{tonk}, we can replace $V_{\bullet}$ by a truncation of sufficiently high order, so that it becomes admissible and bounded.  
\par \medskip
$(\mathrm{iv}) \longrightarrow (\mathrm{v})$ Thanks to Remark \ref{drone}, we can assume up to changing the isomorphism class of $V_{\bullet}$ in $\mathrm{D}^{\mathrm{adm}}(B)$ that $V_{\bullet} \longrightarrow K_{\bullet}$ is a true morphism of complexes. Besides, by adding the null-homotopic complex $\mathrm{cone} \, (\mathrm{id}_{K_{\bullet}})[-1]$ which is admissible, we can assume that the morphism $V_{\bullet} \longrightarrow K_{\bullet}$ is degreewise surjective, and therefore, so are the morphisms $\overline{V}_{\bullet} \longrightarrow \overline{K}_{\bullet}$ and $IV_{\bullet} \longrightarrow IK_{\bullet}$. Let $T_{\bullet}$ be the kernel of the map $IV_{\bullet} \longrightarrow IK_{\bullet}$. We have a diagram
\[
\xymatrix{
0 \ar[r] & IV_{\bullet} / T_{\bullet} \ar[r] \ar[d]^-{\sim} & V_{\bullet} /T_{\bullet} \ar[r] \ar[d] & \overline{V}_{\bullet} \ar[r] \ar[d] & 0 \\
0 \ar[r] & IK_{\bullet} \ar[r] & K_{\bullet}  \ar[r] & \overline{K}_{\bullet} \ar[r]  & 0
}
\]
where first vertical arrow is an isomorphism, and the two remaining ones are componentwise surjective. Hence the kernels of these two arrows are isomorphic, so we get that the middle one is a quasi-isomorphism.
Since $ \overline{V}_{\bullet}\simeq \overline{V_{\bullet}/T_{\bullet}}$, the induced map
\[
\overline{V_{\bullet}/T_{\bullet}} \longrightarrow \overline{K}_{\bullet}
\]
is also a quasi-isomorphism. Hence $K_{\bullet}$ is isomorphic to the complex $V_{\bullet}/T_{\bullet}$ in the admissible derived category $\mathrm{D}^{\mathrm{adm}}(B)$ (\textit{see} Remark \ref{gnagna}). 
\par \medskip
$(\mathrm{v}) \longrightarrow (\mathrm{iv}) \longrightarrow (\mathrm{iii})  \longrightarrow (\mathrm{ii})$ Obvious.
\par \medskip
$(\mathrm{ii}) \longrightarrow (\mathrm{i})$ There is a natural $B$-linear morphism from $\sigma^* K_{\bullet}$ (where $K_{\bullet}$ is considered as a $A$-module) to $K_{\bullet}$. Let $S_{\bullet}$ denote its kernel, and let $L_{\bullet}$ be the resolution of $K_{\bullet}$ given by $L_{\bullet}=\mathrm{cone}\, (S_{\bullet} \longrightarrow \sigma^* K_{\bullet})$. As for any resolution, the corresponding map $L_{\bullet} \otimes_B A \longrightarrow K_{\bullet} \otimes_B A $ admits a right inverse given by the composition
\[
K_{\bullet} \otimes_B A \longrightarrow K_{\bullet} \lltensl{}_B A \simeq  L_{\bullet} \lltensl{}_B A
\longrightarrow L \otimes_B A.
\]
From this we see that the complex $\mathrm{cone}\,(S_{\bullet} \otimes_B A \longrightarrow K_{\bullet})$, where $K_{\bullet}$ is again considered as an $A$-module, is 
isomorphic to $\overline{K}_{\bullet} \oplus \mathrm{Tor}^1_B(K_{\bullet}, A) [1]$ in $\mathrm{D}^{-}(A)$. Now the natural morphism
\[
\xymatrix{
I \otimes \overline{K}_{\bullet}\ar[r] \ar[d] &K_{\bullet} \ar@{=}[d] \\
S_{\bullet} \otimes_B A \ar[r] &K_{\bullet}
}
\]
is a quasi-isomorphism (\textit{see} the proof of Proposition \ref{compris2}) after taking the cones of each line. Hence $\theta_{K_{\bullet}}$ vanishes.
\par \medskip
Let us now prove the last statement of the theorem. We can replace $K_{\bullet}$ by a complex of the form $P_{\bullet}/T_{\bullet}$ where $P_{\bullet}$ is a bounded complex of projective $B$-modules, and $T_{\bullet}$ is a sub-complex of the complex $I{P}_{\bullet}$. Then $\mathrm{Tor}^1_B(K_{\bullet}, A)$ is isomorphic to $T_{\bullet}$. Now we have a distinguished triangle
\[
T_{\bullet}\,  \lltensl{}_B A \longrightarrow P_{\bullet} \lltensl{}_B A \longrightarrow K_{\bullet} \, \lltensl{}_B A \xlongrightarrow{+1}
\]
The second map is a splitting of the map $K_{\bullet} \lltensl{}_B A \longrightarrow K_{\bullet} \otimes_B A$, so that $Q_{K_{\bullet}}$ is isomorphic to $T_{\bullet} \lltensl{}_B A \,[1]$. This gives the result.
\end{proof}
\subsubsection{Intrinsic construction of the local HKR class} \label{esthete}
Definition \ref{extloc} has an important drawback: it strongly depends on the map $\sigma \colon A \longrightarrow B$, that is on a splitting of $B$.
We can provide an alternative definition of the local HKR class to fix this. 
\par \medskip
For any complex $K_{\bullet}$ of $B$-modules, there is an exact sequence
\begin{equation} \label{boulon8}
0 \longrightarrow \mathrm{Tor}^1_B(K_{\bullet}, A)  \longrightarrow E \otimes \overline{K}_{\bullet} \longrightarrow \mathrm{P}^1_B(K_{\bullet}) \otimes_B A \longrightarrow \overline{K}_{\bullet} \longrightarrow 0.	
\end{equation}
given by \eqref{boulon}. Then we have: 

\begin{proposition} \label{campingaz}
For any bounded complex $K_{\bullet}$ of $B$-modules, the morphism from $\overline{K}_{\bullet}$ to $\mathrm{Tor}^1_B(K_{\bullet}, A)[2]$ given by \eqref{boulon8} is the local HKR class $\theta_{K_{\bullet}}$.
\end{proposition}

\begin{proof}
We have a commutative diagram of complexes A-modules (\textit{see}  the proof of Theorem \ref{wazomba}):
\[
\xymatrix{
&& 0  \ar[d] &0 \ar[d] & &  \\
&&  \Omega^1_A \otimes \overline{{K}}_{\bullet} \ar[r] \ar[d] & \Omega^1_A \otimes \overline{{K}}_{\bullet}  \ar[d] & &  \\
0 \ar[r] &\mathrm{Tor}^1_B({K}_{\bullet}, A) \ar[r] \ar@{=}[d]&{E} \otimes \overline{{K}}_{\bullet} \ar[d]\ar[r] & j^* \mathrm{P}^1_B({K}_{\bullet}) \ar[r] \ar[d] & \overline{{K}}_{\bullet} \ar[r] \ar@{=}[d]& 0 \\
0 \ar[r] & \mathrm{Tor}^1_B({K}_{\bullet}, A) \ar[r] &I \otimes  \overline{{K}}_{\bullet} \ar[r] \ar[d]& \sigma_* {K}_{\bullet} \ar[r] \ar[d] & \overline{{K}}_{\bullet} \ar[r] &0 \\
&& 0   &0  & &  \\
}
\]
where all lines and columns are exact. This proves the result.
\end{proof}

Lastly, we provide a result that will turn out to be crucial in the geometric case.

\begin{theorem} \label{shihiro}
Let $K_{\bullet}$ be a complex of $B$-modules, and consider the map 
\[
\Lambda \colon \overline{K}_{\bullet} \longrightarrow E \otimes \overline{K}_{\bullet} / \mathrm{Tor}^1_B(K_{\bullet}, A) [1]
\]
in $\mathrm{D}^-(B)$ obtained by taking the difference of the extension morphism attached to the exact sequence
\[
0 \longrightarrow E \otimes \overline{K}_{\bullet} /\mathrm{Tor}^1_B(K_{\bullet}, A)  \longrightarrow \mathrm{P}^1_B(K_{\bullet}) \otimes_B A \longrightarrow \overline{K}_{\bullet} \longrightarrow 0
\]
and the Atiyah morphism $\mathrm{at}_S(\overline{K}_{\bullet})$. Then $\Lambda$ is equal to the composition
\[
\overline{K}_{\bullet} \xlongrightarrow{\chi_{K_{\bullet}}} IK_{\bullet}[1] \simeq  I \otimes \overline{K}_{\bullet} /\mathrm{Tor}^1_B(K_{\bullet}, A) [1] \hookrightarrow  E \otimes \overline{K}_{\bullet} /\mathrm{Tor}^1_B(K_{\bullet}, A) [1]
\]
where $\chi$ denotes the residual Atiyah morphism (\textit{see} Definition \ref{residual}). 
\end{theorem}

\begin{proof}
The $A$-module $\mathrm{P}^1_B(K_{\bullet}) \otimes_B A$ is, as a $\mathbf{k}$-vector space the quotient 
\[
\frac{ E \otimes \overline{K}_{\bullet} /\mathrm{Tor}^1_B(K_{\bullet}, A) \oplus K_{\bullet}}{ \left\{ (i \otimes \overline k, ik), i \in I, k \in K_{\bullet} \right\}} \cdot
\]
The $A$-module structure is given by $a (e \otimes \overline{k}, k')=(ae \otimes \overline{k}+\overline{db}\otimes \overline{k}', bk')$
where $b$ is any element in $B$ such that $\overline{b}=a$.
Now the Baer difference of the two exact sequences
\[
\begin{cases}
0 \longrightarrow E \otimes \overline{K}_{\bullet} /\mathrm{Tor}^1_B(K_{\bullet}, A)  \longrightarrow \mathrm{P}^1_B(K_{\bullet}) \otimes_B A \longrightarrow \overline{K}_{\bullet} \longrightarrow 0 \\
0 \longrightarrow E \otimes \overline{K}_{\bullet} /\mathrm{Tor}^1_B(K_{\bullet}, A)  \longrightarrow \mathrm{P}^1_B(\overline{K}_{\bullet}) / \mathrm{Tor}^1_B(K_{\bullet}, A)  \longrightarrow \overline{K}_{\bullet} \longrightarrow 0 \\
\end{cases}
\]
is the extension whose middle term is the complex of $B$-modules
\[
M_{\bullet} = \frac{ E \otimes \overline{K}_{\bullet} /\mathrm{Tor}^1_B(K_{\bullet}, A) \oplus E \otimes \overline{K}_{\bullet} /\mathrm{Tor}^1_B(K_{\bullet}, A) \oplus K_{\bullet} }{ \left\{ (i \otimes \overline k, 0, ik)+(e \otimes \overline{k'}, e \otimes \overline{k'}, 0), i \in I, k, k' \in K_{\bullet} \right\}} \cdot
\]
We claim that the morphism from $K_{\bullet}$ to $M_{\bullet}$ given by $k \longrightarrow (0, 0, k)$ is $B$-linear. Indeed, we have
\[
b.(0, 0, k) = (\overline{db} \otimes k, \overline{db} \otimes k, bk)= (0, 0, bk).
\]
This gives a morphism of exact sequences
\[
\xymatrix{
0 \ar[r] & IK_{\bullet} \ar[r] \ar[d] & K_{\bullet} \ar[r] \ar[d] & \overline{K}_{\bullet} \ar[r] \ar@{=}[d] & 0 \\
0 \ar[r] &E \otimes \overline{K}_{\bullet} /\mathrm{Tor}^1_B(K_{\bullet}, A) \ar[r] & M_{\bullet} \ar[r] & \overline{K}_{\bullet} \ar[r] & 0 
}
\]
from which the result follows.
\end{proof}

\section{Deformation theory} \label{ispahan}

\subsection{Infinitesimal thickenings}  \label{setting}
Let $\mathbf{k}$ be a field of characteristic zero and let $(X, \mathcal{O}_X)$ be a $\mathbf{k}$-ringed space that is either a smooth $\mathbf{k}$-scheme or a smooth complex manifold (in this case $\mathbf{k}=\mathbb{C}$). We introduce the following standard notation:
\vspace{0.1cm}
\begin{enumerate}
\item[--] $\mathrm{C}^{\mathrm{b}}(X)$ (resp. $\mathrm{C}^{-}(X)$) is the category of bounded (resp. bounded from above) complexes of sheaves
of $\mathcal{O}_X$-modules.
\vspace{0.2cm}
\item[--] $\mathrm{K}^{\mathrm{b}}(X)$ and $\mathrm{K}^{-}(X)$ are the homotopy categories of $\mathrm{C}^{\mathrm{b}}(X)$ and $\mathrm{C}^{-}(X)$ respectively.
\vspace{0.2cm}
\item[--] $\mathrm{D}^{\mathrm{b}}(X)$ and $\mathrm{D}^{-}(X)$ are the derived categories of $\mathrm{C}^{\mathrm{b}}(X)$ and $\mathrm{C}^{-}(X)$ respectively.
\end{enumerate}
\par \medskip
Let $\mathbf{k}_X$ be the sheaf of locally constant $\mathbf{k}$-valued functions on $X$, and let $\mathcal{I}$ be a locally free sheaf of finite rank on $X$. 
\begin{definition}
An infinitesimal thickening of $X$ by $\mathcal{I}$ is a sheaf of $\mathbf{k}_X$-algebras $\mathcal{O}_S$ on $X$ fitting into an exact sequence of sheaves of $\mathbf{k}_X$-algebras
\begin{equation} \label{atiyah}
0 \longrightarrow \mathcal{I} \longrightarrow \mathcal{O}_S \longrightarrow \mathcal{O}_X \longrightarrow 0
\end{equation}
where $\mathcal{I}$ satisfies $\mathcal{I}^2=0$, that is locally split in the category of sheaves of $\mathbf{k}_X$-algebras. 
\end{definition}
The local splitting condition means that $\mathcal{O}_{S}$ is locally isomorphic to the trivial $\mathbf{k}_X$-extension of $\mathcal{O}_X$ by $\mathcal{I}$, which is the sheaf
$\mathcal{I} \oplus\mathcal{O}_{X}$ endowed with the ring structure
\[
(i, f). (i', f')=(if'+i'f, ff').
\]
\par \medskip
In geometrical terms, if we consider $S=(X, \mathcal{O}_S)$ as a ringed space, then there is a natural closed immersion $j \colon X \longrightarrow S$ that admits locally a right inverse. If we work in the algebraic category, the map $X \longrightarrow S$ is locally of the form $\mathrm{Spec}\, A \longrightarrow \mathrm{Spec} \, B$ where $B$ is the trivial $\mathbf{k}$-extension of $A$ by the free $A$-module $\Gamma(\mathrm{Spec}\, A, \mathcal{I})$.
\par \medskip
Let us introduce again some notation: for any sheaf of $\mathcal{O}_S$-modules $\mathcal{F}$ on $X$, we put
\vspace{0.1cm}
\begin{enumerate}
\item[--] $\overline{\mathcal{F}}=j^* \mathcal{F}$,
\vspace{0.2cm}
\item[--] $\mathrm{Tor}^i_S(\mathcal{F}, \mathcal{O}_X)=\mathrm{Tor}^i_{\mathcal{O}_S}(\mathcal{F}, \mathcal{O}_X)$.
\end{enumerate}
\par \medskip
If $(X, \mathcal{I})$ is given, the isomorphism classes of infinitesimal thickenings of $X$ by $\mathcal{I}$ are classified by the cohomology group $\mathrm{H}^1(X, \mathcal{D}er({\mathcal{O}_X}, \mathcal{I})),
$
and 
\[
\mathrm{H}^1(X, \mathcal{D}er({\mathcal{O}_X}, \mathcal{I})) \simeq \mathrm{H}^1(X, \mathcal{H}om (\Omega^1_X,  \mathcal{I})) \simeq \mathrm{Ext}^1_{\mathcal{O}_X}(\Omega^1_X,  \mathcal{I}).
\]
We can see this latter space as the space of morphisms in $\mathrm{D}^{\mathrm{b}}(X)$ from $\Omega_X^1$ to $\mathcal{I}[1]$. Hence every such ringed space $S$ is given (up to isomorphism) by a morphism
\begin{equation} \label{eta}
\eta \colon \Omega^1_X \longrightarrow \mathcal{I}[1]
\end{equation}
in the derived category $\mathrm{D}^{\mathrm{b}}(X)$ of coherent sheaves on $X$\footnote{The extension class corresponding to $\eta$ is called the Kodaira-Spencer class in \cite{Huybrechts-Thomas}.}. In the sequel, we will therefore consider an infinitesimal thickening of $X$ as a triplet $(X, \mathcal{I}, \eta)$ where $\eta$ is a morphism in $\mathrm{D}^{\mathrm{b}}(X)$ from $\Omega^1_X$ to $\mathcal{I}[1]$.
Let $\mathcal{E}=j^* {\Omega^1_S}$. We can write down the conormal exact sequence of $j$, which is
\begin{equation} \label{conormal}
0 \longrightarrow \mathcal{I} \longrightarrow \mathcal{E} \longrightarrow \Omega^1_X \longrightarrow 0, 
\end{equation}
and its extension class is precisely $\eta$. This shows how to extract $\eta$ intrinsically from the pair $(X, S)$.
\par \medskip
A particular case of this construction is the following one: fix a closed embedding $i \colon X \longrightarrow Y$ of complex manifolds, and define $S$ as the first formal neighbourhood of $X$ in $Y$. Then $\eta$ is the extension class of the conormal exact sequence
\[
0 \longrightarrow \mathrm{N}^*_{X/Y} \longrightarrow \Omega^1_{Y | X} \longrightarrow \Omega^1_X \longrightarrow 0.
\]
Many notions that have been introduced in \S \ref{carrenul} for the local case admit a straightforward adaptation in the geometric setting, and some need to be refined. Let us be more specific:
\begin{enumerate}
\item[--] All the results in \S \ref{nakon} remain unchanged, the most important ones being Proposition \ref{compris2} and Theorem \ref{wazomba}. \vspace{0.2cm}
\item[--] The theory of admissible complexes developed in \S \ref{lauvitel} remains unchanged. Concerning \S \ref{ex}, the derived category $\mathrm{D}^{\mathrm{adm}}(S)$ is well-defined.
However, Proposition \ref{marre} only holds when the thickening $S$ is trivial (that is when $j$ admits a global retraction). Lastly, the characterization of perfect complexes (Proposition \ref{parfait}) is still valid since it is a local property on $X$.
\vspace{0.2cm}
\item[--] The material of \S \ref{tnt} can not be directly adapted unless $S$ is globally trivial. We will explain in the remaining part of the section how to define an analogue of the local HKR class is the geometric setting, in order that Theorem \ref{localobs} be valid.
\vspace{0.2cm}
\item[--] In \S \ref{esthete}, Theorem \ref{shihiro} remains valid.
\end{enumerate}
\par \medskip
For any complex $\mathcal{K}_{\bullet}$ of $\mathcal{O}_S$-modules, we have an exact sequence
\begin{equation} \label{pout}
0 \longrightarrow \mathrm{Tor}^1_S(\mathcal{K}_{\bullet}, \mathcal{O}_X) \longrightarrow \mathcal{E} \otimes \overline{\mathcal{K}}_{\bullet} \longrightarrow j^* \mathrm{P}^1_S(\mathcal{K}_{\bullet}) \longrightarrow \overline{\mathcal{K}}_{\bullet} \longrightarrow 0
\end{equation}
that corresponds to the sheaf version of \eqref{boulon8}. 
\par \medskip
\begin{definition} \label{farid} For any complex $\mathcal{K}_{\bullet}$ in $\mathrm{C}^{-}(S)$, the geometric HKR class of $\mathcal{K}_{\bullet}$ is the morphism
\[
\Theta_{\mathcal{K}_{\bullet}} \colon \overline{\mathcal{K}}_{\bullet} \longrightarrow \mathrm{Tor}^1_S(\mathcal{K}_{\bullet}, \mathcal{O}_X)[2]
\]
given by \eqref{pout}.
\end{definition}
The analogue of Proposition \eqref{campingaz} holds under the assumption that $S$ is locally trivial: 
\begin{proposition}
If $S$ is globally trivial (that is if the embedding $j$ admits a retraction $\sigma \colon S \longrightarrow X$) the global HKR class $\Theta_{\mathcal{K}_{\bullet}}$ is the extension class associated with the exact complex of $\mathcal{O}_X$-modules 
\[
0 \longrightarrow \mathrm{Tor}^1_S(\mathcal{K}_{\bullet}, \mathcal{O}_X) \longrightarrow  \mathcal{I} \otimes \overline{\mathcal{K}}_{\bullet} \longrightarrow \sigma_* \mathcal{K}_{\bullet} \longrightarrow \overline{\mathcal{K}}_{\bullet} \longrightarrow 0
\]
corresponding to the multiplication map.
\end{proposition}

\begin{proof}
The proof is identical to the one of Proposition \ref{campingaz}.
\end{proof}

\begin{proposition} \label{link}
For anycomplex ${\mathcal{K}_{\bullet}}$ in $\mathrm{C}^-(X)$, the morphism $\Theta_{{\mathcal{K}_{\bullet}}}$ is the composition
\[
{\mathcal{K}_{\bullet}} \xlongrightarrow{\mathrm{at}_X(\mathcal{K}_{\bullet})} \Omega^1_X \otimes {\mathcal{K}_{\bullet}} [1] \xlongrightarrow{\mathrm{id}\, \otimes \, \eta}  \mathcal{I} \otimes {\mathcal{K}_{\bullet}} [2]
\]
where $\mathrm{at}_X({\mathcal{K}_{\bullet}})$ denotes the Atiyah class of ${\mathcal{K}_{\bullet}}$.
\end{proposition}
\begin{proof}
The morphisms $\mathrm{at}_X(\mathcal{K}_{\bullet})$ and $\mathrm{id} \otimes  \eta$ correspond to the extension classes of the short exact sequences
\[
\begin{cases}
0 \longrightarrow  \Omega^1_X \, \otimes \, {\mathcal{K}_{\bullet}}  \longrightarrow \mathrm{P}^1_X({\mathcal{K}_{\bullet}}) \longrightarrow {\mathcal{K}_{\bullet}} \longrightarrow 0 \\
0 \longrightarrow \mathcal{I} \, \otimes \, {\mathcal{K}_{\bullet}} \longrightarrow \mathcal{E} \, \otimes \,  {\mathcal{K}_{\bullet}}  \longrightarrow  \Omega^1_X \, \otimes \, {\mathcal{K}_{\bullet}} \longrightarrow 0.
\end{cases}
\]
Their Yoneda product is the exact sequence
\[
0 \longrightarrow\mathcal{I}   \, \otimes \,   {\mathcal{K}_{\bullet}} \longrightarrow \mathcal{E}   \, \otimes \, {\mathcal{K}_{\bullet}} \longrightarrow \mathrm{P}^1_X({\mathcal{K}_{\bullet}}) \longrightarrow {\mathcal{K}_{\bullet}} \longrightarrow 0
\]
and the corresponding morphism from ${\mathcal{K}_{\bullet}}$ to $\mathcal{I}  \otimes    {\mathcal{K}_{\bullet}}\,  [2]$  in the derived category is $\Theta_{\mathcal{K}_{\bullet}}$.
\end{proof}
\begin{corollary} \label{mononoke}
For any element $\mathcal{K}_{\bullet}$ and $\mathcal{L}_{\bullet}$ in $\mathrm{C}^{-}(X)$, the morphism
\[
\Theta_{\mathcal{K}_{\bullet} \, \lltens{}_{\mathcal{O}_X} \, \mathcal{L}_{\bullet}} \colon \mathcal{K}_{\bullet} \, \lltens{}_{\mathcal{O}_X} \, \mathcal{L}_{\bullet} \longrightarrow \mathcal{I} \, \lltens{}_{\mathcal{O}_X} \, \mathcal{K}_{\bullet} \, \lltens{}_{\mathcal{O}_X} \, \mathcal{L}_{\bullet}\,[2]
\]
is equal to $\,\Theta_{\mathcal{K}_{\bullet}} \, \lltens{}_{\mathcal{O}_X} \, \mathrm{id}_{\mathcal{L}_{\bullet}}+ \mathrm{id}_{\mathcal{K}_{\bullet}} \, \lltens{}_{\mathcal{O}_X} \,\Theta_{\mathcal{L}_{\bullet}}$.
\end{corollary}
\begin{proof}
This follows from the analogous formula for the Atiyah morphism, which is well-known (\textit{see} \cite[Lemma 2]{Markarian}).
\end{proof}

\subsection{The global extension theorem}
In this section, we state and prove the geometric version of Theorem \ref{localobs}. 
\begin{theorem} \label{tabriz}
For any bounded complex $\mathcal{K}_{\bullet}$ of $\mathcal{O}_S$-modules, the following properties are equivalent:
\vspace{0.1cm}
\begin{enumerate}
\item[(i)] The HKR class ${\Theta}_{\mathcal{K}_{\bullet}}$ vanishes.
\vspace{0.2cm}
\item[(ii)] The morphism $\mathbb{L}j^* \mathcal{K}_{\bullet} \longrightarrow j^* \mathcal{K}_{\bullet}$ admits a right inverse in $\mathrm{D}^{\mathrm{-}}(X)$.
\vspace{0.2cm}
\item[(iii)] There exists a bounded admissible complex $\mathcal{L}_{\bullet}$ and a morphism in $\mathrm{D}^{\mathrm{b}}(S)$  from  $\mathcal{L}_{\bullet}$ to $\mathcal{K}_{\bullet}$ such that the composition
\[
\mathbb{L}j^* \mathcal{L}_{\bullet} \longrightarrow \mathbb{L} j^* \mathcal{K}_{\bullet} \longrightarrow j^* \mathcal{K}_{\bullet}
\]
is an isomorphism in $\mathrm{D}^{-}(X)$.
\vspace{0.2cm}
\item[(iv)] There exists a bounded admissible complex $\mathcal{L}_{\bullet}$ and a morphism in $\mathrm{D}^{\mathrm{adm}}(S)$  from  $\mathcal{L}_{\bullet}$ to $\mathcal{K}_{\bullet}$ such that the induced morphism from $\overline{\mathcal{L}}_{\bullet}$ to $\overline{\mathcal{K}}_{\bullet}$
is an isomorphism in $\mathrm{D}^{\mathrm{b}}(X)$.
\vspace{0.2cm}
\item[(v)] There exists a bounded admissible complex $\mathcal{L}_{\bullet}$ and a sub-complex $\mathcal{T}_{\bullet}$ of $\mathcal{I} \mathcal{L}_{\bullet}$ such that $\mathcal{K}_{\bullet}$ is isomorphic to $\mathcal{L}_{\bullet} / \mathcal{T}_{\bullet}$ in $\mathrm{D}^{\mathrm{adm}}(S)$.
\end{enumerate}
\vspace{0.2cm}
If any of these properties hold, there is an isomorphism
\[
\mathbb{L}j^* \mathcal{K}_{\bullet} \simeq j^* \mathcal{K}_{\bullet} \oplus \mathbb{L}j^*\, \mathrm{Tor}^1_{S}(\mathcal{K}_{\bullet}, \mathcal{O}_X) [1].
\]
\end{theorem}

\begin{proof}
The implications $(\textrm{v}) \longrightarrow (\textrm{iv}) \longrightarrow (\textrm{iii}) \longrightarrow (\textrm{ii})$ are straightforward.
\par \medskip $(\mathrm{ii}) \longrightarrow (\mathrm{i})$ The proof follows closely the corresponding implication in Theorem \ref{localobs} except that we have to work with the geometric local HKR class. Let $\mathcal{Q}_{\mathcal{K}_{\bullet}}$ be the cone of the morphism $\mathbb{L} {j}^* \mathcal{K}_{\bullet} \longrightarrow j^* \mathcal{K}_{\bullet}$ shifted by $-1$, so that there is an exact triangle
\begin{equation} \label{noos} \mathbb{L} j^* \mathcal{K}_{\bullet} \longrightarrow j^* \mathcal{K}_{\bullet} \longrightarrow \mathcal{Q}_{\mathcal{K}_{\bullet}} [1] \xlongrightarrow{+1}
\end{equation}
Let $\widetilde{\mathcal{K}}_{\bullet}$ be the cone of the natural morphism from $\Omega^1_S \otimes \mathcal{K}_{\bullet}$ to $\mathrm{P}^1_S(\mathcal{K}_{\bullet})$. We fix a projective resolution $\mathcal{P}_{\bullet} \longrightarrow \mathcal{K}_{\bullet}$ of $\mathcal{K}_{\bullet}$. Since $\widetilde{\mathcal{K}}_{\bullet}$ and $\mathcal{K}_{\bullet}$ are isomorphic in the derived category, there exists a morphism from $\mathcal{P}_{\bullet}$ to $\widetilde{\mathcal{K}}_{\bullet}$ such that the diagram 
\[
\xymatrix{
\mathcal{P}_{\bullet} \ar[r] \ar[d] & \mathcal{K}_{\bullet} \ar@{=}[d] \\
\widetilde{\mathcal{K}}_{\bullet} \ar[r] &\mathcal{K}_{\bullet}
}
\]
commutes. Let $\mathcal{M}_{\bullet}$ (resp. $\mathcal{N}_{\bullet}$) be the cone of the the morphism $\mathcal{P}_{\bullet} \longrightarrow \mathcal{K}_{\bullet}$ (resp. $\widetilde{\mathcal{K}}_{\bullet} \longrightarrow \mathcal{K}_{\bullet}$). Then we have a morphism of distinguished triangles
\[
\xymatrix{
j^* \mathcal{P}_{\bullet} \ar[r] \ar[d] & j^* \mathcal{K}_{\bullet} \ar[r] \ar[d] & j^* \mathcal{M}_{\bullet} \ar[r]^-{+1} \ar[d]& \\
j^* \widetilde{\mathcal{K}}_{\bullet} \ar[r] & j^* \mathcal{K}_{\bullet} \ar[r] & j^* \mathcal{N}_{\bullet} \ar[r]^-{+1} &
}
\]
in the homotopy category $\mathrm{K}^{\mathrm{b}}(X)$. Notice that the first triangle is isomorphic to \eqref{noos} in $\mathrm{D}^{\mathrm{b}}(X)$. Now $j^* \mathcal{N}_{\bullet}$ is the iterated cone of the morphisms of complexes
\[
\Omega^1_X \otimes j^* \mathcal{K}_{\bullet} \longrightarrow j^* \mathrm{P}^1_S (\mathcal{K}_{\bullet}) \longrightarrow j^* \mathcal{K}_{\bullet}
\]
so it is isomorphic to $\mathrm{Tor}^1_S(\mathcal{K}_{\bullet}, \mathcal{O}_X)[2]$ in $\mathrm{D}^{\mathrm{b}}(X)$, and via this isomorphism the morphism $j^* \mathcal{K}_{\bullet} \longrightarrow j^* \mathcal{N}_{\bullet}$ is nothing but $\Theta_{\mathcal{K}_{\bullet}}$.
Hence we get a commutative diagram of morphisms
\[
\xymatrix@C=10pt@R=40pt{
j^* \mathcal{K}_{\bullet} \ar[rr] \ar[rd]_-{\Theta_{\mathcal{K}_{\bullet}}} && \mathcal{Q}_{\mathcal{K}_{\bullet}}[1] \ar[ld] \\
&\mathrm{Tor}^1_S(\mathcal{K}_{\bullet}, \mathcal{O}_X) [2]&
}
\]
in the derived category $\mathrm{D}^{\mathrm{b}}(X)$. If the natural morphism from $\mathbb{L} j^* \mathcal{K}_{\bullet}$ to $j^* \mathcal{K}_{\bullet}$ admits a right inverse, the connection morphism from $j^* \mathcal{K}_{\bullet}$ to $\mathcal{Q}_{\mathcal{K}_{\bullet}}[1]$ associated to \eqref{noos} vanishes, and so does $\Theta_{\mathcal{K}_{\bullet}}$.
\par \medskip
$(\textrm{i}) \longrightarrow (\textrm{iii})$. This is the crucial step. We consider the two exact sequences
\[
\begin{cases}
0 \longrightarrow \mathcal{E} \otimes \overline{\mathcal{K}}_{\bullet}/ \mathrm{Tor}^1_S(\mathcal{K}_{\bullet}, \mathcal{O}_X) \longrightarrow j^* \mathrm{P}^1_S(\mathcal{K}_{\bullet}) \longrightarrow \overline{\mathcal{K}}_{\bullet} \longrightarrow 0 \\
0 \longrightarrow \mathrm{Tor}^1_S(\mathcal{K}_{\bullet}, \mathcal{O}_X) \longrightarrow \mathcal{E} \otimes \overline{\mathcal{K}}_{\bullet} \longrightarrow \mathcal{E} \otimes \overline{\mathcal{K}}_{\bullet}/ \mathrm{Tor}^1_S(\mathcal{K}_{\bullet}, \mathcal{O}_X) \longrightarrow 0
\end{cases}
\]
They yield two morphisms
\[
\begin{cases}
\gamma \colon \overline{\mathcal{K}}_{\bullet} \longrightarrow \mathcal{E} \otimes \overline{\mathcal{K}}_{\bullet}/ \mathrm{Tor}^1_S(\mathcal{K}_{\bullet}, \mathcal{O}_X)[1] \\ \delta \colon \mathcal{E} \otimes \overline{\mathcal{K}}_{\bullet}/ \mathrm{Tor}^1_S(\mathcal{K}_{\bullet}, \mathcal{O}_X) \longrightarrow  \mathrm{Tor}^1_S(\mathcal{K}_{\bullet}, \mathcal{O}_X) [1]
\end{cases}
\]
and $\delta \circ \gamma = \Theta_{\mathcal{K}_{\bullet}}$. The exact sequence
\[
\xymatrix{
\mathrm{Hom}_{\mathrm{D}^{\mathrm{b}}(X)}(\overline{\mathcal{K}}_{\bullet}, \mathcal{E} \otimes \overline{\mathcal{K}}_{\bullet}[1]) \ar[r] & \mathrm{Hom}_{\mathrm{D}^{\mathrm{b}}(X)}(\overline{\mathcal{K}}_{\bullet}, \mathcal{E} \otimes \overline{\mathcal{K}}_{\bullet}/ \mathrm{Tor}^1_S(\mathcal{K}_{\bullet}, \mathcal{O}_X)[1]) \ar[d] \\
& \mathrm{Hom}_{\mathrm{D}^{\mathrm{b}}(X)}(\overline{\mathcal{K}}_{\bullet}, \mathrm{Tor}^1_S(\mathcal{K}_{\bullet}, \mathcal{O}_X)[2])
}
\]
shows that the map $\gamma$ can be lifted to a morphism 
\[
\widetilde{\gamma} \colon \overline{\mathcal{K}}_{\bullet} \longrightarrow \mathcal{E} \otimes \overline{\mathcal{K}}_{\bullet}[1].
\]
Let $\widetilde{\Lambda} \colon  \overline{\mathcal{K}}_{\bullet} \longrightarrow \mathcal{E} \otimes \overline{\mathcal{K}}_{\bullet}[1]$ be the morphism $\widetilde{\gamma}- \mathrm{at}_S(\overline{\mathcal{K}}_{\bullet})$  in $\mathrm{D}^{\mathrm{b}}(S)$. Then the composition 
\[\overline{\mathcal{K}} \longrightarrow  \mathcal{E} \otimes \overline{\mathcal{K}}_{\bullet}[1] \longrightarrow \mathcal{E} \otimes \overline{\mathcal{K}}_{\bullet}/ \mathrm{Tor}^1_S(\mathcal{K}_{\bullet}, \mathcal{O}_X)[1]
\]
is the geometric counterpart of the morphism $\Lambda$ introduced in Theorem \ref{shihiro}. Let us consider the following two distinguished triangles in $\mathrm{D}^{\mathrm{b}}(S)$: 
\[
\xymatrix@C=50pt@R=40pt{
\mathcal{K}_{\bullet} \ar[r] \ar@{-->}[d]_{\Psi}  &  \overline{\mathcal{K}}_{\bullet} \ar[r] \ar[d]_-{\widetilde{\Lambda}} \ar[rd]_-{\Lambda} & \mathcal{I} \mathcal{K}_{\bullet}[1] \ar[d]  \ar[r]^-{+1} &  \\
\mathrm{Tor}^1_S(\mathcal{K}_{\bullet}, \mathcal{O}_X)[1] \ar[r] & \mathcal{E} \otimes \overline{\mathcal{K}}_{\bullet}[1] \ar[r]& \mathcal{E} \otimes \overline{\mathcal{K}}_{\bullet}/ \mathrm{Tor}^1_S(\mathcal{K}_{\bullet}, \mathcal{O}_X)[1] \ar[r]^-{+1}  &
}
\]
\par \medskip
In the middle square, the bottom triangle commutes, and the top triangle also commutes thanks to Theorem \ref{shihiro} which remains valid in this context. Hence there exists a morphism $\Psi \colon \mathcal{K}_{\bullet} \longrightarrow \mathrm{Tor}^1_S(\mathcal{K}_{\bullet}, \mathcal{O}_X)[1]$ making the above diagram a morphism of distinguished triangles. 
\par \medskip
Define $\mathcal{L}^{\bullet}$ as the cone of the morphism $\mathcal{K}_{\bullet}[-1] \xlongrightarrow{{\Psi}[-1]} \mathrm{Tor}^1_S(\mathcal{K}_{\bullet}, \mathcal{O}_X)$. We claim that the composition
\[
\mathbb{L}j^* \mathcal{L}_{\bullet}  \longrightarrow \mathbb{L}j^* \mathcal{K}_{\bullet} \longrightarrow j^* \mathcal{K}_{\bullet}
\]
is an isomorphism in $\mathrm{D}^-(X)$. This statement being local, let us assume that we are in the globally split case, so that $j  \colon X \longrightarrow S$ admits a global retraction $\sigma$. Hence $\mathcal{E}$ splits as $\mathcal{I} \oplus \Omega^1_X$. Then we have a morphism of distinguished triangles
\[
\xymatrix@C=50pt@R=40pt{
\mathrm{Tor}^1_S(\mathcal{K}_{\bullet}, \mathcal{O}_X)[1] \ar[r] \ar@{=}[d] & \mathcal{E} \otimes \overline{\mathcal{K}}_{\bullet}[1] \ar[r] \ar[d] & \mathcal{E} \otimes \overline{\mathcal{K}}_{\bullet}/ \mathrm{Tor}^1_S(\mathcal{K}_{\bullet}, \mathcal{O}_X)[1] \ar[r]^-{+1} \ar[d] & \\
\mathrm{Tor}^1_S(\mathcal{K}_{\bullet}, \mathcal{O}_X)[1] \ar[r] & \mathcal{I} \otimes \overline{\mathcal{K}}_{\bullet}[1] \ar[r]&\mathcal{I} \mathcal{K}_{\bullet} [1] \ar[r]^-{+1}  &
}
\]
Putting the two together, this gives
\[
\xymatrix@C=20pt@R=30pt{
\mathcal{L}_{\bullet} \ar[d] \ar@{-->}[r]& \mathcal{M}_{\bullet} \ar[d]&& \\
\mathcal{K}_{\bullet} \ar[r] \ar@{-->}[d]_{\Psi}  &  \overline{\mathcal{K}}_{\bullet} \ar[r] \ar[d] & \mathcal{I} \mathcal{K}_{\bullet}[1] \ar@{=}[d]  \ar[r]^-{+1} &  \\
\mathrm{Tor}^1_S(\mathcal{K}_{\bullet}, \mathcal{O}_X)[1] \ar[d]^-{+1} \ar[r] & \mathcal{I} \otimes \overline{\mathcal{K}}_{\bullet}[1] \ar[r] \ar[d]^-{+1}&\mathcal{I} \mathcal{K}_{\bullet} [1] \ar[r]^-{+1}  & \\
&&&
}
\]
Thanks to the geometric version of proposition \ref{marre}, the cone of the second vertical arrow shifted by $-1$ is isomorphic in $\mathrm{D}^{\mathrm{b}}(S)$ to an admissible complex, and the composite morphism 
\[
\mathbb{L}j^* \mathcal{M}_{\bullet} \longrightarrow \mathbb{L}j^* \overline{\mathcal{K}}_{\bullet} \longrightarrow \overline{\mathcal{K}}_{\bullet}
\]
is an isomorphism. Besides, since the last vertical arrow of the morphisms of distinguished triangles is the identity, the nine lemma shows that any vertical morphism from $\mathcal{L}_{\bullet}$ to $\mathcal{M}_{\bullet}$ defining a morphism between the two vertical distinguished triangles is an isomorphism. Hence the composition
\[
\mathbb{L}j^* \mathcal{L}_{\bullet} \longrightarrow \mathbb{L}j^* {\mathcal{K}}_{\bullet} \longrightarrow \overline{\mathcal{K}}_{\bullet}
\]
is an isomorphism. To conclude, it suffices to replace $\mathcal{L}$ an admissible isomorphic complex, which will be automatically bounded.
\par \medskip
$(\textrm{iv}) \longrightarrow (\textrm{v})$ Same proof as in the local case.
\end{proof}

\subsection{The case of a single sheaf}
In this section, we deal with the case where $\mathcal{K}_{\bullet}$ is concentrated in a single degree. The main result we prove is the following:
\begin{theorem} \label{belote}
Let $\mathcal{K}$ be a sheaf of $\mathcal{O}_S$-modules. Then the cone of $\Theta_{\mathcal{K}}[-1]$ is isomorphic in $\mathrm{D}^{\mathrm{b}}(X)$ to $\tau^{\geq -1} \,\mathbb{L}j^* \,\mathcal{K}$.
\end{theorem}
\begin{proof}
The complex $\Omega^1_S \otimes \mathcal{K} \longrightarrow \mathrm{P}^1_S(\mathcal{K})$ is a resolution of $\mathcal{K}$. Besides, this complex is $1$-admissible (\textit{see} Definition \ref{meije}): indeed, thanks to Theorem \eqref{wazomba}, the map
\[
\mathrm{Tor}^1_S(\Omega^1_S \otimes \mathcal{K}, \mathcal{O}_X) \longrightarrow 
\mathrm{Tor}^1_S(\mathrm{P}^1_S(\mathcal{K}), \mathcal{O}_X)
\]
is surjective. Therefore, $\tau^{\geq -1} \,\mathbb{L}j^* \mathcal{K}$ is isomorphic to
\[
\mathcal{E} \otimes j^* \mathcal{K} \longrightarrow j^* \mathrm{P}^1_S(\mathcal{K})
\]
in $\mathrm{D}^{\mathrm{b}}(X)$. This gives the result.
\end{proof}
In this situation, we can complete the picture of Theorem \ref{wazomba} by the two following results:

\begin{theorem} \label{yek}
For any sheaf $\mathcal{K}$ of $\mathcal{O}_S$-modules, the following properties are equivalent:
\vspace{0.2cm}
\begin{enumerate}
\item[(i)] The HKR class ${\Theta}_{\mathcal{K}}$ vanishes.
\vspace{0.2cm}
\item[(ii)] The morphism $\mathbb{L}j^* \mathcal{K} \longrightarrow j^* \mathcal{K}$ admits a right inverse in $\mathrm{D}^{\mathrm{-}}(X)$.
\vspace{0.2cm}
\item[(iii)] The object $\tau^{\geq -1}  \, \mathbb{L}j^* \mathcal{K}$ is formal in $\mathrm{D}^{-}(X)$.
\vspace{0.2cm}
\item[(iv)] The sheaf $\mathcal{K}$ extends to an admissible sheaf on $S$.
\vspace{0.2cm}
\end{enumerate}
Under any of these conditions, 
$
\mathbb{L}j^* \mathcal{K} \simeq j^* \mathcal{K} \oplus \mathbb{L}j^* \, \mathrm{Tor}^1_S(\mathcal{K}, \mathcal{O}_X) [1].
$
\end{theorem}

\begin{proof} $ $ \par
\vspace{0.2cm}
$(\textrm{ii}) \Leftrightarrow (\textrm{iii})$ Obvious.
\par \medskip
$(\textrm{i}) \Leftrightarrow (\textrm{ii})$ and $(\textrm{iv}) \longrightarrow (\textrm{ii})$ Follows from Theorem \ref{tabriz}.
\par \medskip
$(\textrm{i}) \longrightarrow (\textrm{iv})$ According to Theorem \ref{tabriz}, there exists
an admissible complex $\mathcal{L}_{\bullet}$ concentrated in negative degrees and a morphism from $\mathcal{L}_{\bullet}$ to $\mathcal{K}$ in $\mathrm{D}^{\mathrm{adm}}(S)$ such that the composition
\[
j^* \mathcal{L}_{\bullet} \longrightarrow j^* \mathcal{K}
\]
is an isomorphism. According to Proposition \ref{tonk}, we can replace $\mathcal{L}_{\bullet}$ by its last truncation $\mathcal{H}_0(\mathcal{L}_{\bullet})$, which is still admissible.
\end{proof}

\begin{corollary} \label{hmpf}
Let $\mathcal{K}$ be a sheaf of $\mathcal{O}_S$-modules. Then $\mathbb{L}j^* \mathcal{K}$ is formal in $\mathrm{D}^-(X)$ if and only $\Theta_{\mathcal{K}}$ and $\{ \Theta_{\mathrm{Tor}^p_{\mathcal{O}_S}(\mathcal{K}, \mathcal{O}_X)} \}_{p \geq 0}$ vanish.
\end{corollary}

\section{Structure of derived self intersections} \label{colline}

\subsection{Preliminar material}
We fix a pair $(X, S)$ where $S$ is a locally trivial thickening of $X$ defined by a pair ($\mathcal{I}, \eta)$ where $\mathcal{I}$ is a locally free sheaf on $\mathcal{I}$ and $\eta$ is a class in $\mathrm{Ext}^1_{\mathcal{O}_X}(\Omega^1_X, \mathcal{I})$. We now specialize to the case the complex of $\mathcal{O}_S$-modules $\mathcal{K}_{\bullet}$ is of the form $j_* \mathcal{V}_{\bullet}$ with $\mathcal{V}_{\bullet}$ a complex of $\mathcal{O}_X$-modules.
\begin{lemma} \label{tree}
The direct image functor $j_* \colon \mathrm{C}^{-}(X) \longrightarrow \mathrm{C}^-(S)$ factorizes through a functor 
\[
j_* \colon \mathrm{D}^-(X) \longrightarrow \mathrm{D}^{\mathrm{adm}}(S)
\] 
that lifts the usual push forward functor $j_*$ at the level of derived categories.
\end{lemma}

\begin{proof}
We must prove that for quasi-isomorphism $\varphi \colon \mathcal{V}_{\bullet} \longrightarrow \mathcal{W}_{\bullet} $ between elements in $\mathrm{C}^{-}(X)$, $j_* \varphi$ is an isomorphism in the admissible derived category $\mathrm{D}^{\mathrm{adm}}(S)$. This is straightforward, since $\overline{j_* \varphi}$ equals $\varphi$, and $\mathrm{Tor}^1_S(j_* \varphi, \mathcal{O}_X)$ equals $\mathrm{id}_{\mathcal{I}} \otimes \varphi$.
\end{proof}
As a corollary, we can consider the diagram
\[
\xymatrix@R=30pt@C=50pt{\relax
  \mathrm{D}^-(X)  \ar@/^5pc/[rr]^-{\mathrm{id}_{\mathrm{D}^-(X)}}  \ar@/_5pc/[rr]_-{\mathcal{I} [2] \, \lltens \, *}  \ar[r]^-{j_*} & \mathrm{D}^{\mathrm{adm}}(S) \UN[r]{\mathrm{Tor}^0_{\mathcal{O}_S}(*, \,\mathcal{O}_X)\,\,\,}{\mathrm{Tor}^1_{\mathcal{O}_S}(*, \,\mathcal{O}_X)[2]\, \, \,}{\Theta} & \mathrm{D}^-(X)}
\]
which gives a natural transformation $\Theta_{j_* (\star)} \colon \mathrm{Id} \longrightarrow \mathcal{I}\,[2]  \otimes \star $ between endofunctors of $\mathrm{D}^-(X)$. For any object $\mathcal{V}_{\bullet}$ of $\mathrm{D}^{-}(X)$, the corresponding morphism of $\mathrm{D}^-(X)$ is $\Theta_{j_*\mathcal{V}_{\bullet}}$. By a slight abuse of notation, we will denote it by $\Theta_{\mathcal{V}_{\bullet}}$. Remark that Proposition \ref{link} gives an explicit description of $\Theta_{\mathcal{V}_{\bullet}}$.
\par \medskip
Let us give a few properties of the endofunctor $\mathbb{L}j^*j_*$ of $\mathrm{D}^-(X)$.
\begin{enumerate}
\vspace{0.2cm}
\item[--]
The functor $\mathbb{L}j^*j_*$ is continuous (i.e. commutes with arbitrary limits and colimits).
\vspace{0.2cm} 
\item[--]
Given a nonzero complex $\mathcal{V}_{\bullet}$ in $\mathrm{D}^{\mathrm{b}}(X)$, $\mathbb{L}j^*(j_*\mathcal{V}_{\bullet})$ is \textit{never} bounded.
\vspace{0.2cm}
\item[--]
The functor $\mathbb{L}j^*j_*$ carries a natural lax monoidal structure: if $\mathcal{V}_{\bullet}$ and $\mathcal{W}_{\bullet}$ are sheaves on $\mathcal{O}_X$-modules, the product map 
is given by the composition
\[
\xymatrix{
j^*j_* \mathcal{V}_{\bullet} \, \lltens{}_{\mathcal{O}_X} \, j^*j_* \mathcal{W}_{\bullet} \ar[d]^-{\wr} \\
 j^*(j_* \mathcal{V}_{\bullet} \, \lltens{}_{\mathcal{O}_Y} \, j_*\mathcal{W}_{\bullet})\ar[d] &  \\
 j^* j_* ( \mathcal{V}_{\bullet} \,\lltens{}_{\mathcal{O}_X}\, j^* j_*\mathcal{W}_{\bullet}) \ar[r]& j^*j_*(\mathcal{V}_{\bullet} \,\lltens{}_{\mathcal{O}_X}\, \mathcal{W}_{\bullet}) \\
}
\]
and the unit is $
\mathcal{O}_X \simeq j^* \mathcal{O}_S \simeq \mathbb{L}j^* \mathcal{O}_S \longrightarrow \mathbb{L}j^* (j_* \mathcal{O}_X).
$
\vspace{0.2cm}
\item[--] The ring object $\mathbb{L}j^*(j_*\mathcal{O}_X)$ in $\mathrm{D}^{-}(X)$ is the structural sheaf of the derived intersection of $X$ in $S$.
\end{enumerate}
In the single sheaf case, we can provide a simple formality criterion under some additional hypotheses:
\begin{theorem} \label{sanaa}
Let $\mathcal{V}$ be a coherent sheaf of $\mathcal{O}_X$-modules which is not a torsion sheaf. Then the following properties are equivalent:
\begin{enumerate}
\vspace{0.2cm}
\item[(i)] $\mathbb{L}j^*(j_*\mathcal{V})$ is formal in $\mathrm{D}^-(X)$.
\vspace{0.2cm}
\item[(ii)] $\tau^{\geq -2}\, \mathbb{L}j^*(j_*\mathcal{V})$ is formal in $\mathrm{D}^{\mathrm{b}}(X)$.
\vspace{0.2cm}
\item[(iii)] $\Theta_{\mathcal{V}}$ and $\Theta_{\mathcal{I}}$ vanish.
 \vspace{0.2cm}
\end{enumerate}
\begin{proof}
 $ $ \par
\vspace{0.2cm}
$(\textrm{i}) \longrightarrow (\textrm{ii})$ Obvious.
\par
\vspace{0.2cm}
$(\textrm{ii}) \longrightarrow (\textrm{iii})$ If $\tau^{\geq 2} \, $ is formal, then Theorem \ref{yek} implies that $\Theta_{\mathcal{V}}$ vanishes and that 
\[
\mathbb{L}j^*(j_*\mathcal{V}) \simeq \mathcal{V} \oplus \mathbb{L}j^*( j_*(\mathcal{I} \otimes \mathcal{V}))[1].
\]
Hence $\tau^{\geq -1} \, \mathbb{L}j^*(j_*(\mathcal{I} \otimes \mathcal{V}))$ is formal, so  $\Theta_{\mathcal{I} \otimes \mathcal{V}}$ vanishes. Using Corollary \ref{mononoke}, the derived trace of $\Theta_{\mathcal{I} \otimes \mathcal{V}}$ with respect to the factor $\mathcal{V}$ is $r \times \Theta_{\mathcal{I}}$ where $r$ is the generic rank of $\mathcal{V}$. As $r$ is nonzero, $\Theta_{\mathcal{I}}$ vanishes.
\par
\vspace{0.2cm}
$(\textrm{iii}) \longrightarrow (\textrm{i})$ If $\Theta_{\mathcal{I}}$ and $\Theta_{\mathcal{V}}$ vanish, then all the classes $\Theta_{\mathcal{I}^{\otimes p} \otimes \mathcal{V}}$ vanish, so $\mathbb{L}j^*(j_*\mathcal{V})$ is formal thanks to Corollary \ref{hmpf}.
\end{proof}
\end{theorem}
Assume that the thickening $S$ is globally trivial, that is the morphism $j$ admits a global retraction $\sigma \colon S \longrightarrow X$. Then every complex of sheaves $\mathcal{V}_{\bullet}$ in $\mathrm{C}^{-}(X)$ admits an admissible resolution $K_{\mathcal{V}_{\bullet}}$, which is $\sigma^* \mathcal{V}_{\bullet} \otimes K_{\mathcal{O}_X}$ where $K_{\mathcal{O}_X}$ is the complex
\[
\cdots \longrightarrow  \sigma^* {\mathcal{I}}^{\otimes 3} \longrightarrow  \sigma^* {\mathcal{I}}^{\otimes 2} \longrightarrow \sigma^*{\mathcal{I}}  \longrightarrow \mathcal{O}_S.
\]
This gives a distinguished HKR isomorphism
\[
\Gamma_{\sigma} \colon \mathbb{L}j^*(j_*\mathcal{V}_{\bullet}) \xlongrightarrow{\sim} \bigoplus_{p \geq 0} \,\mathcal{I}^{\otimes p} \otimes \mathcal{V}_{\bullet} \, [p]
\]
in $\mathrm{D}^-(X)$. Besides, via this isomorphism, the lax monoidal structure on $\mathbb{L}j^*j_*$ is simply given by the shuffle product (\textit{see} \cite[Proposition 1.10]{Arinkin-Caldararu}).
We now come back to the general case, and set the following definition:
\begin{definition}
The Arinkin-C\u{a}ld\u{a}raru functor is the element $H$ of $\mathrm{EndFct}_{\mathrm{dg}}^*(\mathrm{C}^{\mathrm{b}}(X))$ defined by
\[
H(\mathcal{V_{\bullet}})=\mathrm{cone}\, (\Omega^1_X \otimes \mathcal{V}_{\bullet} \longrightarrow \mathrm{P}^1_X(\mathcal{V}_{\bullet})).
\]
\end{definition}
Thanks to Propositions \ref{chaberton} and \ref{brique}, the functor $H$ is exact, and is naturally a lax monoidal functor of $\mathrm{C}^{\mathrm{b}}(X)$. Hence:
\begin{enumerate}
\item[--] According to Theorem \ref{canal} (ii), the functors $H^{[n]}$ and $H^{[[n]]}$ are exact and bounded. They are also naturally lax monoidal functors thanks to Proposition \ref{morita2}.
\vspace{0.2cm}
\item[--] By Proposition \ref{morita2}, the natural morphism from $H^{[n]}$ to $H^{[[n]]}$ is multiplicative. Theorem \ref{canal} (ii) implies that this morphism is a quasi-isomorphism.
\vspace{0.2cm}
\item[--] All these structures extend on $\mathrm{C}^-(X)$, and can be defined on $\mathrm{D}^-(X)$ using flat resolutions.
\vspace{0.2cm}
\end{enumerate}
\par \medskip
Let us discuss the case where $S$ is a globally trivial thickening of $X$. Let $T$ be the element of $\mathrm{EndFct}^*_{\mathrm{dg}}(\mathrm{C}^{\mathrm{b}}(X))$ defined by
\[
T(\mathcal{V}_{\bullet})=\mathcal{I} \otimes \mathcal{V}_{\bullet} \,[1] \,\oplus \, \mathcal{V}_{\bullet}.
\]
Thanks to Proposition \ref{since},
\begin{enumerate}
\vspace{0.2cm}
\item[--] $T^{[n]}$ is naturally isomorphic to $\mathcal{V}_{\bullet} \longrightarrow \bigoplus_{p \geq 0} \,\mathcal{I}^{\otimes p} \otimes \mathcal{V}_{\bullet} \,[p]$.
\vspace{0.2cm}
\item[--] The natural map from $T^{[n]}$ to ${T}^{[[n]]}$ is a quasi-isomorphism.
\vspace{0.2cm}
\end{enumerate}
\begin{proposition}
If there exists a global retraction $\sigma \colon S \longrightarrow X$ and if $\,U=\mathrm{cone}\,\left(  \mathrm{id}_{\mathrm{C}^{\mathrm{b}}(X)} \longrightarrow  \mathrm{id}_{\mathrm{C}^{\mathrm{b}}(\mathcal{C})} \right)$, then there is a natural exact sequence
\[
0 \longrightarrow U(\Omega^1_X \otimes \mathcal{V}_{\bullet}) \longrightarrow H(\mathcal{V}_{\bullet}) \longrightarrow T(\mathcal{V}_{\bullet}) \longrightarrow 0 
\]
of dg-endofunctors of $\mathrm{C}^{\mathrm{b}}(X)$, and for any nonnegative integer $n$ the map from $H^{[n]}$ to $T^{[n]}$ is a quasi-isomorphism.
\end{proposition}

\begin{proof}
The first part follows directly from the exact sequence
\[
0 \longrightarrow \Omega^1 \otimes \mathcal{V}_{\bullet} \longrightarrow \mathrm{P}^1_S(\mathcal{V}_{\bullet}) \longrightarrow \sigma^* \mathcal{V}_{\bullet} \longrightarrow 0
\]
obtained in the local case in the proof of Theorem \ref{wazomba}. For the second part we have a commutative diagram
\[
\xymatrix{
H^{[n]} \ar[r] \ar[d]^-{\wr} & T^{[n]} \ar[d]^-{\wr} \\
H^{[[n]]} \ar[r]^-{\sim} & T^{[[n]]}
}
\]
where the bottom horizontal map is a quasi-isomorphism because of Proposition \ref{morita2}. Hence the top horizontal map is a quasi-isomorphism.
\end{proof}
\subsection{Main theorem}
In this section, we compute explicitly the functor $\mathbb{L}j^*j_*$. Consider the endofunctor $\mu \colon \mathcal{K}_{\bullet} \longrightarrow \mathrm{cone}\, \, (\Omega_S^1 \otimes_{\mathcal{O}_S} \mathcal{K}_{\bullet} \longrightarrow \mathrm{P}^1_S(\mathcal{K}_{\bullet})) $ of $\mathrm{C}^{\mathrm{b}}(S)$ (\textit{see} Definition \ref{mu} for the local case); it is a bounded dg-endofunctor of $\mathrm{C}^{\mathrm{b}}(S)$. There is a natural morphism
\[
j^* \circ  {\mu} \circ j_* \longrightarrow j^* \circ (j_*  \circ j^*) \circ{\mu} \circ j_* \xlongrightarrow{\sim}  {{H}}
\]
of dg-endofunctors of $\mathrm{C}^{{\mathrm{b}}}(X)$. For any nonnegative integer $n$, we define a natural transformation $\chi_n \colon \mathbb{L}j^* j_* \longrightarrow H^{[[n]]}$ as follows: for any complex $\mathcal{V}_{\bullet}$ in $\mathrm{C}^{-}(X)$ the morphism $\chi_n(\mathcal{V}_{\bullet})$ is the composition.
\[
\mathbb{L}j^*(j_*\mathcal{V}_{\bullet}) \xleftarrow{\sim}  \mathbb{L} j^* \mu^{[n]} (j_* \mathcal{V}_{\bullet}) \longrightarrow j^* \,\mu^{[n]} (j_* \mathcal{V}_{\bullet}) \longrightarrow H^{[n]}(\mathcal{V}_{\bullet}).
\]
\begin{theorem} \label{winner}
Assume to be given a pair $(X, S)$ where $X$ is either a smooth scheme over a field of characteristic zero or a complex manifold, and $S$ is a locally trivial infinitesimal thickening of $S$. Then the following properties are valid:
\begin{enumerate}
\vspace{0.1cm}
\item[(i)] The morphisms $\chi_n \colon \mathbb{L}j^* j_* \longrightarrow H^{[[n]]}$ are multiplicative.
\vspace{0.2cm}
\item[(ii)] For any complex $\mathcal{V}_{\bullet}$ concentrated in negative degrees, the local homology morphism $\mathcal{H}_p(\chi_n(\mathcal{V}_{\bullet}))$ is an isomorphism for $0 \leq p \leq n$.
\vspace{0.2cm}
\item[(iii)] The sequence of morphisms $(\chi_n)_{n \geq 0}$ define a multiplicative isomorphism
\[
\mathbb{L}j^* j_* \simeq \underset{n}{\varprojlim} \, H^{[n]}.
\]
\item[(iv)] If $S$ is globally trivial and $\sigma$ is an associated retraction of $j$, the composition
\[
\mathbb{L}j^* j_* \simeq \underset{n}{\varprojlim} \, H^{[n]} \longrightarrow  \underset{n}{\varprojlim} \, T^{[n]} \simeq \bigoplus_{p \geq 0} \, \mathcal{I}^{\otimes p} \otimes (\star)\, [p]
\]
is the generalized HKR isomorphism $\Gamma_{\sigma}$.
\vspace{0.2cm}
\end{enumerate}
\end{theorem}

\begin{proof} $ $ \par
\vspace{0.1cm}
$(\textrm{i})$ The functor $\mu^{[n]}$ is a lax multiplicative endofunctor of $\mathcal{C}^{-}(S)$, and the morphism from $\mu^{[n]}$ to $\mathrm{id}_{C^{-}(S)}$ is also multiplicative. If $P_{\mathcal{V}_{\bullet}}$ and $P_{\mathcal{W}_{\bullet}}$ are two flat resolutions over $X$ of complexes $\mathcal{V}_{\bullet}$ and $\mathcal{W}_{\bullet}$ in $\mathrm{C}^-(X)$, we have a commutative diagram
\begingroup \tiny
\[
\scalebox{0.95}{$
\xymatrix{
\mu^{[n]}(j_* P_{\mathcal{V}_{\bullet}}) \, \lltens{}_{\mathcal{O}_Y} \, \mu^{[n]}(j_*P_{\mathcal{W}_{\bullet}}) \ar[r] \ar[d] &\mu^{[n]} (j_*P_{\mathcal{V}_{\bullet}}) \otimes_{\mathcal{O}_Y} \mu^{[n]} (j_* P_{\mathcal{W}_{\bullet}}) \ar[r] \ar[d] & \mu^{[n]} (j_* P_{\mathcal{V}_{\bullet}} \otimes_{\mathcal{O}_Y} j_* P_{\mathcal{W}_{\bullet}})  \ar[d]\\
j_*P_{\mathcal{V}_{\bullet}} \, \lltens{}_{\mathcal{O}_Y} \, j_* P_{\mathcal{W}_{\bullet}} \ar[r] & j_*P_{\mathcal{V}_{\bullet}} \otimes_{\mathcal{O}_Y} j_*P_{\mathcal{W}_{\bullet}}  \ar@{=}[r] & j_* P_{\mathcal{V}_{\bullet}} \otimes_{\mathcal{O}_Y} j_*P_{\mathcal{W}_{\bullet}} 
}$}
\]
\endgroup
The morphism from $j^* \circ \mu^{[n]} \circ j_*$ to $H^{[n]}$ being multiplicative, we get a commutative diagram
\begingroup \tiny
\[
\scalebox{0.85}{$
\xymatrix{ \mathbb{L}j^*(j_*\mathcal{V}_{\bullet}) \, \lltens{} \,\mathbb{L}j^*(j_* \mathcal{W}) \ar[rr] && \mathbb{L}j^*(j_* \mathcal{V}_{\bullet} \,\lltens{}\, \mathcal{W}_{\bullet}) \\
\mathbb{L}j^* \mu^{[n]}(j_* P_{\mathcal{V}_{\bullet}}) \, \lltens{} \, \mathbb{L}j^* \mu^{[n]}(j_* P_{\mathcal{W}_{\bullet}}) \ar[r]\ar[d]\ar[u]^ -{\wr}  & \mathbb{L}j^* \mu^{[n]}(j_* P_{\mathcal{V}_{\bullet}}) \, \otimes \, \mathbb{L}j^* \mu^{[n]}(j_* P_{\mathcal{W}_{\bullet}}) \ar[d]\ar[r] & \mathbb{L}j^* \mu^{[n]} (j_* P_{\mathcal{V}_{\bullet}} \otimes_{\mathcal{O}_S} j_* P_{\mathcal{W}_{\bullet}}) \ar[u]^ -{\wr}\ar[d]  \\
j^* \mu^{[n]}(j_* P_{\mathcal{V}_{\bullet}}) \, \lltens{} \, \mu^{[n]}(j_*P_{\mathcal{W}_{\bullet}}) \ar[r] \ar[d] & j^* \mu^{[n]}(j_* P_{\mathcal{V}_{\bullet}}) \, \otimes \, j^* \mu^{[n]}(j_*P_{\mathcal{W}_{\bullet}}) \ar[r] \ar[d] & j^* \mu^{[n]} (j_* P_{\mathcal{V}_{\bullet}} \otimes_{\mathcal{O}_S} j_* P_{\mathcal{W}_{\bullet}}) \ar[d] \\
H^{[n]}(P_{\mathcal{V}_{\bullet}}) \, \lltens{} \, H^{[n]}(P_{\mathcal{W}_{\bullet}}) \ar[r] & H^{[n]}(P_{\mathcal{V}_{\bullet}}) \, \otimes \, H^{[n]}(P_{\mathcal{W}_{\bullet}}) \ar[r] & H^{[n]} (P_{\mathcal{V}_{\bullet}} \otimes P_{\mathcal{W}_{\bullet}})
}$}
\]
\endgroup
which proves that $\chi_n$ is multiplicative.
\par \medskip
$(\textrm{ii})$ As $\mathbb{L}j^* j_*$ is continuous, we can assume without loss of generality that $\mathcal{V}_{\bullet}$ is bounded. Let us denote by $\widetilde{\Theta}$ the natural morphism from $\mathrm{id}_{\mathrm{C}^{\mathrm{b}}(S)}$ to $\widetilde{\mu}=\mathrm{cone} \, \mu$. There is a natural morphism
\[
\widetilde{\mu} \circ j_* \longrightarrow j_*  \circ j^* \circ \widetilde{\mu} \circ j_* \xlongrightarrow{\sim} j_* \circ  {\widetilde{H}}
\]
of dg-functors from $\mathrm{C}^{{\mathrm{b}}}(X)$ to $\mathrm{C}^{\mathrm{b}}(S)$. For any nonnegative integer $n$, this gives a morphism
\[
\alpha_n \colon \widetilde{\mu}^{n} \circ j_* \simeq \widetilde{\mu} \circ \widetilde{\mu}^{n-1} \circ j_*  \longrightarrow \widetilde{\mu} \circ j_* \circ {\widetilde{H}}^{n-1}.
\]
Now we have a commutative diagram
\[
\xymatrix{
\widetilde{\mu}^n \circ j_*  \ar[d]_-{\alpha_n} \ar[r]^-{S_n^{\widetilde{\mu}} \,\circ \, j_*} & \widetilde{\mu}^{n+1} \circ j_* \ar[d]^-{\alpha_{n+1}}\\
\widetilde{\mu} \circ j_* \circ {\widetilde{H}}^{n-1} \ar[r]_-{W_n} & \widetilde{\mu} \circ j_* \circ {\widetilde{H}}^n
}
\]
where the morphism $S_n^{\widetilde{\mu}}$ is the alternated sum defined by \eqref{sum}\footnote{Since we are going to use the construction of \S \ref{natanz} for the pairs $(\widetilde{\mu}, \widetilde{\Theta})$ and $({{\widetilde{H}}}, \Theta)$, we put a superscript to distinguish them.}, and the bottom morphism  $W_n$ is the sum of 
$\widetilde{\mu} \circ j_* \,(S_{n-1}^{\widetilde{H}})$ and of the morphism obtained as the composition
\[
\widetilde{\mu} \circ j_* \circ {\widetilde{H}}^{n-1} \longrightarrow j_* \circ \widetilde{H}^n  \xlongrightarrow{\widetilde{\Theta}_{j_* \circ {\widetilde{H}}^{n}}} \widetilde{\mu} \circ j_* \circ {\widetilde{H}}^{n}.
\]
Let $\mathfrak{F}_n$ denote the dg-functor from $\mathrm{C}^{{\mathrm{b}}}(X)$ to $\mathrm{C}^{{\mathrm{b}}}(S)$ defined as the iterated cone of the complex
\[
\xymatrix@C=30pt{
j_* \ar[r]^-{W_0} & \widetilde{\mu} \circ j_* \ar[r]^-{W_1} & \widetilde{\mu} \circ j_* \circ {\widetilde{H}} \ar[r] & \cdots  \ar[r]^-{W_{n-1}}& \widetilde{\mu} \circ j_* \circ {\widetilde{H}}^{n-1}
}.
\]
Thanks to the previous discussion, we have a chain of natural transformations
\[
\widetilde{\mu}^{[n]} \circ j_* \longrightarrow \widetilde{\mu}^{[[n]]} \circ j_* \longrightarrow \mathfrak{F}_n.
\]
For any complex $\mathcal{V}_{\bullet}$ in $\mathrm{C}^{{\mathrm{b}}}(X)$ concentrated in nonpositive degrees, the corresponding morphisms
\begin{equation} \label{roulette}
\widetilde{\mu}^{[n]}(j_*\mathcal{V}_{\bullet}) \longrightarrow \widetilde{\mu}^{[[n]]}(j_* \mathcal{V}_{\bullet}) \longrightarrow \mathfrak{F}_n (\mathcal{V}_{\bullet})
\end{equation}
are all quasi-isomorphisms. Indeed, we have a diagram
\[
\xymatrix{\widetilde{\mu}^{[n]}(j_*\mathcal{V}_{\bullet}) \ar[r] \ar[d]^-{\wr} & \widetilde{\mu}^{[[n]]}(j_* \mathcal{V}_{\bullet}) \ar[d] \ar[ld]_-{\sim} \\
j_* \mathcal{V}_{\bullet} & \mathfrak{F}_n (\mathcal{V}_{\bullet}) \ar[l]_-{\sim}
}
\]
where the left vertical arrow is an isomorphism thanks to Theorem \ref{loire}, and the diagonal arrow is an isomorphism since $\widetilde{\mu}$ is quasi-isomorphic to zero. We now claim that $\mathfrak{F}_n(\mathcal{V}_{\bullet})$ is an n-admissible complex, which is a purely local problem. Hence we can assume that $S$ is globally trivial. If $\widetilde{T}$ is the shift functor defined by\footnote{This definition doesn't match with Definition \ref{definitif} when $T$ is the functor $\mathcal{I}[2] \otimes  \star \,$, however it differs from it by an element in the center of $\mathrm{EndFct}_{\mathrm{dg}}(\mathrm{C}^{{\mathrm{b}}}(X))$ whose image lies in null homotopic complexes.} 
\[
\widetilde{T} (\mathcal{V}_{\bullet})=\mathcal{I} \otimes \mathcal{V}_{\bullet} [2]
\]
then there is a natural quasi-isomorphism from $\widetilde{H}$ to $\widetilde{T}$, and the diagram
\[
\xymatrix{
&\mathrm{id}_{\mathrm{C}^{{\mathrm{b}}}(X)} \ar[ld]_-{\Theta} \ar[rd]^-{0}& \\
\widetilde{H} \ar[rr] &&\widetilde{T} 
}
\]
commutes. Hence, if $\mathfrak{F}_n^{\mathrm{loc}}$ is the iterated cone of the complex
\[
\xymatrix@C=30pt{
j_* \ar[r]^-{W_0^{\mathrm{loc}}} & \widetilde{\mu} \circ j_* \ar[r]^-{W_1^{\mathrm{loc}}} & \widetilde{\mu} \circ j_* \circ {\widetilde{T}} \ar[r] & \cdots  \ar[r]^-{W_{n-1}^{\mathrm{loc}}} & \widetilde{\mu} \circ j_* \circ {\widetilde{T}}^{n-1}
}
\]
where $W_p^{\mathrm{loc}}$ is given by the composition
\[
W_p^{\mathrm{loc}} \colon \widetilde{\mu} \circ j_* \circ \widetilde{T}^{p-1} \longrightarrow j_* \circ  \widetilde{H} \circ \widetilde{T}^{p-1} \longrightarrow j_* \circ \widetilde{T}^{p} \xlongrightarrow{\widetilde{\Theta}_{j_* \circ \widetilde{T}^{p}}} \widetilde{\mu} \circ j_* \circ \widetilde{T}^{p},
\]
there is a natural quasi-isomorphism from $\mathfrak{F}_n(\mathcal{V}_{\bullet})$ to $\mathfrak{F}_n^{\mathrm{loc}}(\mathcal{V}_{\bullet})$.
Besides, thanks to Lemma \ref{tree}, all arrows from $ \widetilde{\mu} \circ j_* \circ \widetilde{H}^{p}$ to $\widetilde{\mu} \circ j_* \circ \widetilde{T}^{p}$ induce isomorphisms in the admissible derived category $\mathrm{D}^{\mathrm{adm}}(S)$, so the morphism from $\mathfrak{F}_n(\mathcal{V}_{\bullet})$ to $\mathfrak{F}_n^{\mathrm{loc}}(\mathcal{V}_{\bullet})$ is an isomorphism in $\mathrm{D}^{\mathrm{adm}}(S)$. Hence, the claim is equivalent to the fact that $\mathfrak{F}_n^{\mathrm{loc}}(\mathcal{V}_{\bullet})$ is $n$-admissible.
\par \medskip
For any complex $\mathcal{K}_{\bullet}$ of $\mathcal{O}_X$-modules, thanks to Proposition \ref{coince} and to the isomorphism $\mathrm{Tor}^1_S (j_* \mathcal{K}_{\bullet},  \mathcal{O}_X) \simeq \mathcal{I} \otimes \mathcal{K}_{\bullet} $, we have two commutative diagrams
\begingroup \scriptsize
\[
\xymatrix{
\mathrm{Tor}^1_S(\widetilde{\mu} \!\circ \!j_* (\mathcal{K}_{\bullet}), \mathcal{O}_X) \ar[r] \ar[d]^-{\wr} & \mathrm{Tor}^1_S(j_* \!\circ\! \widetilde{H}(\mathcal{K}_{\bullet}), \mathcal{O}_X) \ar[r] & \mathrm{Tor}^1_S(j_* \!\circ \!\widetilde{T}(\mathcal{K}_{\bullet}) , \mathcal{O}_X) \ar[d]^-{\wr}\\
\mathcal{I}^{\otimes 2} \otimes \mathcal{K}_{\bullet} [2] \oplus  \mathcal{I} \otimes \mathcal{K}_{\bullet}\ar[rr]^-{\begin{pmatrix}
1 & 0
\end{pmatrix} } && \mathcal{I}^{\otimes 2} \otimes \mathcal{K}_{\bullet} [2]
}
\]
\endgroup
and 
\[
\xymatrix@C=50pt@R=40pt{
\mathrm{Tor}^1_S(j_* \mathcal{K}_{\bullet} , \mathcal{O}_X) \ar[d]^-{\wr} \ar[r]^-{\widetilde{\Theta}_{j_* \mathcal{K}_{\bullet}}} & \mathrm{Tor}^1_S(\widetilde{\mu} \circ j_* (\mathcal{K}_{\bullet}) , \mathcal{O}_X) \ar[d]^-{\wr}\\
\mathcal{I} \otimes \mathcal{K}_{\bullet} \ar[r]^-{\tiny{\begin{pmatrix}
0 \\ 
1
\end{pmatrix}}} & \mathcal{I}^{\otimes 2} \otimes \mathcal{K}_{\bullet} [2] \oplus \mathcal{I} \otimes \mathcal{K}_{\bullet}
}
\]
This gives the diagram
\begingroup \scriptsize
\[
\xymatrix{
\mathrm{Tor}^1_S (\widetilde{\mu} \circ j_* \circ \widetilde{T}^{p-1}(\mathcal{V}_{\bullet}), \mathcal{O}_X) \ar[d]^-{\wr} \ar[r]^-{\mathrm{Tor}^1_S(W_p^{\mathrm{loc}}, \mathcal{O}_X)} & \mathrm{Tor}^1_S (\widetilde{\mu} \circ j_* \circ \widetilde{T}^{p}(\mathcal{V}_{\bullet}), \mathcal{O}_X) \ar[d]^-{\wr}\\
\mathcal{I}^{\otimes p+1} \otimes \mathcal{V}_{\bullet} [2p] \oplus \mathcal{I}^{\otimes p} \otimes \mathcal{V}_{\bullet} [2p-2] \ar[r]^-{\begin{pmatrix}
0 & 0 \\ 
1 & 0
\end{pmatrix} } & \mathcal{I}^{\otimes p+2} \otimes \mathcal{V}_{\bullet} [2p+2] \oplus \mathcal{I}^{\otimes p+1} \otimes \mathcal{V}_{\bullet} [2p]
}
\]
\endgroup
and we get
\[
\mathrm{Tor}^1_S(\mathfrak{F}_n^{\mathrm{loc}}(\mathcal{V}_{\bullet}), \mathcal{O}_X) \simeq \mathcal{I}^{\otimes n+1} \otimes \mathcal{V}_{\bullet} [n]
\]
so that $\mathfrak{F}_n^{\mathrm{loc}}(\mathcal{V}_{\bullet})$ is $n$-admissible. This proves the claim.
The chain of morphisms \eqref{roulette} induces a commutative diagram
\[
\xymatrix{j^* {\mu}^{[n]}(j_*\mathcal{V}_{\bullet}) \ar@/^3pc/[rr] ^-{\sim} \ar[r] \ar[d] &  j^* {\mu}^{[[n]]}(j_* \mathcal{V}_{\bullet}) \ar[r] \ar[d] &  j^*\mathfrak{F}_n (\mathcal{V}_{\bullet}) \ar@{=}[d] \\
H^{[n]}(\mathcal{V}_{\bullet}) \ar[r]^-{\sim} & H^{[[n]]}(\mathcal{V}_{\bullet}) \ar@{=}[r] & H^{[[n]]}(\mathcal{V}_{\bullet}) 
}
\]
where the top horizontal row is an isomorphism as ${\mu}^{[n]}(\mathcal{V}_{\bullet})$ and $\mathfrak{F}_n(\mathcal{V}_{\bullet})$ are both $n$-admissible, and the bottom horizontal row is an isomorphism thanks to Theorem \ref{canal}.
\par \medskip
$(\mathrm{iii})$ follows directly from $(\mathrm{ii})$.
\par \medskip
$(\mathrm{iv})$ Let $Z$ be the dg-functor from $\mathrm{C}^{\mathrm{b}}(X)$ to $\mathrm{C}^{\mathrm{b}}(S)$ defined by 
\[
Z(V_{\bullet})=\mathrm{cone}\, (\mathcal{I} \otimes \mathcal{V}_{\bullet} \longrightarrow \sigma^* \mathcal{V}_{\bullet}).
\]
There is a natural morphism $\mu \circ j_* $ to $Z$. If $\widetilde{Z}=\mathrm{cone}\, (S \longrightarrow j_*)$, there is a natural morphism $\widetilde{Z} \longrightarrow \widetilde{T}$ such that the diagram
\[
\xymatrix{\widetilde{\mu} \circ j_* \ar[rr]\ar[rd]&& \widetilde{Z} \ar[ld] \\
& \widetilde{T}&}
\]
commutes. Hence we get another commutative diagram
\[
\xymatrix{
\widetilde{\mu} \circ j_* \ar[r] \ar[d] & j_* \circ \widetilde{H} \ar[r] \ar[d] &  j_* \circ \widetilde{T} \ar[r]\ar@{=}[d]& 
\widetilde{\mu} \circ \widetilde{T} \ar[d] \\
\ar@/_2pc/[rrr] _-{\gamma}\widetilde{Z} \ar[r] &  j_* \circ \widetilde{T} \ar@{=}[r]&  j_* \circ \widetilde{T} \ar[r]&
\widetilde{Z} \circ \widetilde{T}
}
\]
where for any $\mathcal{V}_{\bullet}$, $\gamma_{\mathcal{V}_{\bullet}}$ is the composition
\[
\xymatrix{ &&\mathcal{I} \otimes \mathcal{V}_{\bullet} \ar[r] \ar[d]^-{\mathrm{id}} & \sigma^* \mathcal{V}_{\bullet} \ar[r] & \mathcal{V}_{\bullet}\\
\mathcal{I}^{\otimes 2} \otimes  \mathcal{V}_{\bullet} \ar[r] &\sigma^* (\mathcal{I} \otimes \mathcal{V}_{\bullet}) \ar[r] & \mathcal{I} \otimes \mathcal{V}_{\bullet}
}
\] 
If $\mathfrak{K}_n$ is the functor from $\mathrm{C}^{\mathrm{b}}(X)$ to $\mathrm{C}^{\mathrm{b}}(S)$ defined as the iterated cone of the functors
\[
\xymatrix@C=30pt{
j_* \ar[r]& \widetilde{Z} \ar[r]^{\gamma}& \widetilde{Z} \circ \widetilde{T} \ar[r]^ {\gamma_{\widetilde{T}}} & \cdots \ar[r]^ -{\gamma_{\widetilde{T}^{n-2}}}& \widetilde{Z} \circ \widetilde{T}^{n-1},
}
\]
then there is a natural morphism from $\mathfrak{F}_n^{\mathrm{loc}}$ to $\mathfrak{K}_n$, and if $\Gamma_{\sigma}^{(n)}$ is the composition
\[
\mathbb{L}j^* (j_*\mathcal{V}_{\bullet}) \xlongrightarrow{\Gamma_{\sigma}} \bigoplus_{p \geq 0} \mathcal{I}^{\otimes p} \otimes \mathcal{V}_{\bullet} \, [p] \longrightarrow \bigoplus_{p=0}^n \mathcal{I}^{\otimes p} \otimes \mathcal{V}_{\bullet} \, [p] = T^{[n]}(\mathcal{V}_{\bullet})
\]
we have a commutative diagram
\[
\xymatrix@C=50pt{
H^{[n]}\ar@/^1pc/[rrd]^-{\sim} &&\\
&j^* \mathfrak{F}_n^{\mathrm{loc}} \ar[r] \ar[dd]&H^{[[n]]} \ar[dd]^-{\sim} \\
\mathbb{L}j^* j_* \ar[uu]_-{\chi_n} \ar[dd]^-{\Gamma_{\sigma}^{(n)}} \ar[ru] \ar[rd] && \\
&j^* \mathfrak{K}_n \ar[r] &T^{[[n]]}\\
T^{[n]}\ar@/_1pc/[rru]_-{\sim} &&
}
\]
This finishes the proof.
\end{proof}

\appendix

\section{Multiplicativity of principal parts}
Let $X$ be a smooth scheme over a field of characteristic zero (or a complex manifold). Our aim is to prove that the principal parts functor $\mathrm{P}^1_X$ is naturally a lax monoidal functor. Although we didn't find explicitly the material of this section in the literature, the method we use can be found in a slightly different form in \cite{Markarian}.
\par \medskip
Let $D$ be the diagonal in $X^2$, let $W$ be the subscheme of $X^2$ defined by \[
\mathcal{O}_W=\mathcal{O}_{X^2}/{\mathcal{I}_D^2},
\]
let $p_i$ the two projections from $X^2$ to $X$, and let $q_i$ (resp. $q_{ij}$) be the three projections from $X^3$ to $X$ (resp. from $X^3$ to $X^2$). Then for any sheaves $\mathcal{F}$ and $\mathcal{G}$ of $\mathcal{O}_X$-modules, 
\begin{align*}
\mathrm{P}_X^1(\mathcal{F}) \otimes \mathrm{P}^1_X(\mathcal{G})&=p_{1*} (\mathcal{O}_{W} \otimes p_2^* \mathcal{F}) \otimes \mathrm{P}^1_X(\mathcal{G})\\
&=p_{1*} (\mathcal{O}_{W} \otimes p_2^* \mathcal{F} \otimes p_1^* \,\mathrm{P}^1_X(\mathcal{G})) \\ 
&=p_{1*} (\mathcal{O}_{W} \otimes p_2^* \mathcal{F} \otimes p_1^* \,p_{1*}(\mathcal{O}_{W} \otimes p_2^* \mathcal{G}))\\
&=p_{1*} (\mathcal{O}_{W} \otimes p_2^* \mathcal{F} \otimes q_{12*} \, q_{13}^*(\mathcal{O}_{W} \otimes p_2^* \mathcal{G}))\\
&=q_{1*}(q_{12}^* \,\mathcal{O}_W \otimes q_{13}^* \,\mathcal{O}_W \otimes q_2^* \mathcal{F} \otimes q_3^* \mathcal{G}).
\end{align*}
Let $\delta \colon X \times X \longrightarrow X \times X^2=X^3$ be the partial diagonal injection on the two last factors of $X^3$ given by $\delta(x_1, x_2)=(x_1, x_2, x_2)$, and let $T$ be the image of $\delta$. Then we get a morphism
\begin{align*}
\mathrm{P}_X^1(\mathcal{F}) \otimes \mathrm{P}_X^1(\mathcal{G}) & \longrightarrow q_{1*}(q_{12}^* \,\mathcal{O}_W \otimes q_{13}^* \,\mathcal{O}_W \otimes  \mathcal{O}_T \otimes q_2^* \mathcal{F} \otimes q_3^* \mathcal{G})\\
&=q_{1*}(q_{12}^* \,\mathcal{O}_W \otimes q_{13}^* \,\mathcal{O}_W \otimes  \delta_* \mathcal{O}_{X^2} \otimes q_2^* \mathcal{F} \otimes q_3^* \mathcal{G})\\
&=q_{1*} \delta_{*}(\delta^*(q_{12}^* \,\mathcal{O}_W \otimes q_{13}^* \,\mathcal{O}_W) \otimes \delta^*(q_2^* \mathcal{F} \otimes q_3^*\mathcal{G}))\\
&=p_{1*}(\mathcal{O}_W \otimes p_2^* \mathcal{F} \otimes p_2^* \mathcal{G}) \\
&=\mathrm{P}_X^1(\mathcal{F} \otimes \mathcal{G})
\end{align*}
which is a morphism of bifunctors
\[
\mathfrak{m} \colon \mathrm{P}_X^1(\star) \otimes \mathrm{P}_X^1 (\star \star ) \longrightarrow \mathrm{P}^1_X(\star \otimes \star \star).
\]
\begin{lemma}
The morphism $\mathfrak{m}$ is associative.
\end{lemma}

\begin{proof}
For any positive integer $n$, let us denote by $\Delta_{ij}$ the partial diagonal in $X^n$ corresponding to the equality of the $i^{\mathrm{th}}$ and $j^{\mathrm{th}}$ components, and let $\overline{\Delta}_{ij}$ be its first formal neighbourhood in $X^n$. For $n=3$, there is a natural morphism
\[
\mathcal{O}_{\overline{{\Delta}}_{12}} \otimes \mathcal{O}_{\overline{{\Delta}}_{13}} \longrightarrow \mathcal{O}_{\overline{{\Delta}}_{12}}  \otimes \mathcal{O}_{\overline{{\Delta}}_{13}}  \otimes \mathcal{O}_{{{\Delta}_{23}}}
\]
between subsheaves of $X^3$. This morphism, interpreted as a morphism of correspondences from $X^2$ to $X$, is exactly $\mathfrak{m}$. Then the associativity of $\mathfrak{m}$ follows from the commutativity of the diagram of subsheaves of $X^4$
\[
\xymatrix{ 
\mathcal{O}_{\overline{\Delta}_{12}} \otimes \mathcal{O}_{\overline{\Delta}_{13}}  \otimes \mathcal{O}_{\overline{\Delta}_{14}} \ar[r] \ar[d] & \mathcal{O}_{\overline{\Delta}_{12}} \otimes \mathcal{O}_{\overline{\Delta}_{13}}  \otimes \mathcal{O}_{\overline{\Delta}_{14}} \otimes \mathcal{O}_{\Delta_{23}} \ar[d] \\
\mathcal{O}_{\overline{\Delta}_{12}} \otimes \mathcal{O}_{\overline{\Delta}_{13}}  \otimes \mathcal{O}_{\overline{\Delta}_{14}} \otimes \mathcal{O}_{\Delta_{24}} \ar[r] & \mathcal{O}_{\overline{\Delta}_{12}} \otimes \mathcal{O}_{\overline{\Delta}_{13}}  \otimes \mathcal{O}_{\overline{\Delta}_{14}}  \otimes \mathcal{O}_{\Delta_{234}} 
}
\]
viewed as correspondences between $X^3$ and $X$.
\end{proof}
The sheaf $\mathrm{P}^1_X(\mathcal{O}_X)$ is canonically isomorphic to $\Omega_X^1 \oplus \mathcal{O}_X$. Hence the second inclusion defines a natural morphism
\[
\mu \colon \mathcal{O}_X \longrightarrow \mathrm{P}^1_X(\mathcal{O}_X).
\]
\begin{proposition} \label{chaberton}
The pair $(\mathfrak{m}, \mu)$ endows the principal parts functor $\mathrm{P}^1_X$ with the structure of a lax monoidal functor.
\end{proposition}

\begin{proof}
We must check that the properties of Definition \ref{mul} are satisfied. For any sheaf $\mathcal{F}$ of $\mathcal{O}_X$ modules, let us describe the composition
\[
\mathrm{P}_X^1(\mathcal{F}) \xlongrightarrow{\mathrm{id} \otimes \mu} \mathrm{P}_X^1(\mathcal{F}) \otimes  \mathrm{P}_X^1(\mathcal{\mathcal{O}_X}) \xlongrightarrow{\mathfrak{m}} \mathrm{P}^1_X(\mathcal{F}).
\]
The unit morphism $\mathcal{O}_X \longrightarrow \mathrm{P}_X^1(\mathcal{O}_X)$ is given by the morphism
\[
\mathcal{O}_X \longrightarrow p_{1*}\, \mathcal{O}_{X^2} \longrightarrow p_{1*}\,\mathcal{O}_{\overline{\Delta}_{12}}.
\] 
Let us consider the diagram
\[
\xymatrix{
q_{1*}(\mathcal{O}_{\overline{\Delta}_{12}} \otimes \mathcal{O}_{\overline{\Delta}_{13}} \otimes q_2^* \mathcal{F}) \ar[r] &  q_{1*}(\mathcal{O}_{\overline{\Delta}_{12}} \otimes \mathcal{O}_{\overline{\Delta}_{13}}  \otimes \mathcal{O}_{{\Delta}_{23}} \otimes q_2^* \mathcal{F}) \\
\ar[u] q_{1*}(\mathcal{O}_{\overline{\Delta}_{12}} \otimes	 q_2^* \mathcal{F}) \ar[r]  & \ar@{=}[u] q_{1*}(\mathcal{O}_{\overline{\Delta}_{12}} \otimes \mathcal{O}_{{\Delta}_{23}} \otimes q_2^* \mathcal{F}) \\
\ar@/^4pc/[uu] p_{1*}(\mathcal{O}_{\overline{\Delta}_{12}} \otimes	 p_2^* \mathcal{F}) \ar@{=}[r] \ar[u] & p_{1*}(\mathcal{O}_{\overline{\Delta}_{12}} \otimes	 p_2^* \mathcal{F}) \ar@{=}[u]
} 
\]
The top horizontal arrow is the map $\mathrm{P}_X^1(\mathcal{F}) \otimes \mathrm{P}_X^1(\mathcal{O}_X) \longrightarrow \mathrm{P}^1_X(F)$, and the top round arrow is the map $\mathrm{P}_X^1(\mathcal{F}) \longrightarrow \mathrm{P}^1_X(\mathcal{F}) \otimes \mathrm{P}^1_X(\mathcal{O}_X)$. This proves the first property of Definition \ref{mul}. The second one is proven in the same way.
\end{proof}
\section{Derived equalizers via model categories} \label{quillen}
In this section, we explain briefly how to use model categories to prove that the derived equalizers introduced in \S \ref{modelcat} can interpreted as specific derived Quillen functors.
\par \medskip
Let $\mathcal{M}$ be a model category. For any object $a$ in $\mathcal{M}$, we denote by $\mathcal{M}/a$ the model category of objects lying over $a$. For any morphism $\varphi \colon a \longrightarrow b$ in $\mathcal{M}$, the push forward functor
\[
\varphi_* \colon \mathcal{M}/a \longrightarrow \mathcal{M}/b
\]
is a left Quillen functor, \textit{i.e.}, it admits a right adjoint.  We call it the right adjoint of $\varphi_*$ the pullback functor of $\varphi$, and denote it by 
\[
\varphi^*  \colon \mathcal{M}/b \longrightarrow \mathcal{M}/a.
\]
If $n$ in a positive integer and $b=\Pi_{i \in \{ 1, n \}} \,a=a^n$, we have \[
\mathrm{Hom}_{\mathcal{M}}(a, a^n) \simeq \mathrm{Hom}_{\mathcal{M}}(a, a)^n.
\] 
Hence there is a natural map
$
\mathfrak{i}_n \colon a \longrightarrow a^n
$
corresponding via the above isomorphisms to $(\mathrm{id}_{\mathcal{M}}, \ldots, \mathrm{id}_{\mathcal{M}})$, 
and the pullback functor 
\[
\mathfrak{i}_n^* \colon \mathcal{M}/a^n \longrightarrow \mathcal{M}/a
\] 
is a right Quillen functor. The functor $\mathfrak{i}_n^*$ admits a very simple description: an object in $\mathcal{M}/a^n$ consists of an object $m$ on $\mathcal{M}$ together with $n$ maps in $\mathrm{Hom}_{\mathcal{C}}(m, a)$. Then its image by $\mathfrak{i}_n$ is the equalizer of these $n$ maps.
\par \medskip
We can derive these functors, obtaining a pair of adjoint functors
\[
\xymatrix{
\mathrm{Ho}(\mathcal{M}/a) \ar@<3pt>[r]^-{\mathrm{L} \mathfrak{i}_n *} & \mathrm{Ho}(\mathcal{M}/{a^n}) \ar@<3pt>[l]^-{\mathrm{R} \mathfrak{i}_n^*}
}
\]
Explicitly, the functor $\mathrm{R} \mathfrak{i}_n*$ is obtained as follows: for any object $c \longrightarrow a^n$ in $\mathcal{M}_{a^n}$, we take an object $c'$ such that the composition
\[
c' \longrightarrow c \longrightarrow a^n
\]
is a fibration in $\mathcal{M}$. Then  $\mathrm{R} \mathfrak{i}_n^*(c)=\mathfrak{i}_n^* (c')$.
\par \medskip
We apply this construction to a very specific situation corresponding to the setting of derived equalizers: let $\mathcal{C}$ be a $\mathbf{k}$-linear category and let $\mathcal{M}$ be the category of dg modules on $\mathrm{C}^{\mathrm{b}}(\mathcal{C}) \otimes \mathrm{C}^{\mathrm{b}}(\mathcal{C})^{\mathrm{op}}$. Then $\mathcal{M}$ can be described as follows: its objects are dg-functors from $\mathrm{C}^{\mathrm{b}}(\mathcal{C}) \otimes \mathrm{C}^{\mathrm{b}}(\mathcal{C})^{\mathrm{op}}$ to the category $\mathrm{C}(\mathbf{k})$ of complexes of $\mathbf{k}$-vector spaces, and its morphisms are natural transformations between dg-functors. 
\par \medskip
As any category of dg-modules, $\mathcal{M}$ has a natural model category structure defined by Toën and Vaqui\'{e} (\textit{see} \cite[Def. 3.1]{ToenV1}), where weak equivalences and fibrations admit the following description: if $\Psi \colon U \longrightarrow V$ is a natural transformation between two objects of $\mathcal{M}$ considered as dg-functors, then $\Psi$ is a weak equivalence (resp. a fibration) if and only if for any object $K$ of $\mathrm{C}^{\mathrm{b}}(\mathcal{C}) \otimes \mathrm{C}^{\mathrm{b}}(\mathcal{C})^{\mathrm{op}}$, $\Psi(K)$ is a quasi-isomorphism (resp. $\Psi(K)$ is surjective). There is a fully faithful embedding
\[
\iota \colon \mathrm{EndFct}_{\mathrm{dg}}(\mathrm{C}^{\mathrm{b}}(\mathcal{C})) \longrightarrow \mathcal{M}
\]
given by 
\[
\iota(F) (K \otimes L)=\mathrm{Hom}_{\mathrm{C}(\mathbf{k})}(L, F(K)).
\]
Let $\varphi \colon F \longrightarrow G$ be a natural transformation between to dg-endofunctors of $\mathrm{C}^{\mathrm{b}}(\mathcal{C})$. Then $\iota(\varphi)$ is a weak equivalence (resp. a fibration) in the model category $\mathcal{M}$ if and only if for any object $K$ of $\mathrm{C}^{\mathrm{b}}(\mathcal{C})$, the morphism $\varphi_{K}$ is a quasi-isomorphism\footnote{This means that $\varphi$ is a quasi-isomorphism as defined in Definition \ref{groix}.} (resp. $\varphi_{K}$ is surjective). 
\par \medskip
Assume to be given a couple $(H, \Psi)$ where $H$ is 
in $\mathrm{EndFct}_{\mathrm{dg}}(\mathrm{C}^{\mathrm{b}}(\mathcal{C}))$ and $\Psi \colon H \longrightarrow \mathrm{id}_{\mathrm{C}^{\mathrm{b}}(\mathcal{C})}$ is a natural transformation. For any nonnegative integer $n$, $H^n$ is endowed with $n$ natural maps to $\mathrm{id}_{\mathrm{C}^{\mathrm{b}}(\mathcal{C})}$, so that we can consider $\iota(H^n)$ as an element in the category $\mathcal{M}/\mathrm{id}_{\mathrm{C}^{\mathrm{b}}(\mathcal{C})}^n$.

\begin{proposition}
Assume to be given a triplet $(\mathcal{C}, H, \Psi)$. Then for any nonnegative integer $n$, the following assertions are valid:
\begin{enumerate}
\vspace{0.2cm}
\item[--] $\iota(H^{[n]})=\mathfrak{i}_n^* \{\iota(H^n)\}$
\vspace{0.2cm}
\item[--] $\Delta_{\widetilde{H}}^n $ is a fibrant element in $\mathcal{M}/\mathrm{id}_{\mathrm{C}^{\mathrm{b}}(\mathcal{C})}^n$ and the natural map $\Delta_{\widetilde{H}}^n \longrightarrow H^n $ is a weak equivalence.
\vspace{0.2cm}
\item[--] $\iota(H^{[[n]]})$ is isomorphic to $\mathrm{R}\mathfrak{i}_n^*\{ \iota(H^n)\}$.
\end{enumerate}
\end{proposition}
\begin{proof}
The first assertion is straightforward, and the third assertion is a direct consequence of the second. The second assertion follows from equation \eqref{fib}.
\end{proof}

\bibliographystyle{plain}
\bibliography{biblio}

\begin{thebibliography}{10}

\bibitem{Arinkin-Caldararu}
Dima Arinkin and Andrei C{\u{a}}ld{\u{a}}raru.
\newblock When is the self-intersection of a subvariety a fibration?
\newblock {\em Adv. Math.}, 231(2):815--842, 2012.

\bibitem{ACH}
Dima Arinkin, Andrei C\u{a}ld\u{a}raru, and Marton Hablicsek.
\newblock On the formality of derived intersections.
\newblock {\em Preprint}, 2014.

\bibitem{BNT}
P.~Bressler, R.~Nest, and B.~Tsygan.
\newblock Riemann-{R}och theorems via deformation quantization. {I}, {II}.
\newblock {\em Adv. Math.}, 167(1):1--25, 26--73, 2002.

\bibitem{BuchF}
Ragnar-Olaf Buchweitz and Hubert Flenner.
\newblock The global decomposition theorem for {H}ochschild (co-)homology of
  singular spaces via the {A}tiyah-{C}hern character.
\newblock {\em Adv. Math.}, 217(1):243--281, 2008.

\bibitem{CV}
D.~Calaque and M.~Van~den Bergh.
\newblock {Hochschild cohomology and {A}tiyah classes}.
\newblock {\em Adv. Math.}, 224(5):1839--1889, 2010.

\bibitem{CalaqueCT}
Damien Calaque, Andrei C{\u{a}}ld{\u{a}}raru, and Junwu Tu.
\newblock On the {L}ie algebroid of a derived self-intersection.
\newblock {\em Adv. Math.}, 262:751--783, 2014.

\bibitem{CalaqueEMS}
Damien Calaque and Carlo~A. Rossi.
\newblock {\em Lectures on {D}uflo isomorphisms in {L}ie algebra and complex
  geometry}.
\newblock EMS Series of Lectures in Mathematics. European Mathematical Society
  (EMS), Z\"urich, 2011.

\bibitem{CVDBR}
Damien Calaque, Carlo~A. Rossi, and Michel Van~den Bergh.
\newblock C\u ald\u araru's conjecture and {T}sygan's formality.
\newblock {\em Ann. of Math. (2)}, 176(2):865--923, 2012.

\bibitem{CalMukai}
Andrei C{\u{a}}ld{\u{a}}raru.
\newblock The {M}ukai pairing. {II}. {T}he {H}ochschild-{K}ostant-{R}osenberg
  isomorphism.
\newblock {\em Adv. Math.}, 194(1):34--66, 2005.

\bibitem{DTT}
Vasiliy Dolgushev, Dmitry Tamarkin, and Boris Tsygan.
\newblock The homotopy {G}erstenhaber algebra of {H}ochschild cochains of a
  regular algebra is formal.
\newblock {\em J. Noncommut. Geom.}, 1(1):1--25, 2007.

\bibitem{conjecture}
Julien Grivaux.
\newblock On a conjecture of {K}ashiwara relating {C}hern and {E}uler classes
  of o-modules.
\newblock {\em J. Differential Geom.}, 90(2):267--275, 2012.

\bibitem{Grivaux-formality}
Julien Grivaux.
\newblock Formality of derived intersections.
\newblock {\em Documenta Math.}, 19:1003--1016, 2014.

\bibitem{Grivaux-HKR}
Julien Grivaux.
\newblock The {H}ochschild-{K}ostant-{R}osenberg isomorphism for quantized
  analytic cycles.
\newblock {\em Int. Math. Res. Not. IMRN}, (4):865--913, 2014.

\bibitem{HKRoriginal}
G.~Hochschild, Bertram Kostant, and Alex Rosenberg.
\newblock Differential forms on regular affine algebras.
\newblock {\em Trans. Amer. Math. Soc.}, 102:383--408, 1962.

\bibitem{Huybrechts-Thomas}
Daniel Huybrechts and Richard~P. Thomas.
\newblock Deformation-obstruction theory for complexes via {A}tiyah and
  {K}odaira-{S}pencer classes.
\newblock {\em Math. Ann.}, 346(3):545--569, 2010.

\bibitem{Kapranov}
M.~Kapranov.
\newblock Rozansky-{W}itten invariants via {A}tiyah classes.
\newblock {\em Compositio Math.}, 115(1):71--113, 1999.

\bibitem{lettre}
M.~{Kashiwara}.
\newblock {Unpublished letter to Pierre Schapira}, 1991.

\bibitem{KS1}
M.~Kashiwara and P.~Schapira.
\newblock {\em {Deformation quantization modules}}.
\newblock Ast\'erisque. Soci\'et\'e math\'ematique de France, 2012.

\bibitem{CatSheaves}
Masaki Kashiwara and Pierre Schapira.
\newblock {\em Categories and sheaves}, volume 332 of {\em Grundlehren der
  Mathematischen Wissenschaften [Fundamental Principles of Mathematical
  Sciences]}.
\newblock Springer-Verlag, Berlin, 2006.

\bibitem{miroir}
Maxim Kontsevich.
\newblock Homological algebra of mirror symmetry.
\newblock In {\em Proceedings of the {I}nternational {C}ongress of
  {M}athematicians, {V}ol.\ 1, 2 ({Z}\"urich, 1994)}, pages 120--139.
  Birkh\"auser, Basel, 1995.

\bibitem{Kontsevich}
Maxim Kontsevich.
\newblock Deformation quantization of {P}oisson manifolds.
\newblock {\em Lett. Math. Phys.}, 66(3):157--216, 2003.

\bibitem{Lieblich}
Max Lieblich.
\newblock Moduli of complexes on a proper morphism.
\newblock {\em J. Algebraic Geom.}, 15(1):175--206, 2006.

\bibitem{Lowen}
Wendy Lowen.
\newblock Obstruction theory for objects in abelian and derived categories.
\newblock {\em Comm. Algebra}, 33(9):3195--3223, 2005.

\bibitem{Markarian}
Nikita Markarian.
\newblock The {A}tiyah class, {H}ochschild cohomology and the {R}iemann-{R}och
  theorem.
\newblock {\em J. Lond. Math. Soc. (2)}, 79(1):129--143, 2009.

\bibitem{Quillen}
Daniel Quillen.
\newblock On the (co-) homology of commutative rings.
\newblock In {\em Applications of {C}ategorical {A}lgebra ({P}roc. {S}ympos.
  {P}ure {M}ath., {V}ol. {XVII}, {N}ew {Y}ork, 1968)}, pages 65--87. Amer.
  Math. Soc., Providence, R.I., 1970.

\bibitem{Ramadoss}
Ajay~C. Ramadoss.
\newblock The relative {R}iemann-{R}och theorem from {H}ochschild homology.
\newblock {\em New York J. Math.}, 14:643--717, 2008.

\bibitem{Swan}
Richard~G. Swan.
\newblock Hochschild cohomology of quasiprojective schemes.
\newblock {\em J. Pure Appl. Algebra}, 110(1):57--80, 1996.

\bibitem{ToenV1}
Bertrand To{\"e}n.
\newblock The homotopy theory of {$dg$}-categories and derived {M}orita theory.
\newblock {\em Invent. Math.}, 167(3):615--667, 2007.

\bibitem{Toen}
Bertrand To{\"e}n.
\newblock Derived algebraic geometry.
\newblock {\em EMS Surv. Math. Sci.}, 1(2):153--240, 2014.

\bibitem{Weibel}
Charles~A. Weibel.
\newblock {\em An introduction to homological algebra}, volume~38 of {\em
  Cambridge Studies in Advanced Mathematics}.
\newblock Cambridge University Press, Cambridge, 1994.

\bibitem{Yekutieli}
Amnon Yekutieli.
\newblock The continuous {H}ochschild cochain complex of a scheme.
\newblock {\em Canad. J. Math.}, 54(6):1319--1337, 2002.

\bibitem{Shilin}
Shilin Yu.
\newblock {\em The {D}olbeault dga of a formal neighborhood}.
\newblock ProQuest LLC, Ann Arbor, MI, 2013.
\newblock Thesis (Ph.D.)--The Pennsylvania State University.

\end{thebibliography}

\end{document}